\documentclass[reqno,11pt]{amsart}
\usepackage{amsmath,amsfonts,amssymb,amsxtra,latexsym,amscd,enumerate,enumitem,amsthm,verbatim,here}
\usepackage{mathtools}
\usepackage{array}
\usepackage{booktabs}
\usepackage[table,dvipsnames]{xcolor}
\usepackage{cancel}

\usepackage[margin=1.2in]{geometry}
\numberwithin{equation}{section}
\usepackage{hyperref}
\definecolor{myrefred}{RGB}{180,0,0}
\definecolor{myciteblue}{RGB}{0,70,140} % light blue
\hypersetup{
    colorlinks=true,
    pdfstartview=FitV,
    linkcolor=myrefred,      % \ref, \eqref, section links
    citecolor=myciteblue,    % \cite
    urlcolor=black
}
\definecolor{labelkey}{rgb}{0.6,0,0}

\newcommand{\R}{\mathbb{R}}

\newcommand{\C}{\mathbb{C}}

\newcommand{\la}{\langle}
\newcommand{\ra}{\rangle}
\newcommand{\non}{{\mathcal{N}}}
\newcommand{\smt}{{\mathcal{T}}}
\newcommand{\ent}{{\mathfrak{Q}}}
\newcommand{\nata}{{\mathfrak{r}}}
\newcommand{\dob}{{\mathbf{D}}}
\newcommand{\kfun}{{\mathcal{K}}}
\newcommand{\mass}{{\mathbf{M}}}
\newcommand{\ifun}{{\mathcal{I}}}
\newcommand{\Tone}{{\smt_{L_1}}}
\newcommand{\Ttwo}{{\smt_{L_2}}}
\newcommand{\g}{{\mathfrak{g}}}
\newcommand{\h}{{\mathfrak{h}}}

\newcommand{\dom}{{\R_+}}

\numberwithin{equation}{section} %pour numeroter les equations par section

\newtheorem{theorem}{Theorem}

\newtheorem{lemma}{Lemma}[section]
\newtheorem{corollary}[lemma]{Corollary}

\newtheorem{proposition}[lemma]{Proposition}

\theoremstyle{remark}
\newtheorem{remark}[lemma]{Remark}

\definecolor{red}{RGB}{180,0,0}
\definecolor{blue}{RGB}{0,55,170}
\definecolor{ForestGreen}{RGB}{0,105,55}
\makeatletter
\newcommand{\sectionnotoc}[1]{%
  \begingroup
  \let\@tocwrite\@gobbletwo
  \section*{#1}%
  \endgroup
}
\makeatother
\begin{document}
\title[3D Hartree equation with singular external potentials]{Modified scattering for the 3D Hartree equation with singular external potentials}
\author{Mengyi Xie}
\address{Yale University}
\email{mengyi.xie@yale.edu}

\begin{abstract}
We establish modified scattering for the three-dimensional Hartree equation with Coulomb interaction and singular external potentials, including repulsive Coulomb and inverse-square potentials. In the defocusing case, the result holds for arbitrary initial data in $H^{0,1}(\R^3)\cap H^{1,0}(\R^3)$. In particular, we obtain large-data modified scattering for the flat defocusing Hartree equation. For the focusing problem with repulsive Coulomb potential $V(x)=\kappa|x|^{-1}$, we prove radial modified scattering under the sharp mass condition $|\lambda|\|u_0\|_2^2\leq2\kappa$, including the endpoint, and construct standing waves at every larger mass. Our argument obtains strong $L^2$ convergence from rescaled energy and comparison with a frozen-field evolution.
\end{abstract}

\maketitle
%%%%%%%%%%%%%%%%%%%%%%%%%%%%%%%%%%%%%%%
%%%%%%%%%%%%%%%%%%%%%%%%%%%%%%%%%%%%%%%
%%%%%%%%%%%%%%%%%%%%%%%%%%%%%%%%%%%%%%%

\setcounter{tocdepth}{1}
\tableofcontents

%%%%%%%%%%%%%%%%%%%%%%%%%%%%%%%%%%%%%%%
%%%%%%%%%%%%%%%%%%%%%%%%%%%%%%%%%%%%%%%
%%%%%%%%%%%%%%%%%%%%%%%%%%%%%%%%%%%%%%%

\section{Introduction}\label{sec:intro}
We study the three-dimensional Hartree equation with an external potential $V$,
\begin{equation}\label{main}
i\partial_tu-\Delta u+Vu+\lambda\phi_u u=0, \qquad \phi_u=|x|^{-1}*|u|^2, \qquad u(0)=u_0.
\end{equation}
Here $u:\R_t\times\R_x^3\to\C$, $V$ is real-valued, and $\lambda\in\R\setminus\{0\}$. The interaction is defocusing when $\lambda>0$ and focusing when $\lambda<0$. Since $-\Delta\phi_u=4\pi|u|^2$, \eqref{main} is also called the Schr\"odinger--Poisson equation. Our principal examples are the repulsive Coulomb potential $V(x)=\kappa|x|^{-1}$ and the inverse-square potential $V(x)=\kappa|x|^{-2}$, with $\kappa>0$.

The variable $y=x/t$ is natural for describing the dispersion of \eqref{main}: a free Schr\"odinger wave packet travels at constant group velocity, so its position divided by time approaches that velocity. Observing the solution along $x=ty$ therefore describes its expanding spatial distribution on a fixed scale. The rescaled density $m(t,y):=t^3|u(t,ty)|^2$ preserves the total mass, and a change of variables in the convolution gives, for $t>0$,
\begin{equation}\label{eq:ray_field}
\phi_u(t,ty)=\frac{1}{t} \int_{\R^3}\frac{m(t,z)}{|y-z|}\,dz.
\end{equation}
This identity explains the long-range effect: a nonzero limiting density on the ray scale generates a potential of order $t^{-1}$, whose time integral produces a logarithmic phase. The physical-space approach developed in \cite{PX} and extended in \cite{HX,HX26} uses this relation between the density, the field, and the phase.

The Hartree field $\phi_u$ is generated by the solution, whereas $V$ is prescribed. For $V(x)=\kappa|x|^{-1}$, the identity $V(ty)=t^{-1}\kappa|y|^{-1}$ shows that the external field contributes a second logarithmic term to the phase. For the faster-decaying potentials considered below, including $V(x)=\kappa|x|^{-2}$, the explicit logarithmic phase comes only from the Hartree interaction. Their singularities at the origin nevertheless require additional control in the scattering argument.

\medskip

We work in the space
\begin{equation}\label{eq:Sigma}
    \Sigma:=\big\{f\in H^1(\R^3):xf\in L^2(\R^3)\big\}, \qquad \|f\|_\Sigma:=\|f\|_{H^1}+\|xf\|_2,
\end{equation}
and write
\begin{equation}\label{eq:Coulomb_field_notation}
    \phi_f(y):=\int_{\R^3}\frac{|f(z)|^2}{|y-z|}\,dz,\qquad
    \dob[f]:=\iint_{\R^3\times\R^3} \frac{|f(y)|^2|f(z)|^2}{|y-z|}\,dy\,dz.
\end{equation}
Thus $\phi_f$ is the Coulomb field generated by $|f|^2$, and $\dob[f]=\int\phi_f|f|^2$ is the Coulomb interaction functional. Equation \eqref{main} conserves the mass
\begin{equation}\label{mass}
    \mass:=\|u(t)\|_2^2=\|u_0\|_2^2,
\end{equation}
and the energy
\begin{equation}\label{energy}
    \mathbf E:=\frac12\|\nabla u(t)\|_2^2+\frac12\int_{\R^3} V(x)|u(t,x)|^2\,dx +\frac{\lambda}{4}\dob[u(t)].
\end{equation}
For the nonnegative potentials considered here, $-\Delta+V$ denotes the self-adjoint operator associated with the closed quadratic form
\[
    q_V[f]=\|\nabla f\|_2^2+\int_{\R^3}V|f|^2, \qquad f\in H^1(\R^3).
\]
We use the corresponding global solutions in $C(\R;\Sigma)$. The solution framework and the conservation laws needed below will be recalled in Section \ref{sec:est}; our main task is to determine the long-time asymptotics.

%%%%%%%%%%%%%%%%%%%%%%%%%%%%%%%%%%%%%%%
%%%%%%%%%%%%%%%%%%%%%%%%%%%%%%%%%%%%%%%
%%%%%%%%%%%%%%%%%%%%%%%%%%%%%%%%%%%%%%%

\subsection{Main results for the Coulomb potential}
We first consider
\begin{equation}\label{eq:Coulomb_potential}
    V(x)=\frac{\kappa}{|x|},\qquad \kappa>0.
\end{equation}
In the defocusing case, we prove modified scattering for arbitrary data in $\Sigma$.
\begin{theorem}\label{thm:Coulomb_defocusing}
    Fix $\kappa>0$ and $\lambda>0$. For every $u_0\in\Sigma$, \eqref{main} has a unique global solution $u\in C(\R;\Sigma)$. There exist $U_\infty\in L^2(\R^3)$ and $\Phi_\infty\in\dot H^1(\R^3)$ such that
    \begin{equation}\label{eq:modified_scattering_coulomb}
        u(t,x)=t^{-3/2}e^{-i|x|^2/(4t)}
        e^{i(\kappa t/|x|+\lambda\Phi_\infty(x/t))\log t}
        U_\infty(x/t)+o_{L_x^2}(1)
    \end{equation}
    as $t\to\infty$. Moreover, $\|U_\infty\|_2^2=\mass$, and the limiting field is determined by the asymptotic profile:
    \begin{equation}\label{eq:phase_connection}
        \Phi_\infty=\phi_{U_\infty} = |\cdot |^{-1}*|U_\infty|^2.
    \end{equation}
\end{theorem}

For the focusing interaction, radial symmetry yields a sharp mass threshold, including its endpoint.
\begin{theorem}\label{thm:coulomb_focusing}
    Fix $\kappa>0$ and $\lambda<0$. For radial initial data $u_0\in\Sigma$ satisfying
    \[
        |\lambda|\mass\leq 2\kappa, \qquad \mass=\|u_0\|_2^2,
    \]
    the unique global solution $u\in C(\R;\Sigma)$ to \eqref{main} satisfies
    \begin{equation} \label{eq:modified_scattering_coulomb_radial}
        u(t,x)=t^{-3/2}e^{-i|x|^2/(4t)}
        e^{i(\kappa t/|x|+\lambda\Phi_\infty(x/t))\log t}
        U_\infty(x/t)+O_{L_x^2}(t^{-1/12})
    \end{equation}
    for a radial profile $U_\infty\in L^2(\R^3)$, with $\|U_\infty\|_2^2=\mass$ and $\Phi_\infty=\phi_{U_\infty}\in\dot H^1(\R^3)$.

    Conversely, for each prescribed mass $\mass>2\kappa/|\lambda|$, there exist a nonzero radial $Q\in\Sigma$ and $\omega>0$ such that $\|Q\|_2^2=\mass$ and
    \begin{equation}\label{eq:ODE}
        \left(-\Delta+\frac{\kappa}{|x|}\right)Q+\lambda\phi_Q Q+\omega Q=0.
    \end{equation}
    The standing wave $u(t)=e^{-i\omega t}Q$ is a global solution of \eqref{main} that does not satisfy \eqref{eq:modified_scattering_coulomb}.
\end{theorem}

\begin{remark}
     The exponent $1/12$ in \eqref{eq:modified_scattering_coulomb_radial} is not optimal.
\end{remark}

Theorem \ref{thm:coulomb_focusing} identifies $2\kappa/|\lambda|$ as the sharp mass threshold for modified scattering of \emph{every} radial datum in $\Sigma$ at a prescribed mass. Above this threshold, it asserts the existence of nonscattering data. The endpoint belongs to the scattering regime, but requires an additional argument because the direct coercivity supplied by the external potential degenerates at equality.

\medskip

The Coulomb problem is delicate already at the linear level. Black--Toprak--Vergara--Zou \cite{BTVZ} proved the $L^1$-to-$L^\infty$ decay rate $|t|^{-3/2}$ for radial data through a detailed analysis of the distorted Fourier representation. For nonradial data, work in preparation by Black--Toprak--Zou \cite{BTZ} establishes a dispersive bound with decay slower than the free rate $|t|^{-3/2}$. On the other hand, for the attractive Coulomb operator, \cite{GJZ} established sharp decay of order $|t|^{-1}$ on the absolutely continuous subspace. Global Strichartz estimates for the repulsive Coulomb evolution follow from Mizutani's theory of slowly decaying potentials \cite{Mizutani}.

For the nonlinear problem, one must also identify the limiting Hartree field and control the phase it generates. In the rescaled variable $y=x/t$, the external correction $\kappa|y|^{-1}\log t$ is singular at zero velocity, and differentiating it introduces weights not controlled by the energy. Our argument permits mass at arbitrarily small velocities and, in the defocusing case, arbitrary data in $\Sigma$. It obtains strong convergence from rescaled energy and density estimates, without requiring a uniform pointwise bound for the rescaled profile.

%%%%%%%%%%%%%%%%%%%%%%%%%%%%%%%%%%%%%%%
%%%%%%%%%%%%%%%%%%%%%%%%%%%%%%%%%%%%%%%
%%%%%%%%%%%%%%%%%%%%%%%%%%%%%%%%%%%%%%%

\subsection{Main results for general external potentials}
We next consider nonnegative potentials that decay faster than Coulomb at infinity and may have an inverse-square singularity at the origin. More specifically, we consider potentials satisfying
\begin{equation}\tag{A1}\label{eq:A1}
    V\in C^1(\R^3\setminus\{0\}),
\end{equation}
\begin{equation}\tag{A2}\label{eq:A2}
    V\geq0,
\end{equation}
and
\begin{equation}\tag{A3}\label{eq:A3}
    V(x)+|x||\nabla V(x)|
    \leq C_V\big(|x|^{-2}+|x|^{-p}\big),
    \qquad p\in(1,2].
\end{equation}

\begin{theorem}\label{thm:square_defocusing}
    Suppose that $V$ satisfies \eqref{eq:A1}--\eqref{eq:A3} and
    \begin{equation}\tag{D1}\label{eq:D1}
        x\cdot\nabla V(x)\leq0,
    \end{equation}
    \begin{equation}\tag{D2}\label{eq:D2}
        \big(V(x)+x\cdot\nabla V(x)\big)_+
        \leq C_V\la x\ra^{-7},
    \end{equation}
    where $\la x\ra=(1+|x|^2)^{1/2}$ and
    $s_+:=\max\{s,0\}$ for $s\in\R$. Fix $\lambda>0$. For every $u_0\in\Sigma$, \eqref{main} has a unique global solution $u\in C(\R;\Sigma)$. There exist $U_\infty\in L^2(\R^3)$ and $\Phi_\infty\in\dot H^1(\R^3)$ such that
    \begin{equation} \label{eq:modified_scattering_general}
        u(t,x)=t^{-3/2}e^{-i|x|^2/(4t)}
        e^{i\lambda\Phi_\infty(x/t)\log t}
        U_\infty(x/t)+o_{L_x^2}(1)
    \end{equation}
    as $t\to\infty$, with $\|U_\infty\|_2^2=\mass$ and $\Phi_\infty=\phi_{U_\infty}$.
\end{theorem}
Condition \eqref{eq:D1} requires $V$ to be non-increasing along each ray from the origin, so that the associated force $-\nabla V$ has a nonnegative outward radial component. The stronger inequality
\begin{equation}\tag{D2$^*$}\label{eq:D2_strong}
    V+x\cdot\nabla V\leq0
\end{equation}
would make the rescaled energy \eqref{eq:intro_K} non-increasing. Assumption \eqref{eq:D2} allows a positive error in this inequality, provided that its contribution is decaying fast enough. In particular, it holds whenever $(V+x\cdot\nabla V)_+$ is bounded and compactly supported. 

\begin{theorem}\label{thm:square_focusing}
    Suppose that $V$ satisfies \eqref{eq:A1}--\eqref{eq:A3} and, for some $0\leq b<1/2$,
    \begin{equation}\tag{F1}\label{eq:F1}
    \begin{split}
        &\int_{\R^3}\big(2V(x)+x\cdot\nabla V(x)\big)|f(x)|^2\,dx\\
        &\qquad \qquad \leq b\left(\|\nabla f\|_2^2
        +\int_{\R^3}V(x)|f(x)|^2\,dx\right),
        \qquad \forall f\in H^1(\R^3).
    \end{split}
    \end{equation}
    Fix $\lambda<0$. There exists $\varepsilon_0>0$, depending on $V$, $\lambda$, and $b$, such that, whenever
    \begin{equation}\label{eq:initial_cond}
        \|u_0\|_\Sigma\leq\varepsilon_0,
    \end{equation}
    the unique global solution $u\in C(\R;\Sigma)$ to \eqref{main} satisfies
    \begin{equation} \label{eq:modified_scattering_focusing_general}
        u(t,x)=t^{-3/2}e^{-i|x|^2/(4t)}
        e^{i\lambda\Phi_\infty(x/t)\log t}
        U_\infty(x/t)+O_{L_x^2}(\varepsilon_0 t^{-\frac{1-2b}{16}}),
    \end{equation}
    for some $U_\infty\in L^2(\R^3)$ with $\|U_\infty\|_2^2=\mass$ and $\Phi_\infty=\phi_{U_\infty}\in\dot H^1(\R^3)$. 
\end{theorem}

\begin{remark}
    The smallness condition in Theorem \ref{thm:square_focusing} is needed for modified scattering. Under \eqref{eq:A1}--\eqref{eq:A3}, global well-posedness in $\Sigma$ holds for arbitrary data for either sign of $\lambda$.
\end{remark}

\begin{remark}[Examples of admissible potentials]
Let $a\in C^1(\mathbb S^2)$ be nonnegative. The following examples illustrate the hypotheses of Theorems \ref{thm:square_defocusing} and \ref{thm:square_focusing}.
\begin{enumerate}[label=\textup{(\roman*)},leftmargin=*]
    \item The zero potential $V\equiv0$ satisfies assumptions \eqref{eq:A1}--\eqref{eq:A3} and \eqref{eq:D1}--\eqref{eq:D2}, as well as \eqref{eq:F1} with $b=0$. In particular, Theorem \ref{thm:square_defocusing} establishes \emph{large-data} modified scattering for the \emph{flat defocusing Hartree} equation for every $u_0\in\Sigma$, without any symmetry assumption.
    
    \item The potentials
    \[
        V(x)=a\left(\frac{x}{|x|}\right)|x|^{-p},\qquad 1<p\leq2,
    \]
    satisfy the defocusing hypotheses \eqref{eq:D1}--\eqref{eq:D2}. Since $2V+x\cdot\nabla V=(2-p)V$, they also satisfy \eqref{eq:F1} with $b=2-p$ whenever $3/2<p\leq2$. In particular, $V(x)=\kappa|x|^{-2}$ is covered by both theorems for every fixed $\kappa\geq0$, with $b=0$.

    \item The more regular potentials
    \[
        V(x)=\kappa\la x\ra^{-p},\qquad \kappa\geq0,\quad1<p\leq2,
    \]
    satisfy the defocusing hypotheses \eqref{eq:D1}--\eqref{eq:D2}: the positive part of $V+x\cdot\nabla V$ is bounded and supported in $\{|x|^2\leq(p-1)^{-1}\}$. For $3/2<p\leq2$, the identity
    \[
        2V+x\cdot\nabla V=(2-p)V+p\kappa\la x\ra^{-p-2}
    \]
    and Hardy's inequality verify \eqref{eq:F1} for sufficiently small $\kappa$.

    \item More generally, potentials of the form
    \[
        V(x)=a\left(\frac{x}{|x|}\right)|x|^{-p}(1+|x|)^{-q}, \qquad 1<p\leq2,\quad q\geq0,
    \]
    satisfy the defocusing hypotheses \eqref{eq:D1}-\eqref{eq:D2}. They satisfy \eqref{eq:F1} with $b=2-p$ when $3/2<p\leq2$, since
    \[
        2V+x\cdot\nabla V
        =\left(2-p-\frac{q|x|}{1+|x|}\right)V\leq(2-p)V.
    \]
\end{enumerate}
\end{remark}

\medskip

The inverse-square potential presents a different difficulty from Coulomb. For $V(x)=\kappa|x|^{-2}$, one has $V(ty)=\kappa t^{-2}|y|^{-2}$ along each fixed nonzero ray. Its time dependence is integrable, so it contributes no additional logarithmic phase. However, it has the same scaling as the Laplacian: after removing the free quadratic phase and passing to $y=x/t$, the two operators appear together as
\[
t^{-2}\left(-\Delta_y+\kappa|y|^{-2}\right).
\]
Thus both carry the same factor $t^{-2}$ in the rescaled equation.

Moreover, the singularity at the origin is also stronger than in the Coulomb case. Hardy's inequality gives $|x|^{-1}u\in L^2$ for $u\in H^1$, but does not in general give $|x|^{-2}u\in L^2$. For the inverse-square potential, it therefore controls the potential energy $\int V|u|^2$, while the potential term $Vu$ need not belong to $L^2$. The factor $t^{-2}$ alone does not justify estimating this term as an integrable $L^2$ error.

\medskip

For the associated linear evolution, Burq-Planchon-Stalker-Tahvildar-Zadeh \cite{BPST03,BPST04} established Strichartz estimates for inverse-square and related critical potentials, and Fanelli-Felli-Fontelos-Primo \cite{FFFP} proved the $|t|^{-3/2}$ dispersive estimate for the three-dimensional repulsive inverse-square operator. For nonnegative homogeneous inverse-square potentials, Vega–Visciglia \cite{VV07} established lower bounds complementing local smoothing estimates. Our argument obtains modified scattering for data in $\Sigma$ by controlling the singular potential through energy estimates and identifying the logarithmic phase generated by the limiting Hartree field.

\subsection{Previous work}\label{subsec:previous}
We first recall the scattering theory for the Hartree equation without an external potential, then discuss external potentials and the related Vlasov--Poisson dynamics with a repulsive point charge.

\medskip

\paragraph{\bf The flat Hartree equation.}
We first recall the scattering theory for
\[
    i\partial_tu-\Delta u +\lambda\big(|x|^{-\gamma}*|u|^2\big)u=0, \qquad0<\gamma<3.
\]
For a dispersing solution, the density spreads over the scale $x\sim t$, and the homogeneity of $|x|^{-\gamma}$ gives a nonlinear potential of order $t^{-\gamma}$ along rays, as seen in \eqref{eq:ray_field}. The exponent $\gamma=1$ therefore separates the short-range and long-range regimes for the Hartree interaction: the time integral converges for $\gamma>1$, grows logarithmically for $\gamma=1$, and grows like $t^{1-\gamma}$ for $0<\gamma<1$. For $1<\gamma<3$, Hayashi--Tsutsumi \cite{HT87} established ordinary $L^2$ scattering for defocusing data in the weighted energy space, as well as small-data scattering for either sign.

At the Coulomb exponent $\gamma=1$, Hayashi--Naumkin \cite{HN98} and Hayashi--Naumkin--Ozawa \cite{HNO98} established small-data modified scattering. Subsequent proofs use space-time resonance analysis \cite{KP11}, testing by wave packets \cite{VanHoose2024}, and a physical-space approach \cite{PX}. 

For $1/2<\gamma<1$, \cite{HN01} obtained small-data forward modified scattering with a phase correction of order $t^{1-\gamma}$. Building on the study of Vlasov--Hartree systems \cite{HX}, Huang and the author \cite{HX26} treated the range $0<\gamma\leq1/2$, where further phase corrections arise from the feedback between the action and the density. The complementary construction of modified wave operators was developed by Ginibre--Ozawa \cite{GO93} and Ginibre--Velo \cite{GV00}. Nakanishi \cite{Nakanishi02,Nakanishi02II} proved small-data asymptotic completeness for modified scattering when $1/2<\gamma\leq1$.

\medskip

\paragraph{\bf The Hartree equation with short-range external potentials.}

For integrable interaction kernels, \cite{Dietze} proved sharp decay of small, sufficiently regular solutions under dispersive assumptions on the associated linear evolution generated by $-\Delta+V$. 

For the Coulomb Hartree interaction, Oh--Pasqualotto--Tang \cite{OPT} established forward small-data modified scattering with regular external potentials. Their good commutator approach combines physical-space energy estimates and testing by wave packets. Under decay, threshold, and integrated local energy assumptions on the linear operator, they obtain sharp $t^{-3/2}$ decay and quantitative pointwise asymptotics. Their stated assumptions require an external potential that is regular at the origin and decays faster than $|x|^{-2}$ at infinity, and hence exclude both external potentials $\kappa|x|^{-1}$ and $\kappa|x|^{-2}$.

\medskip

\paragraph{\bf The Hartree equation with singular external potentials.}
For the defocusing Coulomb interaction $\gamma=1$ and a repulsive external Coulomb potential, Hayashi--Ozawa \cite{HO} proved global existence and time decay for arbitrary data that is sufficiently regular and localized. Huang \cite{Huang} subsequently constructed modified wave operators for small asymptotic data whose velocity support avoids the origin, using the global Strichartz estimates of Mizutani \cite{Mizutani}.

For inverse-square external potentials, Suzuki \cite{Suzuki13} developed an energy-method framework for the Hartree equations associated with $-\Delta+\kappa|x|^{-2}$, throughout the range $\kappa>-1/4$ in three dimensions. His subsequent works \cite{Suzuki17,Suzuki19} study ordinary scattering for defocusing Hartree interactions $|x|^{-\gamma}$ in the short-range regime $1<\gamma<3$, for data with finite energy and a spatial moment. In particular, the forward $L^2$ scattering statement for $\kappa>-1/4$ is recorded in \cite[Proposition 3.1]{Suzuki17}; see also the discussion preceding Theorem 3.4 in \cite{Suzuki19}. For the focusing interaction, \cite{CLM} proved ordinary scattering for radial $H^1$ solutions below the ground-state thresholds when $2<\gamma<3$. These inverse-square scattering results concern short-range Hartree interactions and do not include Coulomb interaction.

The variational theory of Hartree standing waves goes back to Lieb's construction and uniqueness of a Choquard minimizer \cite{Lieb} and the constrained minimization and stability framework of Cazenave--Lions \cite{CL}. These works provide the background for our construction of radial energy minimizers at prescribed mass above the Coulomb threshold.

\medskip

\paragraph{\bf Vlasov--Poisson dynamics near a point charge.} As discussed in \cite{HX,HX26}, Hartree and Vlasov equations share a common structure in their long-range dynamics. From this perspective, Vlasov–Poisson with a fixed repulsive point charge is the kinetic analogue of the Hartree equation with a repulsive Coulomb potential. Pausader–Widmayer \cite{PW} proved modified scattering in a radial setting, while Pausader–Widmayer–Yang \cite{PWY} later treated the nonradial problem, including the motion of the point charge. Their approach incorporates the external Coulomb field into the linearized transport equation and integrates the repulsive Kepler flow in action-angle coordinates. These coordinates straighten the linearized characteristics, after which field, moment, and derivative estimates control the self-consistent interaction. As in \eqref{eq:modified_scattering_coulomb}, the asymptotics contain logarithmic corrections from both the external charge and the gas.

\medskip

\paragraph{\bf Large-data scattering and modified scattering} 
Ordinary scattering for large data has been developed through weighted estimates and Morawetz identities. For defocusing NLS with \emph{short-range} nonlinearities, Tsutsumi–Yajima \cite{TY84} established $L^2$ scattering for arbitrary data in $\Sigma$ using pseudoconformal estimates. Burq–Georgiev–Tzvetkov–Visciglia \cite{BGTV23} subsequently upgraded this conclusion to $H^1$ scattering throughout the short-range mass-subcritical regime, using moment estimates.

In the unweighted energy space, Ginibre–Velo \cite{GVEnergy00} established asymptotic completeness for repulsive \emph{short-range} Hartree interactions by the Morawetz–Strauss method, including $|x|^{-\gamma}$ with $2<\gamma<3$ in three dimensions. Planchon–Vega \cite{PV09} developed bilinear virial identities extending interaction Morawetz estimates, providing a method for deriving spacetime bounds both in Euclidean space and on domains with boundaries. Ginibre–Velo \cite{GV10} established quadratic Morawetz inequalities yielding simpler proofs of energy-space scattering for defocusing NLS and Hartree equations. Visciglia \cite{Visciglia09} used interaction Morawetz estimates to prove $L^p$ decay, for arbitrary energy data in defocusing NLS, including \emph{long-range} nonlinearities. Planchon–Vega \cite{PV09,PV12} combined their bilinear virial approach with estimates for boundary terms to establish large-data energy-space scattering for suitable defocusing powers on exterior domains. 

At the \emph{long-range} threshold, Deift–Zhou \cite{DZ03} obtained explicit logarithmically corrected asymptotics for the one-dimensional defocusing cubic NLS with arbitrary weighted energy data, using inverse scattering. In arbitrary dimensions, Cazenave–Naumkin \cite{CN18} proved modified scattering for $|u|^{2/d}u$ for a special class of arbitrarily large, sufficiently regular, nonvanishing and oscillatory data, through a pseudoconformal transformation and weighted derivative estimates. For Hartree with Coulomb interaction, Hirata \cite{Hirata95} proved \emph{weak convergence} of a phase-corrected profile for arbitrary $\Sigma$ data using vector-field estimates. His phase depends on the evolving solution, and the result does not identify a logarithmic correction determined by the limiting profile.

For the complementary final-state problem, Ginibre–Velo \cite{GV00} constructed modified wave operators for long-range Hartree equations, including the Coulomb interaction, without a smallness assumption on sufficiently regular prescribed asymptotic data.

%%%%%%%%%%%%%%%%%%%%%%%%%%%%%%%%%%%%%%%%%%%%%
%%%%%%%%%%%%%%%%%%%%%%%%%%%%%%%%%%%%%%%%%%%%%
%%%%%%%%%%%%%%%%%%%%%%%%%%%%%%%%%%%%%%%%%%%%%

\subsection{Main ideas of the proof}\label{subsec:ideas}
We build on the physical-space approach of Pausader and the author \cite{PX}. For $t\geq 1$, we write $y=x/t$ and introduce the profile\footnote{This profile corresponds to the pseudoconformal change of variables, with the original time variable retained.}
\begin{equation}\label{eq:intro_g}
    g(t,y):=t^{\frac{3}{2}}e^{i\frac{t|y|^2}{4}}u(t,ty).
\end{equation}
This change of variables preserves the $L^2$ norm. The equation for the rescaled profile is
\begin{equation}\label{eq:intro_PDE_g_general}
    i\partial_tg-t^{-2}\Delta g+V(ty)g =-\frac{\lambda}{t}\phi_g g,
\end{equation}
while for the Coulomb potential, this becomes
\begin{equation}\label{eq:intro_PDE_g_coulomb}
    i\partial_tg-t^{-2}\Delta g =-\frac1t\left(\frac{\kappa}{|y|}+\lambda\phi_g\right)g.
\end{equation}
Different from \cite{PX}, we first extract the limiting field $\Phi_\infty$ from the density, which is less sensitive to oscillations of $g$. After finding $\Phi_\infty$, we prove strong $L^2$ convergence of the phase-corrected profile $h(t)=e^{-i\non_\infty\log t}g(t)$, where $\non_\infty=\kappa/|y|+\lambda\Phi_\infty$ for Coulomb and $\non_\infty=\lambda\Phi_\infty$ for general potentials. The difficulty is to establish strong convergence of $h$ using the available energy control. $\|\nabla g(t)\|_2$ grows in time, while the limiting field has only $\dot H^1$ regularity. We first explain the common defocusing argument.

\medskip

\paragraph{\bf Rescaled energy and the limiting field.}
Real multiplication terms cancel from the density equation, giving $\partial_t|g|^2=2t^{-2}\nabla\cdot\Im(\bar g\nabla g)$. The rescaled energy and its derivative are
\begin{align}
    \kfun(t)&:=t^{-1}\|\nabla g(t)\|_2^2+t\int V(ty)|g(t,y)|^2\,dy+\frac\lambda2\dob[g(t)],\label{eq:intro_K}\\
    \kfun'(t)&=-t^{-2}\|\nabla g(t)\|_2^2+\int(V+x\cdot\nabla V)|u(t,x)|^2\,dx.\label{eq:intro_Kprime}
\end{align}
The last term in \eqref{eq:intro_Kprime} vanishes for the Coulomb potential, while for general potentials, its positive part is integrable in time by the Morawetz estimate and \eqref{eq:D2}. Defocusing positivity then yields
\begin{equation}\label{eq:K_bounds}
    \sup_{t\geq1}\kfun(t)<\infty,   \qquad \|\nabla g(t)\|_2\lesssim t^{1/2},     \qquad \int_1^\infty t^{-2}\|\nabla g(t)\|_2^2\,dt<\infty.
\end{equation}
The density equation then makes $\|\partial_t\phi_g(t)\|_2\lesssim t^{-5/4}$, while the interaction energy $\dob[g(t)]$ bounds $\|\nabla\phi_g(t)\|_2$. Consequently, there exists a nonnegative limiting field $\Phi_\infty\in\dot H^1(\R^3)$ with
\begin{equation}\label{eq:intro_field_limit}
    \|\phi_g(t)-\Phi_\infty\|_2\lesssim t^{-1/4}.
\end{equation}
Only the field difference needs to belong to $L^2$. This determines the candidate phase correction. Identifying $\Phi_\infty$ with $\phi_{U_\infty}$ will follow from strong convergence of the modified profile.

\medskip

\paragraph{\bf Comparison with a frozen evolution} 
The available bounds do not make either $t^{-2}\Delta g$ or $t^{-1}(\phi_g-\Phi_\infty)g$ integrable in $L^2$. We instead consider an auxiliary evolution where we freeze its Hartree field. 

For both defocusing problems, choose $0\leq\chi_\delta\leq1$, vanishing on $|y|\leq\delta$ and equal to one on $|y|\geq2\delta$, and set 
\begin{equation}
    g_\delta=\chi_\delta g,\qquad
    \non=\begin{cases}
        \displaystyle\frac\kappa{\max\{|y|,\delta\}}+\lambda\Phi_\infty,&\text{Coulomb},\\[4pt]
        \lambda\Phi_\infty,&\text{general potentials}.
    \end{cases}
\end{equation}
Thus $\non=\non_\infty$ on the support of $g_\delta$ and $\nabla\non\in L^2$. For general potentials, localization places $V(ty)g_\delta$ in $L^1_tL^2_y$ while the Coulomb coefficient has to stay in the frozen equation because its $t^{-1}$ factor is not integrable. The cutoff commutator $t^{-2}[\chi_\delta,\Delta]g$ is also integrable in $L^2$. The Coulomb interaction energy controls the mass near the origin, giving $\|g-g_\delta\|_2^4 \lesssim\delta\dob[g]$. Its uniform bound from \eqref{eq:K_bounds} therefore yields $\|g-g_\delta\|_2 \lesssim\delta^{1/4}$ for both defocusing problems.

Starting at a large time $S$, define
\begin{equation}\label{eq:intro_frozen}
    i\partial_tg_\delta^S-t^{-2}\Delta g_\delta^S=-t^{-1}\non g_\delta^S,
    \qquad g_\delta^S(S)=g_\delta(S).
\end{equation}
We prove that $\sup_{t\geq S}\|g_\delta(t)-g_\delta^S(t)\|_2\to0$ as $S\to\infty$, and that $h_\delta^S(t):=e^{-i\non\log t}g_\delta^S(t)$ converges for fixed $S$. The triangle inequality then gives convergence of $h_\delta:=\chi_\delta h$.

\medskip

\paragraph{\bf Entropy and two truncation levels} 
The field-error correlation contains $t^{-1}(\phi_g-\Phi_\infty) g_\delta \overline{g_\delta^S}$, whose integrability is not supplied by \eqref{eq:intro_field_limit} and conservation of mass. Smooth, phase-equivariant amplitude truncations $\smt_L$ make the product uniformly $L^2$. Their discarded mass is controlled by the entropy
\begin{equation}\label{eq:intro_entropy}
    \ent[f]:=\int|f(x)|^2\log(1+|f(x)|^2)dx,\qquad
    \|f-\smt_Lf\|_2^2\lesssim\frac{\ent[f]}{\log(1+L^2)}.
\end{equation}
For unforced profiles $f=g$ and $f=g_\delta^S$, we have the estimate $|\partial_t\ent[f]|\lesssim t^{-2}\|\nabla f\|_2^2$. For localized forced profile, we use $\ent[g_\delta]\leq \ent[g]$. Together with the gradient estimate \eqref{eq:K_bounds} and the nonnegative frozen energy, we obtain uniform entropy bounds on $g_\delta$ and $g_\delta^S$, including uniformity in the starting time $S$.

\medskip

To compare the two flows $g_\delta$ and $g_\delta^S$, we first express their squared $L^2$ distance in terms of their masses and $\Re \la g_\delta, g_\delta^S\ra$. Since they agree at time $S$ and the frozen flow conserves mass, it suffices to control the variation of $\|g_\delta(t)\|_2^2$ and the change in $\Re\la g_\delta(t),g_\delta^S(t)\ra$. The integrable forcing controls the mass variation. For the pairing, the entropy bound \eqref{eq:intro_entropy} allows us to replace it by $\Re\la \Tone g_\delta, \Ttwo g_\delta^S\ra $. The bounded amplitudes $\Tone g_\delta$ and $\Ttwo g_\delta^S$ make the $L^2$ field estimate
\eqref{eq:intro_field_limit} usable.

When differentiating this truncated pairing, phase equivariance of $\smt_L$ preserves the cancellation of the common coefficient $\non$. However, truncation disturbs the cancellation of the Laplacians. One integration by parts leaves only the first derivatives. After integration in time, these contributions are controlled by
\begin{equation}\label{eq:intro_two_truncation}
    \frac{L_2}{L_1}\nata_\delta(S)
    +\frac{L_1}{L_2}\nata_\delta^S(S),
    \qquad
    \nata_\delta(S)\longrightarrow0,\qquad
    \sup_{S\geq1}\nata_\delta^S(S)<\infty,
\end{equation}
where
\begin{equation}\label{eq:intro_nata}
    \nata_\delta(R):=\int_R^\infty t^{-2}\|\nabla g_\delta(t)\|_2^2\,dt, \qquad \nata^S_\delta(R):=\int_R^\infty t^{-2}\|\nabla g_\delta^S(t)\|_2^2\,dt,\quad\text{for } R\geq S.
\end{equation}
Since these estimates do not show that $\nata^S_\delta(S)$ converge to $0$ as $S\to \infty$, equal truncation levels would leave a part in \eqref{eq:intro_two_truncation} that need not vanish as $S\to \infty$. We therefore first take $L_1$ large to control the entropy error, then $L_2/L_1$ large to suppress the frozen tail, and finally $S$ large to control the original gradient tail and the field and forcing errors.

\medskip

We next fix $S$ and prove that $h_\delta^S$ is Cauchy in $L^2$. Mass conservation again reduces time increments to a pairing, now with the fixed value $h_\delta^S(R)$ at a later time $R$. Truncating this fixed value leads to the test $e^{i\non\log(t/R)}\smt_{L_2}g_\delta^S(R)$. Its bounded amplitude allows the spatial derivative of the phase to be controlled using only $\nabla\non\in L^2$, so one integration by parts is legitimate. Since the entropy error is small for $L_2$ large, uniformly in $S$, we keep $S$ fixed, and the remaining terms vanish as $R\to\infty$ using $\nata_\delta^S(R)\to0$.

Thus, for both defocusing problems, the order of choices is
\begin{equation}\label{eq:intro_order_pick}
    \delta\ \longrightarrow\ L_1\ \longrightarrow\ L_2/L_1\
    \longrightarrow\ S\ \longrightarrow\ R,
\end{equation}
with $\delta$ small and subsequent parameters large. 

The comparison transfers the frozen convergence to $h_\delta$. Removing the spatial cutoff gives strong $L^2$ convergence of $h$. The resulting $L^1$ convergence of densities then identifies $\Phi_\infty=\phi_{U_\infty}$.

\medskip

\paragraph{\bf Radial symmetry and the sharp threshold.} For radial Coulomb solutions with $\lambda<0$, set $\g(t,a)=\sqrt{4\pi}\,a g(t,a)$, so $\|\g(t)\|_{L^2(\dom)}^2=\mass$. Newton's formula rewrites the energy as
\begin{equation}\label{eq:intro_endpoint}
\begin{split}
    \kfun(t)&=t^{-1}\|\partial_a\mathfrak g(t)\|_2^2
    +\left(\kappa-\frac{|\lambda|\mass}{2}\right)\int_0^\infty\frac{|\mathfrak g(t,a)|^2}{a}\,da\\
    &\quad+\frac{|\lambda|}{2}\iint_{0<a<b}\left(\frac1a-\frac1b\right)
    |\mathfrak g(t,a)|^2|\mathfrak g(t,b)|^2\,da\,db.
\end{split}
\end{equation}
Below the threshold, every term is nonnegative. At equality, we prove convexity and eventual linear growth of the first radial moment $\mathcal J(t):=\int_0^\infty r|v(t,r)|^2 dr$, where
$v(t,r)=\sqrt{4\pi}\,r u(t,r)$. This forces a fixed amount of mass to remain at positive rescaled radius.
Its interaction with the inner region through the positive double integral then recovers the uniform bound on $\int_0^\infty a^{-1}|\g(t,a)|^2\,da$.

Newton's formula also improves field convergence away from the origin to $\sup_{a\geq\delta}|\phi_g(t,a)-\Phi_\infty(a)|\lesssim\delta^{-2}t^{-1/2}$. Extending the full limiting coefficient constantly across $(0,\delta)$ yields a bounded Lipschitz coefficient. Ordinary $L^2$ comparison and quantitative frozen time increments without amplitude truncation then give the scattering rate. Above the threshold, thin shells and mass-preserving dilation give negative energy. We then construct a radial energy minimizer at the prescribed mass. This then produces the claimed standing waves. 

\medskip

\paragraph{\bf Small focusing data for general potentials.} 
We replace defocusing coercivity by a bootstrap for the positive form
\[
\mathcal{I}(t):=\|\nabla g(t)\|_2^2+t^2 \int V(ty)|g|^2, \qquad \partial_t \ifun(t) \leq  \frac{b+C_\lambda\|g(t)\|_3^2}{t}\,\ifun(t),
\]
under assumption \eqref{eq:F1}. Gagliardo-Nirenberg interpolation combines the density equation's negative-order estimate with a first-derivative bound to give 
\begin{equation}
    \big\||g(t)|^2-|g(s)|^2\big\|_{3/2}
    \lesssim s^{-1/2} \sup_{\tau\in[s,t]} \|\nabla g(\tau)\|_2^2, \qquad s\leq t\leq2s.
\end{equation}
Under $\ifun(t)^{1/2}\lesssim\varepsilon t^\alpha$, these increments are dyadically summable for $\alpha<1/4$, keeping $\|g(t)\|_3$ small and closing the bootstrap when $b/2<\alpha<1/4$. The improved $L^6$ field convergence and small $L^3$ norm of $\nabla\Phi_\infty$ allow direct frozen comparison with no amplitude truncation needed. For a Coulomb potential, \eqref{eq:F1} would require $b\geq1$. This is where the present bootstrap fails to extend to the nonradial Coulomb problem.

\medskip

\paragraph{\bf Organization of the paper}
Section \ref{sec:est} establishes the rescaled identities, field estimates, and truncation tools. Section \ref{sec:comparison} proves the entropy bounds and frozen comparison, establishes both defocusing theorems, and prepares quantitative estimates for focusing problems. Section \ref{sec:coulomb} treats radial focusing Coulomb scattering and its sharp threshold. Section \ref{sec:general} proves the small-data focusing theorem for general potentials.

%%%%%%%%%%%%%%%%%%%%%%%%%%%%%%%%%%%%%%%
%%%%%%%%%%%%%%%%%%%%%%%%%%%%%%%%%%%%%%%
%%%%%%%%%%%%%%%%%%%%%%%%%%%%%%%%%%%%%%%

\subsection*{Acknowledgment}
This project grew out of discussions at the Brown--Yale PDE Seminar. The author thanks Wilhelm Schlag and Beno\^it Pausader for introducing her to the problem. She also thanks Ebru Toprak for teaching her the linear theory of $-\Delta+|x|^{-1}$ in the summer of 2024, and Wenrui Huang for explaining the papers
\cite{PW,PWY} to her.

%%%%%%%%%%%%%%%%%%%%%%%%%%%%%%%%%%%%%%%
%%%%%%%%%%%%%%%%%%%%%%%%%%%%%%%%%%%%%%%
%%%%%%%%%%%%%%%%%%%%%%%%%%%%%%%%%%%%%%%
%%%%%%%%%%%%%%%%%%%%%%%%%%%%%%%%%%%%%%%
%%%%%%%%%%%%%%%%%%%%%%%%%%%%%%%%%%%%%%%
%%%%%%%%%%%%%%%%%%%%%%%%%%%%%%%%%%%%%%%
%%%%%%%%%%%%%%%%%%%%%%%%%%%%%%%%%%%%%%%
%%%%%%%%%%%%%%%%%%%%%%%%%%%%%%%%%%%%%%%
%%%%%%%%%%%%%%%%%%%%%%%%%%%%%%%%%%%%%%%

\section{Rescaled dynamics and a priori estimates}\label{sec:est}
In this section, we collect the notation and estimates that will be used throughout the paper. After establishing global well-posedness, we derive the identities and estimates for the density, field, and entropy needed for the comparison argument in Section \ref{sec:comparison}. Throughout this section, $V$ is either the repulsive Coulomb potential $\kappa |x|^{-1}$ with $\kappa>0$, or a potential satisfying \eqref{eq:A1}--\eqref{eq:A3}. 

\subsection{Preliminaries and notation}\label{subsec:pre}
Throughout the paper, all integrals are over $\R^3$ unless otherwise specified. We use the following conventions.
\begin{itemize}
    \item We use both $\|\cdot \|_{p}$ and $\|\cdot \|_{L^p}$ for the $L^p$ norm on the spatial domain under consideration: $\R^3$ for the full-space functions and $\dom=(0,\infty)$ for the reduced radial functions. Unless otherwise specified, spatial integrals and norms are taken with respect to Lebesgue measure on the corresponding domain. The domain will be indicated explicitly whenever needed to avoid ambiguity.
    \vspace{3pt}
    \item We use the inner product $\la f,k\ra=\int f\,\overline{k}$, which is linear in the first argument.
    \vspace{3pt}
    \item We denote by $H^1$ the Sobolev space of functions in $L^2$ whose first weak derivatives belong to $L^2$, equipped with the norm $\|f\|_{H^1}^2=\|f\|_2^2+\|\nabla f\|_2^2$. On $\R^3$, $\dot H^1$ denotes the homogeneous Sobolev space, realized as $\{f\in L^6\,\big|\,\nabla f\in L^2\}$, and $H^{-1}$ denotes the dual of $H^1$. On the half-line, $H_0^1(\dom)$ denotes the subspace of $H^1(\dom)$ with zero trace at the origin. 
    \vspace{3pt}
    \item For a time interval $I$ and a Banach space $X$, $C(I;X)$ and $C^1(I;X)$ denote the spaces of continuous and continuously differentiable $X$-valued functions, respectively, with continuity and differentiation understood in the norm of $X$. We use $C^k(\Omega)$ for functions with continuous derivatives up to order $k$, and $C_c^\infty(\Omega)$ for smooth functions with compact support in $\Omega$.
    \vspace{3pt}
    \item We write $A\lesssim B$ if $A\leq CB$ for some constant $C>0$. Unless otherwise indicated, these constants depend only on the dimension and on the fixed smooth cutoffs. Any additional dependence on the datum $u_0$, $\lambda$, $V$, or the parameters $\delta,L_1,L_2,S,R$ will be displayed.
    \vspace{3pt}
    \item The finite bounds $M_V$, $N_V$ and $L_V$ may depend on $u_0$, $\lambda$ and $V$, but not on $t,\delta,L_1,L_2,S,R$. They are distinct from the structural constants $C_V$ in \eqref{eq:A3} and \eqref{eq:D2}. The form-equivalence constant $c_V$ depends only on $V$, and $C_\chi$, $C_w$ depend only on the indicated fixed functions. 
    \vspace{3pt}
    \item We retain $\Sigma$, $\phi_f$, $\dob[f]$, $\mass$, and $\mathbf E$ from Section \ref{sec:intro} in \eqref{eq:Sigma}--\eqref{energy}. In particular, $\mass=\|u_0\|_2^2$.
    \vspace{3pt}
    \item For a normed space $X$ and a positive function $a(t)$, we write $F(t)=O_X(a(t))$ if $\|F(t)\|_X\leq Ca(t)$ for all sufficiently large $t$, and $F(t)=o_X(a(t))$ if $\|F(t)\|_X/a(t)\to0$ as $t\to\infty$. 
\end{itemize}

\medskip

We use Hardy's inequality on $\R^3$, for $f\in \dot H^1(\R^3)$,
\begin{equation}\label{eq:Hardy_full}
    \left\||x|^{-1}f\right\|_{L^2(\R^3)}\leq 2\|\nabla f\|_{L^2(\R^3)},
\end{equation}
and on $\dom$, for $f\in H_0^1(\dom)$,
\begin{equation}\label{eq:Hardy_half_line}
    \left\|r^{-1}f\right\|_{L^2(\R_+)}\leq 2\|f'\|_{L^2(\R_+)}.
\end{equation}

Writing $|\nabla|:=(-\Delta)^{1/2}$, the Riesz transforms defined as 
\begin{equation}
    \mathcal{R}_j:=\partial_j|\nabla|^{-1}, \quad j=1,2,3,
\end{equation}
satisfy
\begin{equation}
    \big\|\mathcal{R}_j f\big\|_p \lesssim_p\|f\|_p,\qquad 1<p<\infty,\,\,\,\text{for }\, f\in L^p(\R^3). 
\end{equation}

We also use the Gagliardo–Nirenberg interpolation inequality in the following form: for $1<q,r<\infty$, 
\begin{equation}\label{eq:GN}
    \|f\|_p\lesssim_{q,r} \|\nabla f\|_q^{\frac{1}{2}}\big\||\nabla|^{-1}f \big\|_r^{\frac{1}{2}},\qquad \frac{1}{p}=\frac{1}{2q}+\frac{1}{2r}
\end{equation}
whenever the two norms on the right-hand side are finite. This follows by applying the usual Gagliardo–Nirenberg inequality to $|\nabla|^{-1}f$, together with the boundedness of the Riesz transforms.

The Hardy-Littlewood-Sobolev inequality states that, for $0<\alpha<3$ and $1<p<3/(3-\alpha)$,
\begin{equation}\label{eq:hls}
      \big\||x|^{-\alpha}*f\big\|_q \lesssim_{\alpha,p}\|f\|_p,\qquad \text{for }\quad \frac1q=\frac1p+\frac{\alpha}{3}-1,\quad f\in L^p(\R^3).
\end{equation}
For $0<q<3$ and nonnegative function $f\in L^1(\R^3)\cap L^\infty(\R^3)$, splitting the integral at radius $\left(\frac{\|f\|_{1}}{\|f\|_\infty}\right)^\frac{1}{3}$ yields
\begin{equation}
    \int_{\R^3} |x|^{-q}f(x)dx\leq \left(1+\frac{4\pi}{3-q}\right)\|f\|_{1}^{1-\frac{q}{3}}\|f\|_{\infty}^\frac{q}{3},
\end{equation}
while the estimate clearly holds when $\|f\|_\infty=0$. 

We also use the Rellich compact embedding on bounded intervals:
\begin{equation}\label{eq:rellich}
    H^1((0,R))\hookrightarrow L^2((0,R)),\qquad \text{compactly for every } 0<R<\infty.
\end{equation}

\medskip

For the potentials under consideration in this paper, the nonnegative sesquilinear form
\begin{equation}
    q_V[f,g]:=\int \nabla f\cdot \overline{\nabla g}+\int Vf\bar g, \qquad \text{for }f,g\in H^1,
\end{equation}
is well-defined. Since \eqref{eq:Hardy_full} gives
\[
\int \frac{|f|^2}{|x|}dx\leq 2\|f\|_2\|\nabla f\|_2,
\]
and $|x|^{-p}\leq 1+|x|^{-2}$ for $1<p\leq 2$, we have 
\begin{align*}
    \|f\|_{H^1}^2 \leq \|f\|_2^2 +q_V[f,f]\leq c_V\|f\|_{H^1}^2,
\end{align*}
where $c_V$ can be taken as $1+\kappa$ in the Coulomb case and $1+8C_V$ for general potentials under assumptions \eqref{eq:A1}-\eqref{eq:A3}. Thus, $q_V$ is densely defined and closed. The representation theorem associates to $q_V$ a unique nonnegative self-adjoint operator $H_V$ that is characterized by 
\begin{equation}
    q_V[f,g]=\left\la H_V^{1/2}f, H_V^{1/2}g \right\ra, \qquad \mathrm{Dom}\left(H_V^{1/2}\right)=H^1.
\end{equation}
Hence, $H_V$ is the nonnegative self-adjoint realization of $-\Delta+V$ associated with the form $q_V$. For $f\in D(H_V)$, we have $H_Vf=\left(-\Delta +V\right)f$ in the distributional sense. For further details on the definition of operator $H_V$, see \cite[Chapter VI]{Kato}.

\medskip

We record the following estimates for the Coulomb field.
\begin{lemma}\label{lem:coulomb_est}
    For $f\in H^1$, we have
    \begin{equation}\label{eq:coulomb_est}
        \|\phi_f\|_\infty\leq 2\|f\|_2\|\nabla f\|_2, \qquad \|\nabla \phi_f\|_3\lesssim\|f\|_3^2,
    \end{equation}
    and 
    \begin{equation}\label{eq:dob_est}
        \dob[f]\leq 2\|f\|_2^3\|\nabla f\|_2.
    \end{equation}
    Moreover, $\phi_f\in \dot H^1$ and 
    \begin{equation}\label{eq:nabla_coulomb}
        -\Delta \phi_f=4\pi |f|^2,\qquad \|\nabla \phi_f\|_2^2=4\pi \dob[f].
    \end{equation}
    The map $f\mapsto\phi_f f$ is locally Lipschitz from $H^1$ to $H^1$
    \begin{equation}\label{eq:lipschitz}
        \|\phi_f f-\phi_g g\|_{H^1}\lesssim\left(\|f\|_{H^1}^2+\|g\|_{H^1}^2\right)\|f-g\|_{H^1}.
    \end{equation}
\end{lemma}
\begin{proof}
By Cauchy-Schwarz and Hardy's inequality centered at $x$, we have
\[
\phi_f(x)=\int\frac{|f(y)|^2}{|x-y|}\,dy\leq\|f\|_2\left(\int\frac{|f(y)|^2}{|x-y|^2}\,dy\right)^{1/2}\leq2\|f\|_2\|\nabla f\|_2.
\]
Moreover, since $\nabla\phi_f=(\nabla|x|^{-1})*|f|^2$ in the sense of distributions and $|\nabla|x|^{-1}|=|x|^{-2}$, the Hardy-Littlewood-Sobolev inequality in \eqref{eq:hls} gives
\[
\|\nabla\phi_f\|_3\leq\big\||x|^{-2}*|f|^2\big\|_3\lesssim\big\||f|^2\big\|_{3/2}=\|f\|_3^2.
\]
This proves \eqref{eq:coulomb_est}, and \eqref{eq:dob_est} follows from $\dob[f]=\int\phi_f|f|^2\leq\|\phi_f\|_\infty\|f\|_2^2$.

Since $f\in L^{12/5}$, another application of the Hardy-Littlewood-Sobolev inequality yields
\[
\|\phi_f\|_6\lesssim\|\nabla\phi_f\|_2\lesssim\big\||f|^2\big\|_{6/5}=\|f\|_{12/5}^2.
\]
Thus $\phi_f\in\dot H^1$. The fundamental solution identity $-\Delta(|x|^{-1})=4\pi\delta_0$ gives $-\Delta\phi_f=4\pi|f|^2$ in distributions. Pairing this identity with $\phi_f$, which is legitimate since $|f|^2\in L^{6/5}$ and $\phi_f\in\dot H^1\subset L^6$, we obtain
\[
\|\nabla\phi_f\|_2^2= 4\pi\int\phi_f|f|^2=4\pi\dob[f].
\]

Finally, using $\big||f|^2-|g|^2\big|\leq|f-g|(|f|+|g|)$, the same Hardy and Hardy-Littlewood-Sobolev estimates give
\[
\|\phi_f-\phi_g\|_\infty \leq2\big(\|f\|_2+\|g\|_2\big)\|\nabla (f-g)\|_2,\qquad \|\nabla\left(\phi_f-\phi_g\right)\|_3\lesssim\big(\|f\|_3+\|g\|_3\big)\|f-g\|_3.
\]
Writing $\phi_f f-\phi_g g=\phi_f (f-g)+(\phi_f-\phi_g) g$, the product rule, H\"older's inequality, and Sobolev embedding imply
\[
\begin{aligned}
\|\phi_f f-\phi_g g\|_{H^1}
&\lesssim \big(\|\phi_f\|_\infty+\|\nabla\phi_f\|_3\big) \|f-g\|_{H^1}+\big(\|\phi_f-\phi_g\|_\infty+ \|\nabla\left(\phi_f-\phi_g\right)\|_3\big)\|g\|_{H^1}\\
&\lesssim\big(\|f\|_{H^1}^2+\|g\|_{H^1}^2\big)\|f-g\|_{H^1}.
\end{aligned}
\]
This proves \eqref{eq:lipschitz}. 
\end{proof}

\medskip

As explained in Subsection \ref{subsec:ideas}, the external Coulomb field produces a phase correction that is singular at zero velocity. We therefore introduce a cutoff away from the origin. 

Fix a radial $\chi\in C^\infty(\R^3)$ with $0\leq\chi\leq1$, $\chi=0$ on $|x|\leq1$, and $\chi=1$ on $|x|\geq2$, and set $\chi_\delta(x)=\chi(x/\delta)$. For $f\in H^1$ and $\delta>0$, by the product rule and Hardy's inequality,
\begin{equation}\label{eq:cutoff_nabla}
    \begin{split}
        \|\nabla(\chi_\delta f)\|_2&\leq\|\nabla f\|_2+\frac{\|\nabla\chi\|_\infty}{\delta}\|f\|_{L^2(\{\delta\leq|x|\leq2\delta\})}\lesssim_\chi\|\nabla f\|_2,
    \end{split}
\end{equation}
where the bound is uniform in $\delta$. Moreover, since $0\leq1-\chi_\delta\leq1$ and $1-\chi_\delta$ vanishes for $|x|\geq2\delta$,
\begin{equation}\label{eq:tail_cutoff}
    \|(1-\chi_\delta)f \|_2^2\leq\int_{|x|\leq2\delta}|f(x)|^2\,dx\leq2\delta\int\frac{|f(x)|^2}{|x|}\,dx.
\end{equation}
We consider commutator $[\chi_\delta,\Delta]f=\chi_\delta\Delta f-\Delta(\chi_\delta f)$ and this further simplifies to
\[
[\chi_\delta,\Delta]f=-2\nabla\chi_\delta\cdot\nabla f-(\Delta\chi_\delta)f.
\]
Hence, we bound the commutator by
\begin{equation}\label{eq:commutator}
    \|[\chi_\delta,\Delta]f\|_2 \lesssim_\chi\delta^{-1}\|\nabla f\|_2.
\end{equation}

%%%%%%%%%%%%%%%%%%%%%%%%%%%%%%%%%%%%%%%
%%%%%%%%%%%%%%%%%%%%%%%%%%%%%%%%%%%%%%%
%%%%%%%%%%%%%%%%%%%%%%%%%%%%%%%%%%%%%%%

\subsection{Global well-posedness}\label{subsec:GWP}
We establish global well-posedness for all the potentials considered in Theorems \ref{thm:Coulomb_defocusing}-\ref{thm:square_focusing}, so that the subsequent content can focus on modified scattering.
\begin{proposition}\label{prop:gwp}
    Let $V$ be either the repulsive Coulomb potential $V(x)=\kappa|x|^{-1}$, with $\kappa>0$, or a potential satisfying \eqref{eq:A1}--\eqref{eq:A3}. 
    For either sign of $\lambda$ and every $u_0\in H^1$, the Cauchy problem
    \begin{equation}
    i\partial_tu-\Delta u+Vu+\lambda\phi_u u=0,\qquad u(0)=u_0,
    \end{equation}
    admits a unique global solution $u\in C(\R;H^1)\cap C^1(\R;H^{-1})$. Moreover, the solution conserves mass and energy:
    \begin{equation}\label{eq:conservation}
        \mass=\|u(t)\|_2^2,\qquad 
        \mathbf E:=\frac12\|\nabla u(t)\|_2^2+\frac12\int_{\R^3} V(x)|u(t,x)|^2\,dx +\frac{\lambda}{4}\dob[u(t)].
    \end{equation}
    If, in addition, $u_0\in\Sigma$, then $u\in C(\R;\Sigma)$ and
    \begin{equation}\label{eq:moment_bound}
         \|xu(t)\|_2 \leq\|xu_0\|_2+2|t|\sqrt{4\mathbf{E}+|\lambda|^2 \mass^3},\qquad t\in\R. 
    \end{equation}
    For every $T>0$, the solution map is continuous from $H^1$ to $C([-T,T];H^1)$, and from $\Sigma$ to $C([-T,T];\Sigma)$.
\end{proposition}
\begin{proof}
    Since the group $e^{itH_V}$ is strongly continuous and uniformly bounded on $H^1$, the local Lipschitz estimate \eqref{eq:lipschitz} and the Duhamel formula
    \begin{equation*}
        u(t)=e^{itH_V}u_0+i\lambda\int_0^t e^{i(t-s)H_V}\phi_{u(s)}u(s)\,ds
    \end{equation*}
\end{proof}
give local well-posedness in $H^1$, with continuation as long as the $H^1$ norm remains bounded. Since $H_V:H^1\to H^{-1}$ is bounded, the equation also gives $\partial_tu\in C(I;H^{-1})$ on the maximal existence interval $I$.

A standard computation, using the reality of $V$ and $\phi_u$ and the symmetry of the Coulomb kernel, gives the conservation laws in \eqref{eq:conservation}. By \eqref{eq:dob_est}, we have
\begin{align*}
    \mathbf{E} &\,\geq\, \frac{1}{2}\|\nabla u(t)\|_2^2+\frac{1}{2}\int V|u(t)|^2-\frac{|\lambda|}{2}\mass ^\frac{3}{2}\|\nabla u(t)\|_2\\
    &\,\geq\, \frac{1}{4}\|\nabla u(t)\|_2^2+\frac{1}{2}\int V|u(t)|^2-\frac{|\lambda|^2}{4}\mass^3.
\end{align*}
Since $V\geq 0$, we have
\begin{equation}\label{eq:nabla_u_bound}
    \|\nabla u(t)\|_2^2\leq 4\mathbf{E}+|\lambda|^2\mass^3, \qquad \int V|u(t,x)|^2dx\leq 2\mathbf{E}+\frac{|\lambda|^2}{2}\mass^3.
\end{equation}
Together with mass conservation, this bounds the $H^1$ norm uniformly and proves global existence. Continuous dependence in $H^1$ on every bounded time interval follows from the local theory.

Suppose now that $u_0\in\Sigma$. Using the local mass identity
\begin{equation}\label{eq:local_mass_u}
    \partial_t |u|^2=2\nabla \cdot \Im(\bar u\nabla u), 
\end{equation}
integrating by parts and applying Cauchy–Schwarz yield
\begin{align*}
    \big|\partial_t\|xu(t)\|_2^2\big| = 4\left|\Im\int x\cdot\overline{u(t)}\nabla u(t)dx\right| \leq 4\|xu(t)\|_2\|\nabla u(t)\|_2.
\end{align*}
Integrating and using \eqref{eq:nabla_u_bound}, we obtain
\begin{equation}\label{eq:weighted_u}
    \|xu(t)\|_2\leq\|xu_0\|_2+ 2|t|\sqrt{4\mathbf{E}+|\lambda|^2 \mass^3}.
\end{equation}
The standard localized moment argument gives $u\in C(\R;\Sigma)$. The corresponding difference estimate, together with continuous dependence in $H^1$, gives continuous dependence in $\Sigma$ on bounded time intervals.

%%%%%%%%%%%%%%%%%%%%%%%%%%%%%%%%%%%%%%%
%%%%%%%%%%%%%%%%%%%%%%%%%%%%%%%%%%%%%%%
%%%%%%%%%%%%%%%%%%%%%%%%%%%%%%%%%%%%%%%

\subsection{Rescaled dynamics and energy identities}\label{subsec:kfun}
From now on, assume $u_0\in\Sigma$ and write $u=u(t)$ as the unique global solution in Proposition \ref{prop:gwp}. For $t\geq1$, as explained in Subsection \ref{subsec:ideas}, we define
\begin{equation}\label{eq:g}
g(t,y):=t^{\frac{3}{2}}e^{i\frac{t|y|^2}{4}}u(t,ty).
\end{equation}
Then $g\in C([1,\infty);\Sigma)$ and satisfies
\begin{equation}\label{eq:PDE_g_general}
    i\partial_tg-t^{-2}\Delta g+V(ty)g =-\frac{\lambda}{t}\phi_g g,
\end{equation}
whereas $g$ satisfies \eqref{eq:intro_PDE_g_coulomb} when $V$ is the Coulomb potential. The scaling relations between $u$ and $g$ are
\begin{equation}
    t^{-1}\phi_g(t,y)=\phi_u(t,ty), \qquad \dob[g(t)]=t\dob[u(t)], \qquad \|g(t)\|_2^2=\mass,
\end{equation}
and 
\begin{equation}\label{eq:Galilean}
    \|\nabla g(t)\|_2=\frac12\|(2t\nabla+ix)u(t)\|_2, \qquad \|yg(t)\|_2=t^{-1}\|xu(t)\|_2.
\end{equation}
Similarly to identity \eqref{eq:local_mass_u}, the real multiplication terms in \eqref{eq:PDE_g_general} cancel and give the density-flux relation
\begin{equation}\label{eq:density_flux_g}
    \partial_t|g|^2=2t^{-2}\nabla\cdot \Im(\bar g\nabla g).
\end{equation}
The corresponding interaction identity is
\begin{equation}\label{eq:dob_g}
    \partial_t \,\dob[g(t)]= -\frac{4}{t^2}\int \nabla \phi_g\cdot \Im(\bar g \nabla g). 
\end{equation}
This integral is absolutely convergent, since
\[
\int|\nabla\phi_g|\, |g|\, |\nabla g| \lesssim\|g\|_3^2\|\nabla g\|_2^2.
\]
This density-flux relation \eqref{eq:density_flux_g} also inspires us to define the rescaled energy
\begin{equation}\label{eq:K}
    \kfun(t):=\frac{\|\nabla g(t)\|_2^2}{t} +t\int V(ty)|g(t,y)|^2\,dy +\frac{\lambda}{2} \dob[g(t)],
\end{equation}
and, since $\kfun$ need not be positive when $\lambda<0$, we also define
\begin{equation}\label{eq:I}
    \ifun[f](t):=\|\nabla f\|_2^2+t^2\int V(ty)|f(y)|^2\,dy,
\end{equation}
and we abbreviate $\ifun[g(t)](t)$ as $\ifun(t)$ for simplicity.

\begin{lemma}\label{lem:rescaled_energy_identity}
    In integrated form, we have
    \begin{equation}\label{eq:Kprime}
        \kfun'(t)=-t^{-2}\|\nabla g(t)\|_2^2+ \int(V+x\cdot\nabla V)|u(t,x)|^2\,dx.
    \end{equation}
    We also have
    \begin{equation}\label{eq:Iprime}
        \begin{split}
            \mathcal{I}'(t)&= t\int\big(2V(ty)+ty\cdot\nabla V(ty)\big)|g(t,y)|^2\,dy+ \frac{2\lambda}{t} \int \nabla\phi_g(t,y)\cdot\Im(\overline{g}\nabla g)(t,y)\,dy.
        \end{split}
    \end{equation}
\end{lemma}
\begin{proof}
    The local mass identity \eqref{eq:local_mass_u} gives
    \begin{equation*}
        \partial_t \|xu(t)\|_2^2=-4\Im\int x\cdot \bar u\nabla u\,dx.
    \end{equation*}
    Using \eqref{main} and integration by part, we have
    \begin{equation}\label{eq:weighted_momentum}
        \begin{split}
            \partial_t \Im\int x\cdot \bar u\nabla u\,dx &= \Im\int x\cdot\big(\overline{\partial_tu}\nabla u+\bar u\nabla\partial_tu\big)\,dx \\
            &=-2\|\nabla u\|_2^2+\int x\cdot\nabla V\,|u|^2\,dx -\frac{\lambda}{2}\dob[u(t)],
        \end{split}
    \end{equation}
    where the last term is obtained by symmetrizing $\int x\cdot \nabla \phi_u|u|^2$ in $x$ and $y$\footnote{These identities are understood in integrated form. The external terms are finite by $|x||\nabla V|\leq C_V(|x|^{-2}+|x|^{-p})$ and Hardy's inequality; for Coulomb, use $|x||\nabla V|=\kappa|x|^{-1}$. The unsymmetrized Hartree integral is absolutely convergent because $xu\in L^2$, $u\in L^6$, and $|x|^{-2}*|u|^2\in L^3$. Thus the kernel symmetrization is justified by Fubini.}. 

    Expanding \eqref{eq:K} using \eqref{eq:Galilean} gives
    \begin{equation}
        \kfun(t)=2t\mathbf{E}+\frac{\|xu(t)\|_2^2}{4t}+\Im\int x\cdot \overline{u(t,x)}\nabla u(t,x)dx,
    \end{equation}
    thus yielding \eqref{eq:Kprime}. Finally, we write
    \[
    \ifun(t)=t\kfun(t)-\lambda t\frac{\dob[g]}{2},
    \]
    so \eqref{eq:dob_g} and \eqref{eq:Kprime} give \eqref{eq:Iprime}. Note that this proof uses the potential through its quadratic form and does not need $V(ty)g\in L^2$, which need not hold for inverse-square singularities.
\end{proof}

\medskip

Suppose that $\lambda>0$. For $V(x)= \kappa|x|^{-1}$, we have $V+x\cdot\nabla V=0$. For general potentials satisfying \eqref{eq:A1}--\eqref{eq:A3} and \eqref{eq:D1}--\eqref{eq:D2}, we first need to estimate the positive contribution of $V+x\cdot\nabla V$. Repeating the weighted estimate as in \eqref{eq:weighted_momentum} with weight $\la x\ra $ gives 
\begin{align*}
    -\partial_t \left(2\Im\int\frac{x}{\la x\ra}\cdot\bar u\nabla u\,dx\right) &\,=4\int \left(\frac{|\nabla u|^2}{\la x\ra} -\frac{|x\cdot\nabla u|^2}{\la x\ra^3}\right)dx-\int \Delta^2 \la x\ra |u(x)|^2\,dx\\
    &\quad -2\int \nabla \la x\ra \cdot \nabla V|u|^2dx-2\lambda\int\nabla \la x\ra \cdot \nabla \phi_u|u|^2 dx\\
    &\geq -\int \Delta^2 \la x\ra |u(x)|^2\,dx\geq 15\int \la x\ra^{-7}|u(t,x)|^2 dx.
\end{align*}
The first term is nonnegative by convexity of $\la x\ra$, and the third by \eqref{eq:D1}\footnote{For inverse-square potential, $|u|^2\nabla V$ need not be locally integrable at the origin. In the third term here, $\nabla \la x\ra =O(|x|)$ near the origin, so the external potential integral is finite by Hardy’s inequality.}. The second term uses $\Delta^2\la x\ra=-15\la x\ra^{-7}$. The last term is also nonnegative for $\lambda>0$ since symmetrization rewrites it as
\[
\lambda\iint \frac{\big(\nabla \la x\ra -\nabla \la y\ra\big)\cdot(x-y)} {|x-y|^3}|u(x)|^2 |u(y)|^2\,dx\,dy,
\]
and we can easily check 
\[
\big(\nabla \la x\ra-\nabla \la y\ra\big)\cdot(x-y)=\left(\frac1{\la x\ra}+\frac1{\la y\ra}\right) \big(\la x\ra\la y\ra-1-x\cdot y\big)\geq 0.
\]

Mass and energy conservation in \eqref{eq:conservation} bound
\begin{equation*}
    \left|2\Im \int\frac{x}{\la x\ra}\cdot\overline{u(t)}\nabla u(t)\,dx\right|\leq 2\|u(t)\|_2\|\nabla u(t)\|_2\leq 2\sqrt{2\mass\mathbf{E}}, \quad t\geq 0.
\end{equation*}
Hence, integrating in time and using \eqref{eq:D2}, we obtain
\begin{equation}\label{eq:Kprime_bound}
    \int_1^\infty\int(V+x\cdot\nabla V)_+ |u|^2\,dx\,dt \leq C_V\int_1^\infty\int\la x\ra^{-7}|u(t)|^2 dxdt \leq\frac{4\sqrt2}{15}C_V\sqrt{\mass\mathbf E}.
\end{equation}
Consequently, integrating the identity for $\kfun'$ yields
\begin{equation}
    \kfun(t)+\int_1^t s^{-2}\|\nabla g(s)\|_2^2ds\leq \kfun(1)+\frac{4\sqrt2}{15}C_V\sqrt{\mass\mathbf E}\,=: M_V<\infty, \quad t\geq 1.
\end{equation}
For the Coulomb potential, the same inequality holds with $M_V=\kfun(1)$. Since all terms in $\kfun$ are nonnegative for $\lambda>0$, we therefore obtain
\begin{equation}\label{eq:nabla_g_grow}
    \|\nabla g(t)\|_2\leq \sqrt{M_V}\,t^\frac{1}{2}, \qquad \dob[g(t)]\leq \frac{2M_V}{\lambda}, \qquad t\geq 1.
\end{equation}
and 
\begin{equation}\label{eq:integrability_nabla_g}
    \int_1^\infty s^{-2}\|\nabla g(s)\|_2^2ds\leq M_V.
\end{equation}
as in \eqref{eq:K_bounds}.

\medskip

Suppose now that $\lambda<0$ and that $V$ satisfies \eqref{eq:A1}--\eqref{eq:A3} and \eqref{eq:F1} with $0\leq b<1/2$. The negative Coulumb interaction contribution to $\kfun$ prevents us from using the preceding positivity argument. We therefore work with $\ifun$, which remains nonnegative and still controls $\|\nabla g\|_2^2$. Its evolution separates the growth allowed by \eqref{eq:F1} from the nonlinear contribution, which we will control through $\|g\|_3$. We first hold $f$ fixed to isolate the contribution from the time dependence of the potential:
\begin{equation}\label{eq:identity_deriving_ifun}
    \partial_t \ifun[f](t)=t\int\big(2V(ty)+ty\cdot\nabla V(ty)\big)|f(y)|^2\,dy.
\end{equation}
Applying \eqref{eq:identity_deriving_ifun} and \eqref{eq:F1} to $f_t(x):=t^{-3/2}f(x/t)$ and changing variables, we obtain
\begin{align*}
    \partial_t \ifun[f](t) &=t \int\big(2V(x)+x\cdot\nabla V(x)\big)|f_t(x)|^2dx\leq \frac{b}{t}\left(\|\nabla f\|_2^2+t^2\int V(ty)|f(y)|^2\,dy\right).
\end{align*}
Gronwall inequality therefore gives, for every fixed $f\in H^1$,
\begin{equation}\label{eq:ifun_fixed}
\ifun[f](t)\leq\left(\frac{t}{s}\right)^b \ifun[f](s), \qquad t\geq s\geq1.
\end{equation}
Along the solution, Lemma \ref{lem:coulomb_est}, H\"older inequality and Sobolev embedding also control the nonlinear contribution to $\ifun'$ by
\begin{align*}
    \left|\int \nabla \phi_g(t,y)\cdot\Im(\overline{g}\nabla g)(t,y)\,dy\right|\leq \|\nabla\phi_g(t)\|_3\|g(t)\|_6\|\nabla g\|_2\lesssim \|g(t)\|_3^2\,\ifun(t).
\end{align*}
Summing up, \eqref{eq:Iprime} gives
\begin{equation}\label{eq:I_bound}
    \ifun'(t)\leq \frac{b+C|\lambda|\|g(t)\|_3^2}{t}\,\ifun(t),\qquad t\geq 1.
\end{equation}
\eqref{eq:I_bound} provides the energy bound for the focusing bootstrap. Indeed, if the bootstrap supports $\|g(t)\|_3\leq\varepsilon$ , Gronwall’s inequality gives
\begin{equation}
    \ifun(t)\leq \ifun(1)t^{b+C|\lambda|\varepsilon^2}
\end{equation}
Choosing $b/2<\alpha<1/4$ and $\varepsilon$ sufficiently small makes this compatible with the bound $\mathcal I(t)\lesssim\varepsilon^2 t^{2\alpha}$ when $\ifun(1)\lesssim \varepsilon^2$. The bootstrap will be discussed in details in Section \ref{sec:general}. 

We finally estimate the initial rescaled norm. Identity \eqref{eq:g} and estimates \eqref{eq:weighted_u} and \eqref{eq:nabla_u_bound} yield
\begin{equation}
    \ifun(1)\leq 2\|\nabla u(1)\|_2^2+\int V|u(1)|^2\,dx+\frac12\|xu(1)\|_2^2\leq 6\big(4\mathbf E+|\lambda|^2\mathbf M^3\big)+\|xu_0\|_2^2.
\end{equation}
Since $\lambda<0$, the initial energy satisfies
\[
\mathbf E\leq\frac12\|\nabla u_0\|_2^2+ \frac12\int V|u_0|^2\,dx.
\]
Thus assumption \eqref{eq:A3} and \eqref{eq:Hardy_full} imply
\begin{equation}\label{eq:small_initial}
    \mass+\ifun(1)\lesssim (1+C_V)\|u_0\|_\Sigma^2+|\lambda|^2\|u_0\|_\Sigma^6.
\end{equation}
This supplies the smallness at $t=1$ used in the focusing bootstrap.

%%%%%%%%%%%%%%%%%%%%%%%%%%%%%%%%%%%%%%%
%%%%%%%%%%%%%%%%%%%%%%%%%%%%%%%%%%%%%%%
%%%%%%%%%%%%%%%%%%%%%%%%%%%%%%%%%%%%%%%

\subsection{Density estimates and the limiting Hartree field}\label{subsec:limiting_field}
We now construct a limiting Coulomb interaction field under explicit bounds on $\|\nabla g(t)\|_2$ and $\dob[g(t)]$. These bounds follow from Subsection \ref{subsec:kfun} in the defocusing case and will be established separately for the focusing problems in Sections \ref{sec:coulomb} and \ref{sec:general}. 

\begin{proposition}\label{prop:limiting_field}
    Suppose that
    \begin{equation}\label{eq:field_assumption}
        A:=\sup_{t\geq1}\frac{\|\nabla g(t)\|_2^2}{t}<\infty,\qquad B:=\sup_{t\geq1} \dob[g(t)]<\infty.
    \end{equation}
    Then there is a nonnegative $\Phi_\infty\in\dot H^1$ such that, for $t\geq1$,
    \begin{equation}\label{eq:limiting_field}
          \big\|\phi_g(t)-\Phi_\infty\big\|_2 \lesssim\mass^{\frac{1}{4}} A^{\frac{3}{4}} t^{-\frac{1}{4}},\qquad \|\nabla\Phi_\infty\|_2 \leq(4\pi B)^{\frac{1}{2}}.
    \end{equation}
    Additionally, for every $f\in H^1$,
    \begin{equation} \label{eq:limiting_field_pair_est}
        \int\Phi_\infty(x)|f(x)|^2dx =\frac{1}{4\pi}\int \nabla\Phi_\infty(x)\cdot\nabla\phi_f(x) \leq B^{\frac{1}{2}} \dob[f]^{\frac{1}{2}}.
    \end{equation}
\end{proposition}
\begin{proof}
    In the sense of distributions, \eqref{eq:density_flux_g} implies
    \begin{equation}
        \partial_t\phi_g=8\pi t^{-2} (-\Delta)^{-1}\nabla\cdot \Im (\bar g\nabla g).
    \end{equation}
    Hardy-Littlewood-Sobolev inequality, interpolation and Sobolev embedding imply
    \begin{equation}\label{eq:partial_t_phi}
        \begin{split}
            \big\|\partial_t\phi_g(t)\|_2 &\lesssim t^{-2}\|g(t)\|_3\|\nabla g(t)\|_2\lesssim t^{-2}\|g(t)\|_2^\frac{1}{2}\|\nabla g(t)\|_2^\frac{3}{2}\leq \mass^\frac{1}{4}A^\frac{3}{4}t^{-\frac{5}{4}}.
        \end{split}
    \end{equation}
    Hence, we have that $\phi_g(t)-\phi_g(1)$ converges in $L^2$ and we denote this limit by $\Phi_\infty-\phi_g(1)$. Note that this convergence does not assert $\Phi_\infty$ belongs to $L^2$. Moreover, integrating \eqref{eq:partial_t_phi}, we obtain
    \begin{equation*}
       \big\| \phi_g(t)-\Phi_\infty\big\|_2\leq\int_t^\infty\|\partial_s\phi_g(s)\|_2ds \lesssim\mass^\frac{1}{4}A^\frac{3}{4}t^{-\frac{1}{4}},
    \end{equation*}
    proving the first estimate in \eqref{eq:limiting_field}. In particular, $\phi_g(t)-\Phi_\infty\to 0$ in $L^2$, and therefore $\phi_g(t_n)\to \Phi_\infty$ almost everywhere along some sequence $t_n\to\infty$. Since $\phi_g(t)\geq 0$, we therefore have $\Phi_\infty \geq 0$.

    Moreover, \eqref{eq:nabla_coulomb} gives
    \begin{equation*}
        \|\nabla\phi_g(t)\|_2^2=4\pi\dob [g(t)] \leq 4\pi B.
    \end{equation*}
    Thus, up to passing to a subsequence, $\phi_g(t_n)$ converges weakly in $\dot H^1$ to $\Phi_\infty$ by uniqueness of limit. Hence, we have $\Phi_\infty\in \dot H^1$ and 
    \begin{equation*}
        \|\nabla\Phi_\infty\|_2^2\leq \liminf_{n\to\infty}\|\nabla\phi_g(t_n)\|_2^2\leq 4\pi B,
    \end{equation*}
    proving \eqref{eq:limiting_field}.

    Finnaly, for $f\in H^1$, Sobolev embedding gives $|f|^2\in L^{6/5}$ and $\Phi_\infty\in L^6$. Testing the weak equation $-\Delta\phi_f=4\pi|f|^2$ in \eqref{eq:nabla_coulomb} against $\Phi_\infty$ and applying H\"older inequality yield
    \begin{align*}
        \int\Phi_\infty(x)|f(x)|^2dx&=\frac{1}{4\pi}\int \nabla\Phi_\infty(x)\cdot\nabla\phi_f(x)dx\leq \frac{1}{4\pi}\|\nabla\Phi_\infty\|_2\|\nabla\phi_f\|_2.
    \end{align*}
    Applying \eqref{eq:limiting_field} and \eqref{eq:nabla_coulomb} proves \eqref{eq:limiting_field_pair_est}. 
\end{proof}
\begin{remark}\label{rmk:defocusing_AB}
    For the \emph{defocusing} interactions, \eqref{eq:field_assumption} has been verified. In fact, \eqref{eq:nabla_g_grow} gives
    \begin{equation}\label{eq:defocusing_AB}
        A\leq M_V,\qquad B\leq \frac{2M_V}{\lambda}. 
    \end{equation}
\end{remark}

We also record in this subsection the density estimate to be used in Section \ref{sec:general}.
\begin{lemma}\label{lem:density_diff}
    For $1\leq s\leq t\leq 2s$, we have
    \begin{equation}\label{eq:density_diff}
        \big\||g(t)|^2-|g(s)|^2\big\|_\frac{3}{2}\lesssim s^{-\frac{1}{2}}\sup_{\tau\in[s,t]}\|\nabla g(\tau)\|_2^2.
    \end{equation}
\end{lemma}
\begin{proof}
   The density-flux relation \eqref{eq:density_flux_g} controls the density increment with one derivative lost. We combine this bound with the Gagliardo–Nirenberg inequality to recover an $L^\frac{3}{2}$ estimate \eqref{eq:density_diff}. \eqref{eq:density_flux_g}, boundedness of the Riesz transforms, and Sobolev embedding give 
   \begin{equation}\label{eq:density_diff_1}
       \begin{split}
           \big\||\nabla|^{-1}\left(|g(t)|^2-|g(s)|^2\right)\big\|_\frac{3}{2}\lesssim\int_s^t\tau^{-2}\|g(\tau)\|_6\|\nabla g(\tau)\|_2d\tau\lesssim s^{-1}\sup_{\tau\in[s,t]}\|\nabla g(\tau)\|_2^2.
       \end{split}
   \end{equation}
   Moreover, 
   \begin{equation}\label{eq:density_diff_2}
       \big\|\nabla\left(|g(t)|^2-|g(s)|^2\right)\big\|_\frac{3}{2}\lesssim \sup_{\tau\in[s,t]}\|g(\tau)\|_6\|\nabla g(\tau)\|_2\lesssim \sup_{\tau\in[s,t]}\|\nabla g(\tau)\|_2^2.
   \end{equation}
   Applying the Gagliardo–Nirenberg inequality \eqref{eq:GN}, we have 
   \begin{align*}
       \big\||g(t)|^2-|g(s)|^2\big\|_\frac{3}{2}&\lesssim \big\||\nabla|^{-1}\left(|g(t)|^2-|g(s)|^2\right)\big\|_\frac{3}{2}^\frac{1}{2}\big\|\nabla\left(|g(t)|^2-|g(s)|^2\right)\big\|_\frac{3}{2}^\frac{1}{2}.
   \end{align*}
   Plugging in estimates \eqref{eq:density_diff_1} and \eqref{eq:density_diff_2} proves \eqref{eq:density_diff}. 
\end{proof}

%%%%%%%%%%%%%%%%%%%%%%%%%%%%%%%%%%%%%%%
%%%%%%%%%%%%%%%%%%%%%%%%%%%%%%%%%%%%%%%
%%%%%%%%%%%%%%%%%%%%%%%%%%%%%%%%%%%%%%%

\subsection{Entropy and amplitude truncation}\label{subsec:smt}
As explained in Subsection \ref{subsec:ideas}, the available gradient bounds and $L^2$ convergence of the field differences do not suffice for a direct comparison of the original and frozen evolutions. We therefore introduce smooth amplitude truncations, together with an entropy functional that controls the discarded mass.

For $f\in H^1$, define
\begin{equation}\label{eq:entropy}
    \ent[f]:=\int |f(y)|^2 \log(1+|f(y)|^2)\,dy.
\end{equation}
Since 
\[
\log(1+r)=\int_0^r\frac{ds}{1+s} \leq\int_0^r s^{-2/3}ds =3r^{1/3},
\]
interpolation gives
\begin{equation}\label{eq:ent_well_defined}
\ent[f]\lesssim\|f\|_{\frac{8}{3}}^{\frac{8}{3}} \lesssim \|f\|_2 ^{\frac{5}{3}} \|\nabla f\|_2.
\end{equation}

\medskip

We now construct the amplitude truncations. For $L>0$, let 
\begin{equation}\label{eq:Truncation}
    T_L(z):=
    \begin{cases}
        z,\qquad &|z|\leq L,\\[4pt]
        \displaystyle L\frac{z}{|z|}, \qquad &|z|>L.
    \end{cases}
\end{equation}
We then construct the smooth truncation from $T_L$. Fix $w\in C_c^\infty(1,2)$ with $w\geq 0$ and $\int w=1$. Set
\begin{equation}\label{eq:smt}
    \smt_L(z):=\int _1^2 w(\theta)T_{\theta L}(z)\,d\theta.
\end{equation}
We prove the following properties on this smooth amplitude truncation operator.
\begin{lemma}\label{lem:smt}
    The map $\smt_L:\C\to \C$ is smooth, equals the identity on $|z|\leq L$, and satisfies
    \begin{equation}\label{eq:smt_bound_map}
        \big|\smt_L(z)\big|\leq \min\{|z|,2L\},
    \end{equation}
    and 
    \begin{equation}\label{eq:Dsmt_bound}
        \big\|D\smt_L(z)\|_{\mathrm{op}}\leq 1, \qquad \big\|D^2\smt_L(z)\|_{\mathrm{op}}\lesssim_w L^{-1}.
    \end{equation}
    Here, derivatives are taken over $\R^2\simeq\C$, and $\|\cdot\|_{\mathrm{op}}$ denotes the operator norm induced by the Euclidean norm. 
    For $\theta\in\R$,
    \begin{equation}\label{eq:smt_phase_equivariance}
        \smt_L(e^{i\theta}z)= e^{i\theta}\smt_L(z),\qquad D\smt_L(z)[iz]=i\smt_L(z).
    \end{equation}
    For $f\in H^1$,
    \begin{equation}\label{eq:smt_bdd}
         \|\smt_L f\|_p\leq\|f\|_p\quad\text{for }\,2\leq p \leq 6, \qquad \|\nabla\smt_L f\|_2\leq\|\nabla f\|_2,
    \end{equation}
    and
    \begin{equation}\label{eq:smt_tail}
         \big\|f-\smt_L f\big\|_2^2\leq\frac{\ent[f]}{\log(1+L^2)}.
    \end{equation}
\end{lemma}
\begin{proof}
    The identity property and \eqref{eq:smt_bound_map} follow directly from the definition. For $z\neq0$, write 
    \[
    \smt_1(z)=a(|z|)\frac{z}{|z|}, \qquad  a(r):=\int_1^2 w(\theta) \min\{r,\theta\}d\theta.
    \]
    Extending $w$ by zero outside $(1,2)$, we have
    \[
    a'(r)=\int_r^\infty w(\theta) d\theta, \qquad a''(r)=-w(r).
    \]
    Thus $a$ is smooth, equals $r$ for $r\leq1$, and is constant for $r\geq2$. In particular, $\smt_1$ is smooth, including at the origin. The radial and tangential eigenvalues of $D\smt_1(z)$ are $a'(|z|)$ and $a(|z|)/|z|$, respectively, and both belong to $[0,1]$. Hence $\|D\smt_1(z)\|_{\mathrm{op}}\leq1$. 
    Moreover, $D^2\smt_1=0$ on $|z|\leq1$, is bounded on $1\leq|z|\leq2$, and is $O_w(|z|^{-2})$ for $|z|\geq2$. The scaling identity $\smt_L(z)=L\smt_1(z/L)$ therefore gives \eqref{eq:Dsmt_bound}. 

    Phase equivariance follows from the radial form of $\smt_L$. Differentiating this identity in $\theta$ at $\theta=0$ yields \eqref{eq:smt_phase_equivariance}.
    
    The $L^p$ bound in \eqref{eq:smt_bdd} follows from \eqref{eq:smt_bound_map}, while the Sobolev chain rule gives
    \begin{align*}
        \big|\nabla\smt_L f\big|^2 =\sum_{j=1}^3 \big|D\smt_L(f)[\partial_jf]\big|^2 \leq |\nabla f|^2,
    \end{align*}
    proving \eqref{eq:smt_bdd}. 
    
    Finally, $\smt_L(z)$ is a nonnegative multiple of $z$, with modulus at most $|z|$, and equals $z$ for $|z|\leq L$. Consequently,
    \begin{align*}
        \|f-\smt_L f\|_2^2\leq \int_{\{|f|>L\}}|f(y)|^2dy \leq\frac{\ent[f]}{\log(1+L^2)}.
    \end{align*}
\end{proof}

%%%%%%%%%%%%%%%%%%%%%%%%%%%%%%%%%%%%%%%
%%%%%%%%%%%%%%%%%%%%%%%%%%%%%%%%%%%%%%%
%%%%%%%%%%%%%%%%%%%%%%%%%%%%%%%%%%%%%%%
%%%%%%%%%%%%%%%%%%%%%%%%%%%%%%%%%%%%%%%
%%%%%%%%%%%%%%%%%%%%%%%%%%%%%%%%%%%%%%%
%%%%%%%%%%%%%%%%%%%%%%%%%%%%%%%%%%%%%%%

\section{Comparison with the frozen evolution}\label{sec:comparison}
In this section, we prove the comparison and convergence estimates described in Subsection \ref{subsec:ideas} for the convergence of the corrected profiles. The first three subsections develop the comparison argument and apply it to the \emph{defocusing} problems, using the energy bounds and limiting field from Section \ref{sec:est}. The last subsection prepares the focusing arguments by giving quantitative estimates for more regular fields, with \emph{no} sign condition.

\subsection{Frozen evolutions and entropy bounds}\label{subsec:frozen_evolution}

The limiting field determines the correction phase, but convergence of the field alone does not give convergence of the profile. We therefore freeze the field and consider the corresponding evolution. 

Define the modified profile as
\begin{equation}\label{eq:h}
    h(t)=e^{-i\non_\infty\log t}g(t),\qquad
    \non_\infty(y)=
    \begin{cases}
        \displaystyle \frac{\kappa}{|y|}+\lambda\Phi_\infty, \quad &\text{Coulomb potential,}\\[10pt]
        \lambda \Phi_\infty, \quad &\text{general potentials}.
    \end{cases}
\end{equation}

For both problems, we use the spatial cutoff from Subsection \ref{subsec:pre} and define
\begin{equation}\label{eq:profile_cutoff}
    g_\delta=\chi_\delta g,\qquad h_\delta=\chi_\delta h,\qquad 0<\delta<1.
\end{equation}
The cutoff avoids the Coulomb singularity on the support of $g_\delta$. Although unnecessary for general potentials, we also use it for general potentials so that we work with the same localized profiles. For general potentials, the cutoff makes $V(ty)g_\delta$ an $L^1_tL^2_y$ forcing term on $[1,\infty)$. The Coulomb term, however, has the nonintegrable time factor $t^{-1}$ and remains part of the frozen coefficient. Since $g_\delta$ vanishes on $|y|\leq\delta$, we cap the Coulomb term near the origin without changing its product with $g_\delta$, and define
\begin{equation}\label{eq:non_coulomb}
    \non_{\infty,\delta}(y)=\frac{\kappa}{\max\{|y|,\delta\}}+\lambda\Phi_\infty(y).
\end{equation}
Throughout Subsections \ref{subsec:frozen_evolution}--\ref{subsec:convergence_frozen}, $\non$ denotes the frozen coefficient:
\begin{equation}\label{eq:frozen_coefficient}
    \non=\non_{\infty,\delta}\quad\text{for Coulomb,}\qquad \non=\non_\infty\quad\text{for general potentials}.
\end{equation}
It agrees with $\non_\infty$ on the support of $g_\delta$, so $h_\delta=e^{-i\non\log t}g_\delta$. Moreover, for each $\delta>0$, the capped Coulomb coefficient belongs to $L^6$ and has an $L^2$ weak gradient. Together with $\Phi_\infty\in\dot H^1\subset L^6$, this gives $\non\in L^6$ and $\nabla\non\in L^2$ for both the Coulomb and general potentials.

\medskip

For $S\geq 1$, define the frozen evolution and its corrected profile by
\begin{equation}\label{eq:frozen_PDE}
    i\partial_t g_\delta^S-t^{-2}\Delta g_\delta^S=-t^{-1}\non g_\delta^S,\qquad g_\delta^S(S)=g_\delta(S),
\end{equation}
and 
\begin{equation}\label{eq:h_cutoff}
    h_\delta^S(t)=e^{-i\non\log(t)}g_\delta^S(t).
\end{equation}
The frozen solution can spread into $|y|<\delta$, so its equation uses $\non$ on the whole space. The localized profile satisfies
\begin{equation}\label{eq:PDE_g_cutoff}
    i\partial_t g_\delta-t^{-2}\Delta g_\delta=-t^{-1}\non g_\delta-\lambda t^{-1}\left(\phi_g-\Phi_\infty\right)g_\delta+\mathcal R_\delta,
\end{equation}
where, with the convention in \eqref{eq:commutator},
\begin{equation}\label{eq:comparison_forcing}
    \mathcal R_\delta=
    \begin{cases}
        t^{-2}[\chi_\delta,\Delta]g,\quad & \text{Coulomb potential},\\
        t^{-2}[\chi_\delta,\Delta]g-V(ty)g_\delta,\quad &\text{general potentials}.
    \end{cases}
\end{equation}
For fixed $\delta$, this forcing has a small tail. Indeed, \eqref{eq:commutator} and the bound $A$ in \eqref{eq:field_assumption} give
\begin{equation}\label{eq:cutoff_forcing_tail}
    \int_S^\infty t^{-2}\|[\chi_\delta,\Delta]g(t)\|_2\,dt
    \lesssim_\chi\delta^{-1}\sqrt A\,S^{-1/2}.
\end{equation}
For general potentials, \eqref{eq:A3} and $g_\delta=0$ on $|y|\leq\delta$ also give
\begin{equation}\label{eq:external_forcing_tail}
    \int_S^\infty\|V(ty)g_\delta(t)\|_2\,dt
    \leq C_V\sqrt\mass\left(\delta^{-2}S^{-1}
    +\frac{\delta^{-p}}{p-1}S^{1-p}\right).
\end{equation}
Thus $\mathcal R_\delta\in L^1([1,\infty);L^2)$ in both cases. The cutoff is fixed before sending $S\to\infty$ as in \eqref{eq:intro_order_pick}. 

\medskip

The frozen equation \eqref{eq:frozen_PDE} is well-posed for real $\non\in L^6$. If $\non\geq 0$, this supplies the integrated gradient bound needed for the defocusing comparison. 
\begin{lemma}\label{lem:frozen_est}
    Let $\non\in L^6$ be real. Equation \eqref{eq:frozen_PDE} with initial datum $g_\delta(S)\in H^1$, has a unique solution in $C([S,\infty);H^1)\cap C^1([S,\infty);H^{-1})$. Moreover, it has mass conservation
    \begin{equation}\label{eq:frozen_mass}
        \big\|g_\delta^S(t)\big\|_2=\big\|g_\delta(S)\big\|_2\leq \|g(S)\|_2=\sqrt{\mass}.
    \end{equation}
    Define the frozen energy
    \begin{equation*}
        \kfun_\delta^S(t):=\frac{\big\|\nabla g_\delta^S(t)\big\|_2^2}{t}+\int \non(y) |g_\delta^S(t,y)|^2dy.
    \end{equation*}
    We have
    \begin{equation}\label{eq:K_cutoff_prime}
        \partial_t\kfun_\delta^S(t)=-\frac{\big\|\nabla g_\delta^S(t)\big\|_2^2}{t^2}. 
    \end{equation}
    If, in addition, $\non\geq 0$, then 
    \begin{equation}\label{eq:frozen_est}
        \sup_{t\geq S}\frac{\big\|\nabla g_\delta^S(t)\big\|_2^2}{t}\leq \kfun_\delta^S(S), \qquad \int _S^\infty \frac{\big\|\nabla g_\delta^S(t)\big\|_2^2}{t^2}dt\leq \kfun_\delta^S(S). 
    \end{equation}
\end{lemma}
\begin{proof}
    Since $\non\in L^6$ is real, the operators $t^{-2}\Delta-t^{-1}\non$ are self-adjoint on the common domain $H^2$ and depend smoothly on $t$ as maps from $H^2$ to $L^2$. \cite{RS} gives a unitary $L^2$ propagator that preserves $H^2$ and is locally bounded there. Interpolation between $L^2$ and $H^2$ gives the asserted $H^1$ well-posedness, and the equation gives $C^1$ regularity in $H^{-1}$. Unitarity and $|\chi_\delta|\leq1$ give \eqref{eq:frozen_mass}.

    Differentiation and integration by parts give
    \begin{align*}
        \partial_t \kfun_\delta^S(t)&=-\frac{\|\nabla g_\delta^S(t)\|_2^2}{t^2}+2\Re\int\left(-\frac1t\Delta g_\delta^S+\non g_\delta^S\right) \overline{\partial_tg_\delta^S}dy\\
        &=-\frac{\|\nabla g_\delta^S(t)\|_2^2}{t^2}
        +2\Re\int(-it)|\partial_t g_\delta^S|^2 dy = -\frac{\|\nabla g_\delta^S(t)\|_2^2}{t^2}.
    \end{align*}
    If $\non\geq 0$, then both terms in $\kfun_\delta^S$ are nonnegative. \eqref{eq:K_cutoff_prime} then implies \eqref{eq:frozen_est}. 
\end{proof}

Since $\non$ is real, the potential term cancels from the density equation. We next derive estimates on the entropies.
\begin{lemma}\label{lem:ent_evolution}
    Let $f$ denote the original profile $g$ or the frozen solution $g^S_\delta$. For $s\leq t$ in the corresponding interval of existence, we have
    \begin{equation}\label{eq:ent_evolution}
        \big|\ent[f(t)]-\ent[f(s)]\big|\leq 4\int_s^t \tau^{-2}\big\|\nabla f(\tau)\|_2^2d\tau.
    \end{equation}
\end{lemma}
\begin{proof}
   Since $\non$ is real, the frozen profile's density satisfies the same density-flux relation as $g$ in \eqref{eq:density_flux_g}. Hence, we write
   \[
   \partial_t|f|^2=2t^{-2}\nabla \cdot \Im(\bar f\,\nabla f). 
   \]
    Write $F(r):=r\log(1+r)$. The density-flux relation and integration by parts give
    \begin{align*}
        \partial_t \ent[f(t)]=-2t^{-2}\int F''(|f(t,y)|^2)\nabla |f(t,y)|^2\cdot \Im (\bar f\,\nabla f)(t,y)dy
    \end{align*}
    Using $F''(r)=(r+2)/(1+r)^2$ and $0\leq rF''(r)\leq 1$, the integral is bounded in absolute value by $2\|\nabla v\|_2^2$. This proves \eqref{eq:ent_evolution}.
\end{proof}

\medskip

From this point through the end of Subsection \ref{subsec:convergence_frozen}, we work with the \emph{defocusing} problems and assume $\lambda>0$, and, for general potentials, \eqref{eq:A1}--\eqref{eq:A3} and \eqref{eq:D1}--\eqref{eq:D2}. We now use the \emph{defocusing} bounds \eqref{eq:nabla_g_grow} and \eqref{eq:integrability_nabla_g} from Section \ref{sec:est}. Lemma \ref{lem:ent_evolution} gives
\begin{equation}\label{eq:defocusing_ent_bound}
   \sup_{t\geq 1} \ent[g(t)]\leq \ent[g(1)]+4M_V, \qquad \ent[g_\delta(t)]\leq \ent[g(t)].
\end{equation}
The second inequality comes from $|\chi_\delta|\leq 1$. Moreover, since $\Phi_\infty\geq 0$, \eqref{eq:limiting_field_pair_est} and Remark \ref{rmk:defocusing_AB} give
\begin{align*}
    0\leq \int \Phi_\infty(y)\big|g_\delta(S,y)\big|^2 dy\leq \int \Phi_\infty(y)\big|g(S,y)\big|^2 dy\leq \frac{2M_V}{\lambda}.
\end{align*}
For general potentials, together with \eqref{eq:cutoff_nabla} and Remark \ref{rmk:defocusing_AB}, this yields a uniform bound on the initial frozen energy:
\begin{equation}
    \kfun_\delta^S(S)\lesssim_\chi M_V. 
\end{equation}
For the Coulomb potential, the support of $g_\delta$ and positivity of the rescaled energy \eqref{eq:K} give
\[
    \int\frac{\kappa}{\max\{|y|,\delta\}} |g_\delta(S,y)|^2 dy=\kappa \int\frac{|g_\delta(S,y)|^2}{|y|} dy\leq \kappa \int\frac{|g(S,y)|^2}{|y|} dy\leq \kfun(S)\leq M_V.
\]
Hence, in both cases, we have
\begin{equation}\label{eq:initial_frozen_bound}
    \kfun_\delta^S(S)\lesssim_\chi  M_V,
\end{equation}
and Lemma \ref{lem:ent_evolution},  \eqref{eq:frozen_est} and \eqref{eq:integrability_nabla_g} imply
\begin{equation}\label{eq:ent_frozen_bound}
    \begin{split}
        \sup_{S\geq 1}\sup_{t\geq S}\ent[g_\delta^S(t)]\leq \sup_{S\geq 1}\left(\ent[g_\delta(S)]+4\kfun_\delta^S(S) \right)\lesssim_\chi \ent[g(1)]+M_V.
    \end{split}
\end{equation}
These bounds hold uniform in $\delta$. 

For later use, we also define the gradient tails here as
\begin{equation}\label{eq:nata}
    \nata_\delta(S):=\int_S^\infty t^{-2}\big\|\nabla g_\delta(t)\big\|_2^2dt,\qquad \nata_\delta^S(R):=\int_R^\infty t^{-2}\big\|\nabla g_\delta^S(t)\big\|_2^2dt,\quad \text{for }R\geq S.
\end{equation}
\eqref{eq:integrability_nabla_g}, \eqref{eq:frozen_est} and \eqref{eq:initial_frozen_bound} imply
\begin{equation}\label{eq:nata_limit}
    \nata_\delta(S)\longrightarrow 0, \quad \text{as }\,S\to \infty,\quad \text{and}\quad \nata_\delta^S(R)\longrightarrow 0\quad \text{as }\,R\to \infty, \text{ with }S \text{ fixed.}
\end{equation}
These bounds do not imply $\nata_\delta^S(S)\to 0$, as $S\to \infty$. They give only
\begin{equation}\label{eq:nata_bound}
    \sup_{S\geq 1}\nata_\delta^S(S)\lesssim_\chi M_V.
\end{equation}
We will use unequal truncation levels in the next subsection to balance these two bounds. Both convergences in \eqref{eq:nata_limit} are only qualitative. Accordingly, he argument below gives qualitative modified scattering in Theorems \ref{thm:Coulomb_defocusing} and \ref{thm:square_defocusing}.

%%%%%%%%%%%%%%%%%%%%%%%%%%%%%%%%%%%%%%%
%%%%%%%%%%%%%%%%%%%%%%%%%%%%%%%%%%%%%%%
%%%%%%%%%%%%%%%%%%%%%%%%%%%%%%%%%%%%%%%

\subsection{Comparison with two truncation levels}\label{subsec:2level}
Equality of mass turns a distance into an inner product. Since we start the frozen evolution at $g_\delta(S)$, the conserved mass for the frozen evolution \eqref{eq:frozen_mass} implies
\begin{equation}\label{eq:comparison_frozen_decom}
\begin{split}
    \big\|g_\delta(t)-g_\delta^S(t)\big\|_2^2&=\left(\|g_\delta(t)\|_2^2-\|g_\delta(S)\|_2^2\right)\\
    &\qquad +2\Re \left(\int g_\delta(S,y)\overline{g_\delta^S(S,y)}dy-\int g_\delta(t,y)\overline{g_\delta^S(t,y)}dy\right).
\end{split}
\end{equation}
The common Laplacian and frozen coefficient cancel, giving
\begin{equation}\label{eq:comparison_frozen_norm_PDE}
    \partial_t \|g_\delta(t)\|_2^2=2\Im \la \mathcal{R}_\delta(t),g_\delta(t)\ra,
\end{equation}
and 
\begin{equation}\label{eq:comparison_frozen_pair_PDE}
    \partial_t \Re \int  g_\delta(t,y)\overline{ g_\delta^S(t,y)}dy=-\frac{\lambda}{t}\int (\phi_g(t,y)-\Phi_\infty(y))\Im(g_\delta\overline{g_\delta^S})(t,y)dy+\Im \la \mathcal R_\delta(t),g_\delta^S(t)\ra.
\end{equation}

The field difference in $L^2$ gives decay, but the product of the two profiles is not uniformly controlled in $L^2$. Amplitude truncation supplies this missing bound, while entropy controls the discarded part. More specifically, we prove the following estimate. 
\begin{proposition}\label{prop:frozen_comparison_convergence}
    Under $\lambda>0$, for $0<\delta<1$, $S\geq 1$, and $1\leq L_1\leq L_2$, we have
    \begin{equation}\label{eq:frozen_comparison_bound}
        \begin{split}
            \sup_{t\geq S}\|g_\delta(t)-g_\delta^S(t)\|_2^2
        &\lesssim_{\chi,w}
        \sqrt{\mass\big(\ent[g(1)]+M_V\big)}
        \left(\frac{1}{\sqrt{\log(1+L_1^2)}}+\frac{1}{\sqrt{\log(1+L_2^2)}}\right)\\
        &\quad+ \frac{L_2}{L_1}\nata_\delta(S)
        +\frac{L_1}{L_2}\nata_\delta^S(S)
        +|\lambda|\mass^{\frac{3}{4}} M_V^{\frac{3}{4}}L_1 S^{-\frac{1}{4}}\\
        &\quad+ \sqrt\mass\int_S^\infty\|\mathcal R_\delta(t)\|_2\,dt.
        \end{split}
    \end{equation}
    In particular, for every fixed $0<\delta<1$, we have
    \begin{equation}\label{eq:frozen_comparison_convergence}
        \sup_{t\geq S}\|g_\delta(t)-g_\delta^S(t)\|_2 \longrightarrow 0
        \qquad \text{as }S\to\infty.
    \end{equation}
\end{proposition}
\begin{proof}
    Lemma \ref{lem:smt} and the mass bounds \eqref{eq:conservation} and \eqref{eq:frozen_mass} imply
    \begin{equation}\label{eq:discarded_truncation}
    \begin{split}
        \left|\la g_\delta(t),g_\delta^S(t)\ra-\la \Tone g_\delta(t), \Ttwo g_\delta^S(t)\ra\right|
        \leq \sqrt\mass \left(
        \sqrt{\frac{\ent[g_\delta(t)]}{\log(1+L_1^2)}}
        +\sqrt{\frac{\ent[g_\delta^S(t)]}{\log(1+L_2^2)}}\right).
    \end{split}
    \end{equation}
    Using \eqref{eq:smt_phase_equivariance}, we have
    \begin{equation*}
        \begin{split}
            \partial_t (\Tone g_\delta)&=-\frac{1}{t^2} D\Tone(g_\delta)[i\Delta g_\delta]
            +\frac{i}{t}\big(\non+\lambda(\phi_g-\Phi_\infty)\big)\Tone g_\delta-D\Tone (g_\delta)[i\mathcal R_\delta],\\
            \partial_t (\Ttwo g_\delta^S)&=-\frac{1}{t^2}D\Ttwo(g_\delta^S)[i\Delta g_\delta^S] +\frac{i}{t}\non \Ttwo g_\delta^S.
        \end{split}
    \end{equation*}
    Thus the terms containing $\non$ cancel when differentiating the real inner product. Integrating the Laplacian terms by parts, we obtain 
    \begin{equation}
        \begin{split}
            &\partial_t \Re \la\Tone g_\delta,\Ttwo g_\delta^S\ra\\
            &=\frac{1}{t^2}\sum_{k=1}^3\Re\int\Big( D\Tone (g_\delta)[i\partial_k g_\delta]\overline{D\Ttwo (g_\delta^S)[\partial_k g_\delta^S]}+\Tone (g_\delta) \overline{D^2\Ttwo (g_\delta^S)[\partial_k g_\delta^S,i\partial_k g_\delta^S]}\Big)dy\\
            &\quad+\frac{1}{t^2}\sum_{k=1}^3 \Re \int\Big(D\Tone (g_\delta)[\partial_k g_\delta]\overline{D\Ttwo(g_\delta^S)[i\partial_k g_\delta^S]}+D^2\Tone (g_\delta)[\partial_k g_\delta,i\partial_k g_\delta] \overline{\Ttwo g_\delta^S}\Big)dy\\
            &\quad -\frac{\lambda}{t}\int(\phi_g-\Phi_\infty) \Im\big(\Tone g_\delta\,\overline{\Ttwo g_\delta^S}\big)dy-\Re\int D\Tone (g_\delta)[i\mathcal R_\delta]\, \overline{\Ttwo g_\delta^S}\,dy
        \end{split}
    \end{equation}
    The truncation bounds \eqref{eq:smt_bound_map}, \eqref{eq:Dsmt_bound} and \eqref{eq:smt_bdd} imply
    \begin{equation}\label{eq:bound_comparison_differentiated}
        \begin{split}
            \left|\partial_t\Re\la \Tone g_\delta,\Ttwo g_\delta^S\ra\right|&\leq  \frac{\lambda}{t}\left|\int(\phi_g-\Phi_\infty) \Im\big(\Tone g_\delta\,\overline{\Ttwo g_\delta^S}\big)dy\right|+\|\mathcal R_\delta\|_2\|g_\delta^S\|_2\\
            &\quad + \frac{1}{t^2}\left( 2\|\nabla g_\delta\|_2\|\nabla g_\delta^S\|_2 +2C_w\frac{L_2}{L_1}\|\nabla g_\delta\|_2^2 +2C_w\frac{L_1}{L_2}\|\nabla g_\delta^S\|_2^2\right).
        \end{split}
    \end{equation}
    Young's inequality bounds
    \begin{equation}
        2\|\nabla g_\delta\|_2\|\nabla g_\delta^S\|_2 \leq\frac{L_2}{L_1}\|\nabla g_\delta\|_2^2 +\frac{L_1}{L_2}\|\nabla g_\delta^S\|_2^2.
    \end{equation}
    The truncated amplitude makes the field error integrable. Integrating \eqref{eq:bound_comparison_differentiated}, Proposition \ref{prop:limiting_field} and \eqref{eq:frozen_mass} give
    \begin{equation}\label{eq:bound_comparison_integrated}
        \begin{split}
            &\big|\Re \la\Tone g_\delta(t),\Ttwo g_\delta^S(t)\ra-\Re \la\Tone g_\delta(S),\Ttwo g_\delta^S(S)\ra\big|\\
        &\lesssim_w\frac{L_2}{L_1}\nata_\delta(S)+\frac{L_1}{L_2}\nata_\delta^S(S) \\
        &\qquad +|\lambda|L_1\int_S^\infty t^{-1}\|\phi_g(t)-\Phi_\infty\|_2\big\| g_\delta^S(t)\big\|_2 dt+\sqrt{\mass}\int_S^\infty \|\mathcal{R}_\delta(t)\|_2 dt\\
        &\lesssim_w \frac{L_2}{L_1}\nata_\delta(S)
        +\frac{L_1}{L_2}\nata_\delta^S(S)
        +|\lambda|\mass^{\frac{3}{4}} M_V^{\frac{3}{4}}L_1 S^{-\frac{1}{4}}+ \sqrt\mass\int_S^\infty\|\mathcal R_\delta(t)\|_2\,dt.
        \end{split}
    \end{equation}
    Finally, integrating \eqref{eq:comparison_frozen_norm_PDE} yields
    \begin{equation}\label{eq:bound_comparison_norm}
        \left|\|g_\delta(t)\|_2^2-\|g_\delta(S)\|_2^2\right| \leq 2\sqrt\mass\int_S^\infty\|\mathcal R_\delta(t)\|_2dt.   
    \end{equation}
    Applying \eqref{eq:comparison_frozen_decom} and adding \eqref{eq:discarded_truncation} at both $t$ and $S$, \eqref{eq:bound_comparison_integrated} and \eqref{eq:bound_comparison_norm} prove \eqref{eq:frozen_comparison_bound}. 

    \medskip

    The two levels have different roles. First take $L_1$ large to control the discarded mass, then $L_2/L_1$ large to suppress the uniformly bounded frozen tail. With these choices fixed, take $S$ large. The gradient tail bounds \eqref{eq:nata_limit}, \eqref{eq:nata_bound} and the $\mathcal R_\delta$ estimates \eqref{eq:cutoff_forcing_tail} and \eqref{eq:external_forcing_tail} together prove \eqref{eq:frozen_comparison_convergence}.     
\end{proof}

%%%%%%%%%%%%%%%%%%%%%%%%%%%%%%%%%%%%%%%
%%%%%%%%%%%%%%%%%%%%%%%%%%%%%%%%%%%%%%%
%%%%%%%%%%%%%%%%%%%%%%%%%%%%%%%%%%%%%%%

\subsection{Convergence of the frozen profiles}\label{subsec:convergence_frozen}
We now fix the starting time $S$ for the frozen evolution and prove that $h_\delta^S(t)$ is a Cauchy sequence. 

\begin{proposition}\label{prop:frozen_profile_Cauchy}
    Let $\lambda>0$ and let $\non$ be as in \eqref{eq:frozen_coefficient}, so that $\non\in \dot H^1$ and nonnegative. For fixed $0<\delta<1$ and $S\geq 1$, the corrected profile $h_\delta^S$ in \eqref{eq:h_cutoff} converges strongly in $L^2$. More precisely, for $R\geq S$ and $L_2\geq 1$, \begin{equation}\label{eq:frozen_profile_cauchy}
        \begin{split}
            \sup_{t\geq R} &\big\|h_\delta^S(t)-h_\delta^S(R)\big\|_2^2\\
            &\quad \leq \frac{4\left(\ent[g_\delta(S)]+4\kfun_\delta^S(S)\right)}{\log(1+L_2^2)}+4\nata_\delta^S(R)^\frac{1}{2}\left(\kfun_\delta^S(S)^\frac{1}{2}+\frac{2\sqrt{2}L_2\|\nabla \non\|_2}{\sqrt{R}}\right)\\
            &\quad \lesssim_\chi  \frac{\ent[g(1)]+M_V}{\log(1+L_2^2)}+\nata_\delta^S(R)^\frac{1}{2}\left(\sqrt{M_V}+\frac{L_2\|\nabla \non\|_2}{\sqrt{R}}\right).
        \end{split}
    \end{equation}
    Here the third line follows from the second line using \eqref{eq:defocusing_ent_bound} and \eqref{eq:initial_frozen_bound}. The constant here depends only on $\chi$ and the dimension, and is independent of $\delta, L_1, L_2, S$ and $R$. 
\end{proposition}
\begin{proof}
    Mass conservation gives 
    \begin{equation}\label{eq:cauchy_rewrite}
        \big\|h_\delta^S(t)-h_\delta^S(R)\big\|_2^2=2\Re \left\la h_\delta^S(R)-h_\delta^S(t),h_\delta^S(R)\right\ra.
    \end{equation}
    By phase equivariance \eqref{eq:smt_phase_equivariance}, truncation tail bound \eqref{eq:smt_tail} and entropy bound \eqref{eq:ent_frozen_bound}, 
    \begin{equation}\label{eq:smt_h_cutoff_bound}
       \sup_{t\geq R} \big\|h_\delta^S(t)-\Ttwo h_\delta^S(t)\big\|_2^2=\sup_{t\geq R} \big\|g_\delta^S(t)-\Ttwo g_\delta^S(t)\big\|_2^2\leq \frac{\ent[g_\delta(S)]+4\kfun_\delta^S(S)}{\log(1+L_2^2)};
    \end{equation}
    i.e. the truncation estimate is applied to $g_\delta^S(t)\in H^1$ and no $H^1$ regularity of $h_\delta^S$ is needed or necessarily holds. Hence, expression \eqref{eq:cauchy_rewrite}, \eqref{eq:smt_h_cutoff_bound} and Young's inequality give
    \begin{equation}\label{eq:cauchy_truncated_pair}
        \begin{split}
            \big\|h_\delta^S(t)-h_\delta^S(R)\big\|_2^2\leq 4\big|\Re \left \la h_\delta^S(R)-h_\delta^S(t), \Ttwo h_\delta^S(R)\right\ra \big|+\frac{4\left(\ent[g_\delta(S)]+4\kfun_\delta^S(S)\right)}{\log(1+L_2^2)}.
        \end{split}
    \end{equation}

    Differentiating the inner product in \eqref{eq:cauchy_truncated_pair}, phase equivariance gives
    \begin{equation*}
        \begin{split}
            \partial_t \Re \left \la h_\delta^S(t)-h_\delta^S(R), \Ttwo h_\delta^S(R)\right\ra&=\partial_t \Re \left \la g_\delta^S(t), e^{i\non \log(t/R)}\Ttwo g_\delta^S(R) \right\ra. 
        \end{split}
    \end{equation*}
    Differentiating the real inner product, the time derivative of the phase cancels the frozen multiplication term. Integration by parts then gives 
    \begin{equation}\label{eq:cauchy_rewrite_pair}
        \partial_t \Re \left \la h_\delta^S(t)-h_\delta^S(R), \Ttwo h_\delta^S(R)\right\ra=-t^{-2}\Im\int\nabla g_\delta^S(t)\cdot \overline{\nabla\left(e^{i\non\log(t/R)}\Ttwo g_\delta^S(R)\right)}dy.
    \end{equation}
    This is valid because the truncation gives $\Ttwo g_\delta^S\in H^1\cap L^\infty$: 
    \begin{equation}\label{eq:cauchy_test_bound}
        \|\nabla\left(e^{i\non\log(t/R)}\Ttwo g_\delta^S(R)\right)\|_2 \leq \|\nabla g_\delta^S(R)\|_2+2L_2\log(t/R)\|\nabla\non\|_2.
    \end{equation}
    Hence \eqref{eq:cauchy_rewrite_pair} is justified in integrated form for the $\non$ regularity in $\dot H^1$. The estimate \eqref{eq:cauchy_test_bound} also gives, for $t\geq R$,
    \begin{equation*}
        \big|\partial_t \Re \left \la h_\delta^S(t)-h_\delta^S(R), \Ttwo h_\delta^S(R)\right\ra\big|\leq \frac{\big\|\nabla g_\delta^S(t)\|_2}{t^2}\left(\|\nabla g_\delta^S(R)\|_2+2L_2\log\left(t/R\right)\|\nabla\non\|_2\right).
    \end{equation*}
    Integrating this, H\"older's inequality implies
    \begin{equation}\label{eq:cauchy_pair_est}
        \begin{split}
            \big|\Re \left \la h_\delta^S(t)-h_\delta^S(R), \Ttwo h_\delta^S(R)\right\ra\big|\leq \nata_\delta^S(R)^\frac{1}{2}\frac{\big\|\nabla g_\delta^S(R)\big\|_2}{\sqrt{R}} +\nata_\delta^S(R)^\frac{1}{2}\, \frac{2\sqrt{2}L_2\|\nabla\non\|_2}{\sqrt{R}},
        \end{split}
    \end{equation}
    since
    \[
    \int_R^\infty t^{-2}\log^2(t/R) dt=\frac{1}{R}\int_0^\infty e^{-s}s^2 ds=\frac{2}{R}.
    \]
    \eqref{eq:cauchy_truncated_pair}, \eqref{eq:cauchy_pair_est} and \eqref{eq:frozen_est} prove \eqref{eq:frozen_profile_cauchy}. With $\delta$ and $S$ fixed, first choose $L_2$ large to make the first term in \eqref{eq:frozen_profile_cauchy} small. Then let $R\to\infty$. Using \eqref{eq:nata_limit}, this proves the Cauchy property and strong $L^2$ convergence of the modified profile. 
\end{proof}

\medskip

We are now ready to prove the two defocusing theorems. 
\begin{proof}[Proof of Theorems \ref{thm:Coulomb_defocusing} and \ref{thm:square_defocusing}]
We prove the two theorems together, since the comparison and convergence estimates above apply to both problems. Global well-posedness follows from Proposition \ref{prop:gwp}. The bounds \eqref{eq:defocusing_AB} and Proposition \ref{prop:limiting_field} give a nonnegative limiting field $\Phi_\infty\in\dot H^1$.

We first control the error from the spatial cutoff, uniformly in time. Positivity of the Coulomb kernel gives
\[
    \dob[g(t)]\geq\iint_{|y|,|z|\leq2\delta}
    \frac{|g(t,y)|^2|g(t,z)|^2}{|y-z|}\,dy\,dz
    \geq\frac1{4\delta}\left(\int_{|y|\leq 2\delta}|g(t,y)|^2\,dy\right)^2.
\]
Since $0\leq1-\chi_\delta\leq1$ is supported in $|y|\leq2\delta$, \eqref{eq:nabla_g_grow} implies
\begin{equation}\label{eq:remove_common_cutoff}
    \sup_{t\geq1}\|h(t)-h_\delta(t)\|_2
    =\sup_{t\geq1}\|(1-\chi_\delta)g(t)\|_2
    \leq\left(\frac{8\delta M_V}{\lambda}\right)^{\frac{1}{4}}.
\end{equation}
Given $\varepsilon>0$, first choose $0<\delta<1$ so that the right-hand side is less than $\varepsilon$, and keep $\delta$ fixed. The coefficient $\non$ in \eqref{eq:frozen_coefficient} is then nonnegative and belongs to $\dot H^1$, with a norm that may depend on $\delta$.

\medskip

Next, choose the comparison levels $L_1$ and $L_2/L_1$ large\footnote{By \eqref{eq:defocusing_ent_bound} and \eqref{eq:initial_frozen_bound}, the numerator of the first term in \eqref{eq:frozen_profile_cauchy} is bounded uniformly in $S$. We therefore take $L_2$ large enough to make this term less than $\varepsilon^2/2$ before choosing $S$.}, and then $S$ large, as in Proposition \ref{prop:frozen_comparison_convergence}, so that
\[
    \sup_{t\geq S}\|g_\delta(t)-g_\delta^S(t)\|_2<\varepsilon.
\]

With $\delta$ and $S$ fixed and $L_2$ already chosen sufficiently large, Proposition \ref{prop:frozen_profile_Cauchy} allows us to choose $R\geq S$ so that
\[
    \sup_{t\geq R}\|h_\delta^S(t)-h_\delta^S(R)\|_2<\varepsilon.
\]

Since multiplication by the phase preserves the $L^2$ norm, for $t,s\geq R$,
\begin{equation}\label{eq:transfer_frozen_convergence}
    \begin{split}
        \|h(t)-h(s)\|_2
        &\leq 2\sup_{\tau\geq1}\|h(\tau)-h_\delta(\tau)\|_2
        + 2\sup_{\tau\geq S}\|g_\delta(\tau)-g_\delta^S(\tau)\|_2\\
        &\quad+\|h_\delta^S(t)-h_\delta^S(R)\|_2
        +\|h_\delta^S(s)-h_\delta^S(R)\|_2
        <6\varepsilon.
    \end{split}
\end{equation}
Thus $h(t)$ is Cauchy in $L^2$. Denote its limit by $U_\infty$.

We next identify the limiting field. Since $|h(t)|=|g(t)|$, mass conservation and strong convergence give
\begin{equation}\label{eq:limit_density_identification}
    \|U_\infty\|_2^2=\mass,\qquad
    \big\||g(t)|^2-|U_\infty|^2\big\|_1
    \leq2\sqrt\mass\|h(t)-U_\infty\|_2\longrightarrow0.
\end{equation}
For $\zeta\in C_c^\infty(\R^3)$, the function $|\cdot|^{-1}*|\zeta|$ is bounded. Fubini's theorem and \eqref{eq:limit_density_identification} therefore yield
\[
    \int\phi_g(t)\zeta
    =\int|g(t)|^2\big(|\cdot|^{-1}*\zeta\big)
    \longrightarrow\int|U_\infty|^2\big(|\cdot|^{-1}*\zeta\big)
    =\int\phi_{U_\infty}\zeta.
\]
All pairings are absolutely convergent. On the other hand, \eqref{eq:limiting_field} gives $\phi_g(t)\to\Phi_\infty$ in distributions. Hence
\begin{equation}\label{eq:identified_limiting_field}
    \Phi_\infty=\phi_{U_\infty}\in\dot H^1.
\end{equation}

Finally, the definitions \eqref{eq:g} and \eqref{eq:h} give the modified scattering
\[
    \left\|u(t,x)-t^{-3/2}e^{-i|x|^2/(4t)}
    e^{i\non_\infty(x/t)\log t}U_\infty(x/t)\right\|_{L_x^2}
    =\|h(t)-U_\infty\|_2\longrightarrow0.
\]
Substituting the two expressions for $\non_\infty$ in \eqref{eq:h} proves \eqref{eq:modified_scattering_coulomb} and \eqref{eq:modified_scattering_general}.
\end{proof}

%%%%%%%%%%%%%%%%%%%%%%%%%%%%%%%%%%%%%%%
%%%%%%%%%%%%%%%%%%%%%%%%%%%%%%%%%%%%%%%
%%%%%%%%%%%%%%%%%%%%%%%%%%%%%%%%%%%%%%%

\subsection{Frozen evolution without a sign condition}\label{subsec:frozen_focusing}
This subsection prepares the focusing arguments in Sections \ref{sec:coulomb} and \ref{sec:general}. When $\lambda<0$, the frozen coefficient $\non$ need not be nonnegative, so the energy identity \eqref{eq:K_cutoff_prime} alone does not give \eqref{eq:frozen_est}. 

We instead impose additional control on $\nabla\non$ and derive quantitative convergence estimates without any sign condition on $\non$, while retaining the definitions \eqref{eq:frozen_PDE} and \eqref{eq:h_cutoff}. This additional bound on $\nabla \non$ controls the growth of $\|\nabla g_\delta^S(t)\|_2$ and allows us to control the change in $h_\delta^S(t)$ without amplitude truncation. Sections \ref{sec:coulomb} and \ref{sec:general} will verify these hypotheses for the radial Coulomb problem and for small focusing data with general potentials, respectively.

\begin{proposition}\label{prop:focusing_cauchy_frozen}
    Let $\non\in \dot H^1$ be real. Suppose that, for every $f\in H^1$, $\nabla \non$ satisfies
    \begin{equation}\label{eq:focusing_cauchy_assumption}
        \big\|\left(\nabla \non\right)f\big\|_2\leq a_*\|\nabla f\|_2+b_*\|f\|_2,\qquad 0\leq a_*<1/2,\quad b_*\geq 0.
    \end{equation}
    The frozen solution to \eqref{eq:frozen_PDE} then satisfies, for $t\geq R\geq S$, 
    \begin{equation}\label{eq:focusing_nabla_frozen_bound}
        \big\|\nabla g_\delta^S(t)\big\|_2\leq \left(\frac{t}{R}\right)^{a_*} \left(\big\|\nabla g_\delta^S(R)\big\|_2+b_*\sqrt{\mass}\,\log(t/R)\right),
    \end{equation}
    and 
    \begin{equation}\label{eq:focusing_frozen_cauchy_est}
        \sup_{t\geq R} \big\|h_\delta^S(t)-h_\delta^S(R)\big\|_2^2\leq \frac{C_{a_*}}{R} \left(\big\|\nabla g_\delta^S(R)\big\|_2+b_*\sqrt{\mass}\right)^2.
    \end{equation}
    Hence, $h_\delta^S(t)$ converges in $L^2$ for each fixed $S$. Moreover, if $b_*=0$, we have
    \begin{equation}\label{eq:nata_bound_frozen}
        \nata_\delta^S(R)\leq \frac{\big\|\nabla g_\delta^S(R)\big\|_2^2}{(1-2a_*)R}.
    \end{equation}
\end{proposition}
\begin{proof}
    Using the frozen equation \eqref{eq:frozen_PDE} and applying integration by parts, we obtain
    \begin{align*}
        \partial_t\big\|\nabla g_\delta^S(t)\big\|_2^2&=\frac{2}{t}\Re\int i\left(\non \nabla g_\delta^S+g_\delta^S\nabla\non\right)\cdot\overline{\nabla g_\delta^S}dy= -\frac{2}{t}\Im\int g_\delta^S\nabla\non \cdot\overline{\nabla g_\delta^S}dy.
    \end{align*}
    H\"older's inequaltiy, \eqref{eq:focusing_cauchy_assumption} and \eqref{eq:frozen_mass} give
    \begin{align*}
        \left| \partial_t\big\|\nabla g_\delta^S(t)\big\|_2^2\right|\leq 2t^{-1} \big\|\nabla g_\delta^S(t)\big\|_2\left(a_*\big\|\nabla g_\delta^S(t)\big\|_2+b_*\sqrt{\mass}\right).
    \end{align*}
    Applying Gronwall's inequality, we get
    \begin{align*}
        \big\|\nabla g_\delta^S(t)\big\|_2 \leq \left(\frac {t}{R}\right)^{a_*}\big\|\nabla g_\delta^S(R)\big\|_2 +b_*\sqrt{\mass}\int_R^t\left(\frac {t}{\tau}\right)^{a_*}\frac{d\tau}{\tau},
    \end{align*}
    which gives \eqref{eq:focusing_nabla_frozen_bound}. When $b_*=0$, integrating the square of bound \eqref{eq:focusing_nabla_frozen_bound} proves \eqref{eq:nata_bound_frozen}.

    The multiplier assumption \eqref{eq:focusing_cauchy_assumption} ensures that the untruncated relative-phase test belongs to $H^1$. Its time derivative lies in $L^2$ because $\non\in L^6$ and the fixed profile belongs to $L^3$. With the untruncated expression \eqref{eq:cauchy_rewrite}, the same computation as in \eqref{eq:cauchy_rewrite_pair} gives
    \begin{equation}
        \begin{split}
            \partial_t \big\|h_\delta^S(t)-h_\delta^S(R)\big\|_2^2 &=2\partial_t \Re \left\la h_\delta^S(R)-h_\delta^S(t), h_\delta^S(R)\right\ra \\
            &=2t^{-2} \Im \int \nabla g_\delta^S(t)\cdot \overline{\nabla \left(e^{i\non\log(t/R)}g_\delta^S(R)\right)}dy.
        \end{split}
    \end{equation}
    H\"older's inequality,  \eqref{eq:focusing_cauchy_assumption} and \eqref{eq:frozen_mass} bound 
    \begin{align*}
        \left|\partial_t \big\|h_\delta^S(t)-h_\delta^S(R)\big\|_2^2\right|\leq 2t^{-2} \big\|\nabla g_\delta^S(t)\big\|_2\left( \big\|\nabla g_\delta^S(R)\big\|_2+\log(t/R) \left(a_* \big\|\nabla g_\delta^S(R)\big\|_2+b_*\sqrt{\mass}\right)\right).
    \end{align*}
    Integrating this and using \eqref{eq:focusing_nabla_frozen_bound}, together with $a_*<1/2$, gives
    \begin{align*}
        \big\|h_\delta^S(t)-h_\delta^S(R)\big\|_2^2 &\leq 2\left(\big\|\nabla g_\delta^S(R) \big\|_2 + b_*\sqrt{\mass}\right)^2 \int_R^t \tau^{-2}\left(\frac{\tau}{R}\right)^{a_*} \left(1+\log\frac{\tau}{R}\right)^2 d\tau\\
        &\leq \frac{2}{R} \left(\big\|\nabla g_\delta^S(R)\big\|_2+b_*\sqrt{\mass}\right)^2 \int_1^\infty z^{a_*-2}(1+\log z)^2 dz.
    \end{align*}
    Since
    \[
    \int_1^\infty z^{a_*-2}(1+\log z)^2\,dz
    =\frac{1}{1-a_*}+\frac{2}{(1-a_*)^2}+\frac{2}{(1-a_*)^3}<\infty,
    \]
    we prove \eqref{eq:focusing_frozen_cauchy_est}. For fixed $\delta$, $S$,  \eqref{eq:focusing_nabla_frozen_bound} gives 
    \[
    \big\|\nabla g_\delta^S(R)\big\|_2\lesssim_S R^{a_*}(1+\log R).
    \]
    Since $a_*<1/2$, \eqref{eq:focusing_frozen_cauchy_est} tends to 0 and proves $L^2$ convergence of $h_\delta^S(t)$ for fixed $S\geq 1$.
\end{proof}

%%%%%%%%%%%%%%%%%%%%%%%%%%%%%%%%%%%%%%%
%%%%%%%%%%%%%%%%%%%%%%%%%%%%%%%%%%%%%%%
%%%%%%%%%%%%%%%%%%%%%%%%%%%%%%%%%%%%%%%
%%%%%%%%%%%%%%%%%%%%%%%%%%%%%%%%%%%%%%%
%%%%%%%%%%%%%%%%%%%%%%%%%%%%%%%%%%%%%%%
%%%%%%%%%%%%%%%%%%%%%%%%%%%%%%%%%%%%%%%

\section{The radial focusing Coulomb problem}\label{sec:coulomb} 
We now prove Theorem \ref{thm:coulomb_focusing}. Throughout this section, we consider the repulsive Coulomb potential $V(x)=\kappa|x|^{-1}$ with $\kappa>0$, $\lambda<0$, and radial initial data $u_0\in\Sigma$. Proposition \ref{prop:gwp} provides a unique global solution $u(t)$, which remains radial by the rotation invariance of \eqref{main} and uniqueness.

Below the mass threshold $|\lambda|\mass<2\kappa$, Newton's formula makes the coercivity of the rescaled energy explicit, yielding a uniform weighted bound. At the endpoint, we recover this bound using the convexity and eventual linear growth of the first radial moment. In both cases, Newton's formula yields uniform convergence of the field away from the origin, allowing a common \emph{quantitative} comparison with the frozen evolution \eqref{eq:frozen_PDE}. We finally construct standing waves at every mass above the threshold.

%%%%%%%%%%%%%%%%%%%%%%%%%%%%%%%%%%%%%%%
%%%%%%%%%%%%%%%%%%%%%%%%%%%%%%%%%%%%%%%
%%%%%%%%%%%%%%%%%%%%%%%%%%%%%%%%%%%%%%%

\subsection{Radial reduction and the mass threshold}\label{subsec:newton}
Since $u$ is radial, we write $u(t,x)=u(t,r)$, where $r=|x|$. The equation \eqref{main} with $V(r)=\kappa /r$ becomes 
\begin{equation}\label{eq:radial_coulomb}
    i\partial_t u-\partial_{rr}u-\frac{2}{r}\partial_r u+\frac{\kappa}{r}u+\lambda \phi_u u=0, \qquad r>0.
\end{equation}
Newton's formula expresses the Coulomb field as
\begin{equation}\label{eq:Newton}
    \phi_u(t,r)=4\pi \int_0^\infty \frac{s^2|u(t,s)|^2}{\max\{r,s\}}\,ds. 
\end{equation}
In particular, radial symmetry gives pointwise control of the field and its derivative using only the mass:
\begin{equation}\label{eq:radial_phi_bound}
    0\leq \phi_u(t,r)\leq \frac{\mass}{r}, \quad 0\leq -\partial_r \phi_u(t,r)=\frac{4\pi}{r^2}\int_0^r s^2|u(t,s)|^2 ds\leq \frac{\mass}{r^2}.
\end{equation}

To remove the first-order radial derivative in \eqref{eq:radial_coulomb}, we consider instead
\begin{equation}\label{eq:v}
    v(t,r):=\sqrt{4\pi}\, ru(t,r).
\end{equation}
Then $v$ satisfies the half-line equation
\begin{equation}\label{eq:v_PDE}
    i\partial_t v-\partial_{rr}v+\frac{\kappa}{r}v+\lambda \varphi_v v=0,\qquad v(t,0)=0,
\end{equation}
where
\begin{equation}
    \varphi_v(t,r)=\int_0^\infty \frac{|v(t,s)|^2}{\max\{r,s\}}ds.
\end{equation}
Here $v(t)\in H_0^1(\dom)$, and the normalization gives
\begin{equation}\label{eq:radial_scale_mass}
    \|v(t)\|_{L^2(\dom)}^2=\|u(t)\|_{L^2(\R^3)}^2=\mass,
\end{equation}
\begin{equation}\label{eq:radial_scale_nabla}
    \|\partial_r v(t)\|_{L^2(\dom)}=\|\nabla u(t)\|_{L^2(\R^3)}, \quad \|rv(t)\|_{L^2(\R_+)}=\|xu(t)\|_{L^2(\R^3)},
\end{equation}
and 
\begin{equation}\label{eq:radial_energy}
    \mathbf{E}=\frac{1}{2}\int_0^\infty |\partial_rv(t,r)|^2 dr+\frac{\kappa}{2}\int_0^\infty \frac{|v(t,r)|^2}{r}dr+\frac{\lambda}{4}\int_0^\infty\int_0^\infty \frac{|v(t,r)|^2|v(t,s)|^2}{\max\{r,s\}}drds.
\end{equation}
The potential terms are finite by the half-line Hardy inequality \eqref{eq:Hardy_half_line}, $v\in H_0^1(\dom)$ and the uniform bound $\varphi_v(t,r)\leq \mass/r$. 

\medskip

We set the corresponding half-line representative of the radial profile $g$ defined in \eqref{eq:g} as
\begin{equation}\label{eq:g_radial}
    \g(t,a):=\sqrt{4\pi}\,ag(t,a) =t^{\frac{1}{2}} e^{i\frac{ta^2}{4}} v(t,ta), \qquad t\geq 1.
\end{equation}
The profile $\g$ satisfies
\begin{equation}\label{eq:g_PDE_radial}
    i\partial_t\g -t^{-2}\partial_{aa}\g=-\frac{1}{t}\left(\frac{\kappa}{a}+\lambda\varphi_\g\right)\g,\qquad \g(t,0)=0,
\end{equation}
with
\begin{equation}
    \varphi_\g(t,a)=\int_0^\infty \frac{|\g(t,b)|^2}{\max\{a,b\}}db.
\end{equation}
The same norm identities give
\begin{equation}\label{eq:radial_profile_scale}
    \|\g(t)\|_{L^2(\R_+)}^2=\mass, \quad \|\partial_a \g(t)\|_{L^2(\dom)}=\|\nabla g(t)\|_{L^2(\R^3)}, \quad\, \|a\g(t)\|_{L^2(\dom)}=\|yg(t)\|_{L^2(\R^3)}.
\end{equation}
Thus, the rescaled energy in \eqref{eq:K} takes the form 
\begin{equation}\label{eq:K_radial}
    \kfun(t)=\frac{\|\partial_a \g(t)\|_{L^2(\dom)}^2}{t}+\kappa\int_0^\infty \frac{|\g(t,a)|^2}{a}da+\frac{\lambda}{2}\mathfrak{D}[\g(t)].
\end{equation}
where $\mathfrak{D}$ is the corresponding Hartree interaction functions
\begin{equation}
    \mathfrak{D}[f]:=\int_0^\infty \varphi_f(a) |f(a)|^2 da, \qquad \varphi_f(a):=\int_0^\infty\frac{|f(b)|^2}{\max\{a,b\}}db.
\end{equation}
Newton's formula bounds these Hartree interaction terms by
\begin{equation}\label{eq:half_line_coulomb_est}
    \varphi_f(a)\leq \frac{\|f\|_{L^2(\dom)}^2}{a}, \qquad \mathfrak{D}[f]\leq \|f\|_{L^2(\dom)}^2 \int_0^\infty\frac{|f(a)|^2}{a}da.
\end{equation}
To treat the threshold case as well, we record the following identity. Throughout the remainder of this section, we write $\|\cdot\|_p$ for the $L^p(\dom)$ norm, while retaining the explicit notation $\|\cdot\|_{L^p(\R^3)}$ for norms on the full space. Also, for half-line functions, primes denote differentiation with respect to the spatial variable.
\begin{lemma}\label{lem:radial_energy_difference}
    For every $f\in H_0^1(\dom)$, we have
    \begin{equation}\label{eq:radial_coulomb_identity}
        \|f\|_2^2\int_0^\infty \frac{|f(a)|^2}{a}da-\mathfrak{D}[f]=\iint_{0<a<b<\infty}\left(\frac{1}{a}-\frac{1}{b}\right)|f(a)|^2|f(b)|^2 dadb\geq 0.
    \end{equation}
    The inequality is strict whenever $f\neq 0$ in $H_0^1(\dom)$.
\end{lemma}
\begin{proof}
    H\"older's inequality and Hardy's inequality \eqref{eq:Hardy_half_line} give
    \begin{equation*}
        \int_0^\infty\frac{|f(a)|^2}{a}da\leq 2\|f\|_2\|\partial_a f\|_2
    \end{equation*}
    \eqref{eq:half_line_coulomb_est} proves that both sides are finite. The identity then follows by writing the difference as
    \[
    \int_0^\infty\int_0^\infty\left(\frac{1}{a}-\frac{1}{\max\{a,b\}}\right)|f(a)|^2|f(b)|^2 dadb.
    \]
    Finally, symmetrization in $a,b$ gives
    \[
    \iint_{0<a<b<\infty} |f(a)|^2|f(b)|^2dadb=\frac{1}{2}\|f\|_2^4.
    \]
    For $f\neq 0$ in $H_0^1(\dom)$, this is positive. Since $1/a-1/b$ is strictly positive on $a<b$. This proves the strict inequality.
\end{proof}
By \eqref{eq:radial_coulomb_identity}, we then rewrite \eqref{eq:K_radial} as
\begin{equation}\label{eq:K_radial_rewrite}
    \begin{split}
        \kfun(t)&=\frac{\|\g'(t)\|_2^2}{t}+\left(\kappa-\frac{|\lambda|\mass}{2}\right)\int_0^\infty\frac{|\g(t,a)|^2}{a}da\\
        &\quad\,\,\, +\frac{|\lambda|}{2}\iint_{0<a<b<\infty}\left(\frac{1}{a}-\frac{1}{b}\right) |\g(t,a)|^2|\g(t,b)|^2dadb.
    \end{split}
\end{equation}
Each term on the right is nonnegative whenever $|\lambda|\mass\leq 2\kappa$, which is the mass threshold in Theorem \ref{thm:coulomb_focusing}. Assume that $|\lambda|\mass\leq 2\kappa$. Since the Coulomb potential satisfies $V+x\cdot \nabla V=0$, \eqref{eq:Kprime} yields $\kfun'(t)\leq 0$ for $t\geq 1$ and also
\begin{equation}\label{eq:k_radial_positivity}
    \kfun(t)+\int_1^t \frac{\|\g'(s)\|_2^2}{s^2}ds= \kfun(1). 
\end{equation}
Consequently, for $t\geq 1$, we have
\begin{equation}\label{eq:radial_k_bound}
    0\leq \kfun(t)\leq \kfun(1), \quad \|\g'(t)\|_2^2\leq \kfun(1)t, \quad \int_1^\infty \frac{\|\g'(t)\|_2^2}{t^2}dt\leq \kfun(1).
\end{equation}
When mass is strictly below the threshold, we additionally have
\begin{equation}\label{eq:radial_coulomb_energy}
    \int_0^\infty \frac{|\g(t,a)|^2}{a}da\leq \frac{2\kfun(1)}{2\kappa-|\lambda|\mass},\quad \text{for }|\lambda|\mass<2\kappa.
\end{equation}
At $|\lambda|\mass=2\kappa$, the coefficient of the weighted integral vanishes and we will use the remaining positive double integral to recover a uniform bound for the weighted integral in \eqref{eq:radial_coulomb_energy} in the next subsection. One important gradient is the moment estimate \eqref{eq:moment_bound} from Section \ref{sec:est}. Using $\|a\g(t)\|_2=t^{-1}\|xu(t)\|_{L^2(\R^3)}$, \eqref{eq:moment_bound} yields
\begin{equation}\label{eq:radial_moment_profile}
    \sup_{t\geq 1} \|a\g(t)\|_2\leq \|xu_0\|_{L^2(\R^3)}+2\sqrt{4\mathbf{E}+|\lambda|^2 \mass^3}=:N_V.
\end{equation}

%%%%%%%%%%%%%%%%%%%%%%%%%%%%%%%%%%%%%%%
%%%%%%%%%%%%%%%%%%%%%%%%%%%%%%%%%%%%%%%
%%%%%%%%%%%%%%%%%%%%%%%%%%%%%%%%%%%%%%%

\subsection{The endpoint of the mass threshold}\label{subsec:endpoint}
We handle the endpoint $|\lambda|\mass=2\kappa$ and prove the needed estimates in this subsection. At the endpoint, we recover the uniform bound on $\int a^{-1}|\g(t,a)|^2da$ by showing that at least some fixed positive amount of mass must remain at some radius level.
\begin{lemma}\label{lem:recover_endpt}
    Consider $u_0\neq 0$ with $|\lambda|\mass=2\kappa$. Set
    \begin{equation}\label{eq:J}
        \mathcal{J}(t):=\int_0^\infty r|v(t,r)|^2 dr.
    \end{equation}
    $\mathcal J(t)$ is convex on $[0,\infty)$ and tends to infinity as $t\to \infty$. Thus, there exist constants $c_0>0$, $T_0\geq 1$, depending on $u_0,\lambda,\kappa$, such that $\mathcal{J}(t)\geq c_0t$ for $t\geq T_0$. This then gives
    \begin{equation}\label{eq:radial_recovered_coulomb}
        \sup_{t\geq 1} \int_0^\infty \frac{|\g(t,a)|^2}{a}da\leq \max\left\{2\sqrt{\mass\kfun(1)T_0},\,\, \frac{16\kfun(1)N_V^2}{|\lambda|c_0^2}+\frac{4\mass^2}{c_0} \right\}.
    \end{equation}
\end{lemma}
\begin{proof}
    We first prove convexity. Define the local mass as
    \[
    m(t,r):=\int_0^r |v(t,s)|^2 ds.
    \]
    Formally, differentiating twice and integrating by parts, together with the Dirichlet condition $v(t,0)=0$, give\footnote{To justify convexity without assuming a boundary trace of $v'$, one needs to set $F_\varepsilon(r):=\sqrt{r^2+\varepsilon^2}$ and work with $\mathcal J_\varepsilon(t):=\int_0^\infty F_\varepsilon(r)|v(t,r)|^2dr$. $F_\varepsilon'(0)=0$ and the Dirichlet condition $v(t,0)=0$ eliminate all boundary contributions at the origin. For fixed $\varepsilon>0$, truncate the weight at infinity and approximate $u$ in $C([-T,T];\Sigma)$. Hardy's inequality permits passage to the limit, so the identity below holds in distributions in time. Integration by parts and completing the square give 
    \[
    \mathcal{J}_\varepsilon''(t)=4\int_0^\infty F_\varepsilon'' \left|v'-\frac{v}{r}\right|^2dr+15\int_0^\infty \frac{\varepsilon^4|v|^2}{(r^2+\varepsilon^2)^{7/2}} dr-|\lambda|\int_0^\infty \left(\frac{F_\varepsilon'}{r^2}\right)' m(\mass-m)dr,
    \]
    at the endpoint. Hardy’s inequality ensures the finiteness of these integrals. Since $F_\varepsilon''\geq 0$ and $(F_\varepsilon'/r^2)'<0$, $\mathcal{J}_\varepsilon''\geq 0$. Finally, $0\leq \mathcal{J}_\varepsilon-\mathcal{J}\leq \varepsilon\mass$ passes convexity to $\mathcal{J}$.
    }
    \begin{equation*}
        \mathcal{J}''(t)=2|v'(t,0)|^2+2\int_0^\infty \frac{\kappa-|\lambda|m(t,r)}{r^2}\partial_r m(t,r)dr.
    \end{equation*}
    At the endpoint $2\kappa=|\lambda|\mass$, another integration by parts yields
    \begin{equation*}
        \mathcal{J}''(t)=2|v'(t,0)|^2 +2 |\lambda|\int_0^\infty \frac{m(t,r)\big(\mass-m(t,r)\big)}{r^3}dr\geq 0.
    \end{equation*}
    Hence, $\mathcal{J}(t)$ is convex.

    \medskip
    
    We next show that $\mathcal{J}(t)\to \infty$. At the endpoint, \eqref{eq:K_radial_rewrite}, \eqref{eq:g_radial}, \eqref{eq:k_radial_positivity} give
    \begin{equation}\label{eq:radial_decay_double}
        0\leq \iint_{0<r<s<\infty}  \left(\frac{1}{r}-\frac{1}{s}\right)|v(t,r)|^2|v(t,s)|^2 drds\leq \frac{2\kfun(1)}{|\lambda|t}.
    \end{equation}
    Suppose that $\mathcal J(t_n)$ is bounded along some sequence $t_n\to\infty$. Then 
    \begin{equation}\label{eq:radial_mass_infinity_tail}
    \sup_n \int_R^\infty |v(t_n,r)|^2 dr\leq \frac{\sup_n \mathcal{J}(t_n)}{R}\longrightarrow 0,\quad \text{as }\,R\to \infty.
    \end{equation}
    Since $v(t_n)$ is uniformly bounded in $H_0^1(\dom)$ from \eqref{eq:nabla_u_bound} and scaling \eqref{eq:radial_scale_nabla}, up to passing to subsequences, $v(t_n)$ converges weakly to some $q\in H_0^1(\dom)$. Moreover, Rellich's theorem gives strong $L^2$ convergence on each bounded interval $(0,R)$. The diagonal argument and the uniqueness of limit then prove, up to passing to subsequences, 
    \[
    \int_0^R |v(t_n,r)-q(r)|^2dr \longrightarrow 0,\qquad \text{for every fixed }R.
    \]
    Passing to a further subsequence gives $v(t_n,r)\to q(r)$ almost everywhere. Fatou's lemma and \eqref{eq:radial_mass_infinity_tail} then give
    \[
    \int_R^\infty |q(r)|^2dr \leq\frac{\sup_n \mathcal J(t_n)}{R}.
    \]
    Splitting at $R$ and using $|v-q|^2\leq2|v|^2+2|q|^2$ on the tail, we obtain
    \begin{equation*}
        \limsup_{n\to\infty} \|v(t_n)-q\|_2^2
        \leq \lim_{n\to\infty} \int_0^R|v(t_n,r)-q(r)|^2 dr
        + \frac{4\sup_n \mathcal J(t_n)}{R}
        =\frac{4\sup_n \mathcal J(t_n)}{R}.
    \end{equation*}
    Sending $R\to\infty$ proves convergence on the whole half-line. Hence, $v(t_n)$ converges to $q$ strongly in $L^2$ and $\|q\|_2^2=\mass$ from \eqref{eq:radial_scale_mass} and \eqref{eq:conservation}. Fatou's lemma and \eqref{eq:radial_decay_double} now imply 
    \[
    0\leq \iint_{0<r<s<\infty}  \left(\frac{1}{r}-\frac{1}{s}\right)|v(t,r)|^2|v(t,s)|^2 drds\leq \liminf_{n\to\infty }\frac{2\kfun(1)}{|\lambda|t_n}=0.
    \]
    This contradicts the positivity in Lemma \ref{lem:radial_energy_difference} and hence, $\mathcal{J}(t)\to \infty$. Combined with the fact that $\mathcal{J}(t)$ is also convex, we have 
    \[
    \mathcal{J}(t)\geq c_0t,\quad \text{for some }c_0>0, \,\,\text{ and }t\geq T_0.
    \]

    \medskip

    Finally, we use $\mathcal{J}(t)$ to prove \eqref{eq:radial_recovered_coulomb}. This moment bound $\mathcal{J}(t)$ give a fixed amount of mass away from zero. For $t\geq T_0$, splitting the integral at $a_0=c_0/(2\mass)$ and using the moment estimate \eqref{eq:radial_moment_profile} with scaling \eqref{eq:radial_profile_scale} give
    \begin{equation}\label{eq:radial_fixed_amount}
        c_0\leq \frac{\mathcal{J}(t)}{t}=\int_0^\infty a|\g(t,a)|^2 da\leq \frac{c_0}{2}+N_V\left(\int_{a_0}^\infty |\g(t,a)|^2da \right)^\frac{1}{2},
    \end{equation}
    which implies
    \begin{equation}
        \int_{a_0}^\infty |\g(t,a)|^2da\geq \frac{c_0^2}{4N_V^2}. 
    \end{equation}
    We now use \eqref{eq:K_radial_rewrite} and \eqref{eq:radial_k_bound} on the double integral and the lower bound \eqref{eq:radial_fixed_amount} to control $\int a^{-1}|\g(t,a)|^2da$ near the origin. With
    \[
    \frac{1}{a}-\frac{1}{b}> \frac{1}{2a}, \qquad \text{for }a< \frac{a_0}{2}, \,\, b>a_0,
    \]
    we have
    \begin{equation*}
        \begin{split}
            \frac{2\kfun(1)}{|\lambda|}&\geq \iint_{0<a<b<\infty} \left(\frac{1}{a}-\frac{1}{b}\right)|\g(t,a)|^2|\g(t,b)|^2 dadb\\
            &\geq \int_0^\frac{a_0}{2}\int_{a_0}^\infty \frac{|\g(t,a)|^2}{2a}|\g(t,b)|^2 db\,da \geq \frac{c_0^2}{8N_V^2}\int_0^{\frac{a_0}{2}} \frac{|\g(t,a)|^2}{a}da.
        \end{split}
    \end{equation*}
    When $a\geq a_0/2$, we bound the weighted integral easily and have
    \begin{equation}\label{eq:radial_small_time_coulomb}
        \int_0^\infty \frac{|\g(t,a)|^2}{a}da\leq \frac{16\kfun(1)N_V^2}{|\lambda|c_0^2} +\frac{4\mass}{c_0}\mass,\quad \text{for }t\geq T_0. 
    \end{equation}
    For $1\leq t\leq T_0$, Hardy's inequality \eqref{eq:Hardy_half_line} and \eqref{eq:radial_k_bound} yield
    \begin{equation*}
        \int_0^\infty \frac{|\g(t,a)|^2}{a}da\leq 2\sqrt{\mass}\|\g'(t)\| _2\leq 2\sqrt{\mass \kfun(1)T_0}.
    \end{equation*}
    Together with \eqref{eq:radial_small_time_coulomb}, we prove \eqref{eq:radial_recovered_coulomb}. 
\end{proof}

In summary, together with \eqref{eq:radial_k_bound}, for every radial datum with $|\lambda|\mass\leq 2\kappa$, \eqref{eq:radial_k_bound}, \eqref{eq:half_line_coulomb_est} and Lemma \ref{lem:recover_endpt} give
\begin{equation}\label{eq:radial_uniform_bound}
    \sup_{t\geq 1}\frac{\|\g'(t)\|_2^2}{t}\leq L_V, \quad \sup_{t\geq 1}\int_0^\infty \frac{|\g(t,a)|^2}{a}da\leq L_V,\quad \sup_{t\geq 1}\mathfrak{D}[\g(t)]\leq \mass L_V,
\end{equation}
where $L_V$ is a constant depending on $\lambda$, $\kappa$, and initial datum $u_0$. 

%%%%%%%%%%%%%%%%%%%%%%%%%%%%%%%%%%%%%%%
%%%%%%%%%%%%%%%%%%%%%%%%%%%%%%%%%%%%%%%
%%%%%%%%%%%%%%%%%%%%%%%%%%%%%%%%%%%%%%%

\subsection{Radial modified scattering at and below the mass threshold}\label{subsec:radial_modified_scattering}

We now use \eqref{eq:radial_uniform_bound} to prove the modified scattering \eqref{eq:modified_scattering_coulomb_radial} in Theorem \ref{thm:coulomb_focusing} under the assumption $|\lambda|\mass\leq 2\kappa$. By scaling \eqref{eq:radial_profile_scale} and Newton's formula, the hypotheses of Proposition \ref{prop:limiting_field} hold with $A\leq L_V$, and $B\leq \mass L_V$. Thus, there is a nonnegative limiting field $\Phi_\infty \in \dot H^1$. 

We will show that radial symmetry improves the field convergence to $L^\infty$ away from the origin. This makes the localized field error integrable in time, hence allowing a direct comparison with the frozen evolution without amplitude truncation. Moreover, this improvement gives a quantitative decay rate in \eqref{eq:modified_scattering_coulomb_radial} instead of a qualitative error in the nonradial problem.

\begin{lemma}\label{lem:radial_field_convergence}
    The field $\Phi_\infty$ has a radial representative, and is locally Lipschitz on the half line. Moreover, we have
    \begin{equation}\label{eq:radial_field_convergence}
        \big|\varphi_\g(t,a)-\Phi_\infty(a)\big|\leq 4\sqrt{\mass L_V}a^{-2} t^{-\frac{1}{2}}, \quad \text{for }a>0, \,\, t\geq 1.
    \end{equation}
    Additionally, we have
    \begin{equation}\label{eq:Phi_inf_bound}
        0\leq \Phi_\infty (a)\leq \frac{\mass}{a}, \quad \big|\Phi'_\infty(a)\big|\leq \frac{\mass}{a^2}\quad \text{for almost every }a>0.
    \end{equation}
\end{lemma}
\begin{proof}
    Since $\varphi_\g$ is real-valued, \eqref{eq:g_PDE_radial} gives the density-flux relation
    \begin{equation}\label{eq:radial_density_flux}
        \partial_t |\g(t,a)|^2 =2t^{-2}\partial_a \Im(\overline{\g(t,a)}\g'(t,a)).
    \end{equation}
    Splitting the integral of $\varphi_\g$ at $b=a$, we apply integration by parts and get\footnote{To avoid a boundary trace of $\g'$ at the origin, we write 
    \[
    \varphi_\g(t,a)= \frac{\mass}{a}+\int_a^\infty \left(\frac{1}{b}-\frac{1}{a}\right) |\g(t,b)|^2 db.
    \]
    Testing the density-flux relation with
    $(b^{-1}-a^{-1})\mathfrak{1}_{\{b>a\}}$, by approximation and a cutoff at infinity, is legitimate since $\overline\g\g'\in L^1(\dom)$. Thus $\varphi_\g(\cdot,a)$ is locally absolutely continuous.
    }
    \begin{equation*}
        \partial_t \varphi_\g (t,a)=2t^{-2}\int_a^\infty \frac{\Im (\overline{\g(t,b)}\g'(t,b))}{b^2}db.
    \end{equation*}
    \eqref{eq:radial_uniform_bound} bounds
    \begin{align*}
        \big|\partial_t \varphi_\g(t,a)\big|\leq 2t^{-2}a^{-2}\|\g(t)\|_2\|\g'(t)\|_2\leq 2\sqrt{\mass L_V}a^{-2}t^{-\frac{3}{2}}.
    \end{align*}
    Integrating this proves \eqref{eq:radial_field_convergence}. The limit agrees with $\Phi_\infty$ by the uniqueness of distributional limit in Proposition \ref{prop:limiting_field}, proving that $\Phi_\infty$ is radial. Finally, Newton's formula gives
    \begin{equation}
        0\leq \varphi_\g(t,a)\leq \frac{\mass}{a},\qquad \varphi_\g'(t,a)=-\frac{1}{a^2}\int_0^a |\g(t,b)|^2 db.
    \end{equation}
    The pointwise bound passes directly to $\Phi_\infty$. Moreover, integrating the derivative bound and sending time to infinity, we obtain
    \[
    \big|\Phi_\infty(b)-\Phi_\infty(a)\big|\leq \mass \int_a^b\frac{ds}{s^2}, \qquad 0<a<b.
    \]
    Hence $\Phi_\infty$ is locally Lipschitz, with \eqref{eq:Phi_inf_bound}.
\end{proof}

\medskip

We next use the frozen evolution from Subsection \ref{subsec:frozen_evolution} with the half-line formulation, choosing the frozen coefficient so that Proposition \ref{prop:focusing_cauchy_frozen} applies. Recall from \eqref{eq:h} that
\[
\non_\infty(a)=\frac{\kappa}{a}+ \lambda\Phi_\infty(a).
\]
For $0<\delta<1$, set $\g_\delta:=\chi_\delta \g$, using the radial cutoff from Subsection \ref{subsec:pre}. To control the singularity of $\Phi_\infty$ near the origin, we define
\begin{equation}
    \non(a):=\non_\infty(\max\{a,\delta\}).
\end{equation}
Thus, $\non$ is constant on $(0,\delta)$ and agrees with $\non_\infty$ on the support of $\g_\delta$. 

For $S\geq 1$, we consider the frozen evolution
\begin{equation}\label{eq:radial_frozen_PDE}
    \begin{cases}
        i\partial_t\g_\delta^S-t^{-2}\partial_{aa}\g_\delta^S=-t^{-1}\non \g_\delta^S, \quad &t\geq S, \,\,a>0,\\[6pt]
        \g_\delta^S(t,0)=0,\qquad \g_\delta^S(S)=\g_\delta(S). 
    \end{cases}
\end{equation}
Following \eqref{eq:h} and \eqref{eq:h_cutoff}, we define the corresponding corrected profiles by
\begin{equation}
    \h_\delta(t):=e^{-i\non_\infty\log(t)}\g_\delta(t)=e^{-i\non\log(t)}\g_\delta(t),\qquad \h_\delta^S(t)=e^{-i\non\log(t)}\g_\delta^S(t).
\end{equation}
This is the half-line formulation of the frozen equation \eqref{eq:frozen_PDE} with the coefficient as $\non(|\cdot |)$. Under the radial transformation \eqref{eq:g_radial}, we have
\begin{equation}\label{eq:g_cutoff_radial}
    \g_\delta^S(t,a)=\sqrt{4\pi} \, ag_\delta^S(t,a),
\end{equation}
where $g_\delta^S$ solves \eqref{eq:frozen_PDE} with coefficient $\non(|\cdot |)$. By \eqref{eq:Phi_inf_bound}, spherical coordinates give
\begin{equation}\label{eq:full_space_from_radial}
\begin{split}
    \big\|\non(|\cdot |)\big\|_{L^6(\R^3)}^6&=4\pi \int_0^\infty|\non(a)|^6a^2da\\
    &\leq 4\pi(\kappa+|\lambda|\mathbf M)^6 \left(\delta^{-6}\int_0^\delta a^2da +\int_\delta^\infty a^{-4}da\right)<\infty.
\end{split}
\end{equation}
Since $\non(|\cdot |)$ is real and belongs to $L^6(\R^3)$, Lemma \ref{lem:frozen_est}, \eqref{eq:g_cutoff_radial} and scaling \eqref{eq:radial_profile_scale} give a unique frozen solution $\g_\delta^S\in C([S,\infty); H_0^1(\dom))$ with
\begin{equation}\label{eq:radial_frozen_mass}
    \big\|\g_\delta^S(t)\big\|_2^2\leq \mass. 
\end{equation}
Moreover, \eqref{eq:g_cutoff_radial} and norm identities in \eqref{eq:radial_profile_scale} allow us to use Proposition \ref{prop:focusing_cauchy_frozen} in the half-line formulation. 

We first verify the hypotheses. By Lemma \ref{lem:radial_field_convergence}, we have
\begin{equation*}
    |\non(a)|\leq \frac{\kappa+|\lambda|\mass}{\max\{a,\delta\}},\qquad |\non'(a)|\leq \frac{\kappa+|\lambda|\mass}{a^2}\mathfrak{1}_{\{a>\delta\}}\quad \text{for almost every $a>0$}.
\end{equation*}
Thus, a similar computation as in \eqref{eq:full_space_from_radial} proves the radial coefficient $\non(|\cdot|)$ belongs to $L^6(\R^3)\cap \dot H^1(\R^3)$ with
\begin{equation*}
    \big\|\nabla\non (|\cdot|)\big\|_{L^\infty(\R^3)}\leq (\kappa+|\lambda|\mass)\delta^{-2},
\end{equation*}
and the hypothesis of Proposition \ref{prop:focusing_cauchy_frozen} holds with $a_*=0$ and $b_*=(\kappa+|\lambda|\mass) \delta^{-2}$. Therefore, the following estimates follow from Proposition \ref{prop:focusing_cauchy_frozen} and scaling. 
\begin{corollary}\label{cor:radial_cauchy_frozen}
    For $t\geq R\geq S$, we have
    \begin{equation}\label{eq:radial_gprime_frozen}
       \big\|(\g_\delta^S)'(t)\big\|_2\leq \big\|(\g_\delta^S)'(R)\big\|_2 +(\kappa+|\lambda|\mass)\sqrt{\mass} \delta^{-2}\log(t/R),
    \end{equation}
    and 
    \begin{equation}\label{eq:radial_cauchy_frozen}
       \sup_{t\geq R}\big\|\h_\delta^S(t) -\h_\delta^S(R)\big\|_2^2\lesssim R^{-1}\left(\big\|(\g_\delta^S)'(R)\big\|_2+(\kappa+|\lambda|\mass)\sqrt{\mass}\,\delta^{-2} \right)^2.
    \end{equation}
\end{corollary}
With \eqref{eq:radial_uniform_bound}, we further bound
\begin{equation}
    \big\|\g_\delta'(S)\|_2\leq C_\chi \|\g'(S)\|_2\leq C_\chi\sqrt{L_V S}. 
\end{equation}
\eqref{eq:radial_gprime_frozen} then implies
\begin{equation*}
    \big\|(\g_\delta^S)'(R)\big\|_2\leq C_\chi\sqrt{L_V S}+(\kappa+|\lambda|\mass)\sqrt{\mass} \delta^{-2}\log(R/S).
\end{equation*}
Plugging this in \eqref{eq:radial_cauchy_frozen}, we obtain
\begin{equation}\label{eq:radial_cauchy_frozen_final}
    \sup_{t\geq R}\big\|\h_\delta^S(t) -\h_\delta^S(R)\big\|_2\lesssim_\chi\sqrt{L_V\frac{S}{R}}+(\kappa+|\lambda|\mass)\sqrt{\mass}\delta^{-2}\frac{1}{\sqrt{R}}\left(1+\log\frac{R}{S}\right).
\end{equation}

\medskip

We are now ready to prove the modified scattering \eqref{eq:modified_scattering_coulomb_radial}. 
\begin{proof}[Proof of Theorem \ref{thm:coulomb_focusing} for $|\lambda|\mass\leq 2\kappa$] We first compare the localized profile with the frozen evolution. Fix $0<\delta<1$. Since $\non=\non_\infty$ on the support of $\g_\delta$, \eqref{eq:g_PDE_radial} gives 
\begin{equation}\label{eq:radial_g_PDE_cutoff}
    i\partial_t\g_\delta-t^{-2}\partial_{aa}\g_\delta= -t^{-1} \non\g_\delta -\lambda t^{-1}(\varphi_\g-\Phi_\infty)\g_\delta- t^{-2}\big(2\chi_\delta' \g' +\chi_\delta'' \g\big).
\end{equation}
The last two terms measure the field error and the error from commuting the spatial with the Laplacian. Hardy's inequality \eqref{eq:Hardy_half_line} gives
\begin{align*}
    \big\|\chi_\delta' \g'(t) +\chi_\delta'' \g(t)\big\|_2\lesssim_\chi \delta^{-1}\left(\|\g'(t)\|_2+\|a^{-1}\g(t)\|_2\right)\lesssim_\chi \delta^{-1}\|\g'(t)\|_2.
\end{align*}
Together with \eqref{eq:radial_field_convergence}
and \eqref{eq:radial_uniform_bound}, this yields
\begin{align*}
    \big\|t^{-1}\lambda (\varphi_\g-\Phi_\infty)\g_\delta+t^{-2}\big(2\chi_\delta' \g' +\chi_\delta'' \g\big)\big\|_2\lesssim_\chi t^{-\frac{3}{2}}\sqrt{L_V}\delta^{-1}\left(1+|\lambda|\mass\delta^{-1} \right),
\end{align*}
which is integrable in time. Comparing \eqref{eq:radial_frozen_PDE} and \eqref{eq:radial_g_PDE_cutoff}, using their common initial data at $S$ and the unitarity of the frozen propagator, we bound
\begin{equation}\label{eq:radial_frozen_comparison}
    \sup_{t\geq S}\big\|\h_\delta^S(t)-\h_\delta(t)\big\|_2=\sup_{t\geq S}\big\|\g_\delta^S(t)-\g_\delta(t)\big\|_2\lesssim_\chi \sqrt{L_V}\delta^{-1}\left(1+|\lambda|\mass\delta^{-1} \right)S^{-\frac{1}{2}}.
\end{equation}

\medskip

We next remove the spatial cutoff. Define
\begin{equation}\label{eq:h_radial}
    \h(t,a)=e^{-i\non_\infty(a)\log t}\g(t,a),
\end{equation}
then $\h_\delta=\chi_\delta\h$. Since $1-\chi_\delta$ vanishes for $a\geq 2\delta$, \eqref{eq:radial_uniform_bound} gives
\begin{equation}\label{eq:radial_remove_cutoff}
    \|\h(t)-\h_\delta(t)\|_2^2\leq \int_0^{2\delta} |\g(t,a)|^2 da \leq 2\delta \int_0^\infty \frac{|\g(t,a)|^2}{a} da \leq 2\delta L_V.
\end{equation}

For $t\geq s\geq R\geq S$, the triangle inequality, \eqref{eq:radial_cauchy_frozen_final}, \eqref{eq:radial_frozen_comparison} and \eqref{eq:radial_remove_cutoff} therefore give
\begin{equation}\label{eq:radial_h_final_est}
    \begin{split}
        \big\|\h(t)-\h(s)\big\|_2&\lesssim_\chi \sqrt{L_V}\left(\delta^{\frac{1}{2}}+(|\lambda|\mass\delta^{-2}+\delta^{-1})S^{-\frac{1}{2}}+\sqrt{\frac{S}{R}}\right)\\
        &\qquad+(\kappa+|\lambda|\mass) \sqrt{\mass}\delta^{-2}R^{-\frac{1}{2}}\left(1+\log\frac{R}{S}\right).
    \end{split}
\end{equation}
Here, the implicit constant and $L_V$ are uniform in $\delta,S,R,t,s$. For $R\geq 2$, we choose 
\[
\delta=R^{-\frac{1}{6}}, \qquad S=R^\frac{5}{6}.
\]
The first three contributions in \eqref{eq:radial_h_final_est} are bounded by $R^{-\frac{1}{12}}$, up to multiplication by a constant, while 
\[
    \delta^{-2}R^{-\frac{1}{2}}\left(1+\log\frac{R}{S}\right)\lesssim R^{-\frac{1}{6}}(1+\log R) \lesssim R^{-\frac{1}{12}}.
\]
Hence $\h(t)$ is Cauchy in $L^2(\dom)$. Denoting its limit by $\h_\infty$, we obtain
\begin{equation}\label{eq:radial_h_convergence_rate}
    \|\h(t)-\h_\infty\|_2\lesssim_\chi \left[(1+|\lambda|\mass)\sqrt{L_V}+(\kappa+|\lambda|\mass)\sqrt{\mass}\right]t^{-\frac{1}{12}}, \qquad t\geq2.
\end{equation}
Since $\|\h(t)\|_2^2=\|\g(t)\|_2^2=\mass$, strong convergence also gives $\|\h_\infty\|_2^2=\mass$. 

\medskip

We next identify the limiting field using Newton's formula. For every fixed $a>0$, we have
\begin{align*}
    \left|\varphi_\g(t,a) -\int_0^\infty \frac{|\h_\infty(b)|^2}{\max\{a,b\}} db\right| \leq \frac{1}{a}\big\| |\h(t)|^2-|\h_\infty|^2\big\|_1\leq \frac{2\sqrt{\mass}}{a}\|\h(t)-\h_\infty\|_2\,\longrightarrow 0.
\end{align*}
Together with Lemma \ref{lem:radial_field_convergence}, we prove
\begin{equation*}
    \Phi_\infty(a)=\int_0^\infty \frac{|\h_\infty(b)|^2}{\max\{a,b\}} db, \quad a>0.
\end{equation*}

Finally, define the full-space profile by
\begin{equation}
    U_\infty(y):=\frac{\h_\infty(|y|)}{\sqrt{4\pi}|y|}, \quad \text{for }y\neq 0.
\end{equation}
Then $U_\infty$ is radial with $\|U_\infty\|_{L^2(\R^3)}^2=\|\h_\infty\|_2^2=\mass$ and Newton's formula gives $\phi_{U_\infty}=\Phi_\infty \in \dot H^1(\R^3)$. Using \eqref{eq:v} and \eqref{eq:g_radial}, we obtain
\begin{equation}
    \begin{split}
        \left\|u(t,x)-t^{-3/2}e^{-i|x|^2/(4t)}e^{i(\kappa t/|x|+ \lambda\Phi_\infty(x/t))\log t}U_\infty(x/t)\right\|_{L^2(\R^3)}&=\|\h(t)-\h_\infty\|_2.
    \end{split}
\end{equation}
Estimate \eqref{eq:radial_h_convergence_rate} proves \eqref{eq:modified_scattering_coulomb_radial}.  
\end{proof}

%%%%%%%%%%%%%%%%%%%%%%%%%%%%%%%%%%%%%%%
%%%%%%%%%%%%%%%%%%%%%%%%%%%%%%%%%%%%%%%
%%%%%%%%%%%%%%%%%%%%%%%%%%%%%%%%%%%%%%%

\subsection{Standing waves above the threshold}\label{subsec:standing_wave}
We now assume $|\lambda|\mass>2\kappa$ and construct a standing wave at each such mass. 

For $f\in H_0^1(\dom)$, we write the energy in \eqref{eq:radial_energy} as a functional of $f$ and set
\begin{equation}\label{eq:energy_functional}
    \mathbf{E}[f]:=\frac{1}{2}\|f'\|_2^2+\frac{\kappa}{2}\int_0^\infty \frac{|f(r)|^2}{r}dr+\frac{\lambda}{4}\mathfrak{D}[f]. 
\end{equation}
We also define
\begin{equation}
    \mathfrak{s}(\mass):=\inf\left\{\mathbf{E}[f] \,\big|\,f\in H_0^1(\R_+), \,\,\|f\|_2^2=\mass \right\}.
\end{equation}
For $\|f\|_2^2=\mass$, Lemma \ref{lem:radial_energy_difference} gives the exact decomposition
\begin{equation}\label{eq:energy_rewrite}
    \begin{split}
        \mathbf E[f]&=\frac12 \|f'\|_2^2+ \frac{2\kappa-|\lambda|\mass}{4} \int_0^\infty\frac{|f(r)|^2}{r}dr\\ &\quad +\frac{|\lambda|}{4} \iint_{0<r<s<\infty}\left(\frac{1}{r}-\frac{1}{s}\right) |f(r)|^2 |f(s)|^2 drds.
    \end{split}
\end{equation}
Thus the energy is positive for nonzero functions when $|\lambda|\mass\leq 2\kappa$. Above the threshold, the following proposition produces a minimizer with negative energy.
\begin{proposition}\label{prop:stationary_solution}
    For every $\mass$ with $|\lambda|\mass>2\kappa$, $\mathfrak{s}(\mass)<0$ is attained by a nonnegative $q\in H_0^1(\dom)$ with $\|q\|_2^2=\mass$. Moreover, there exists $\omega>0$ such that
    \begin{equation}\label{eq:stationary_equation}
        -q''+\frac{\kappa}{r}q+\lambda\varphi_q q+\omega q =0
    \end{equation}
    in distributional sense on $\R_+$. Additionally, we have $rq\in L^2(\dom)$.
\end{proposition}
\begin{proof}
    We first show $\mathfrak{s}<0$. Choose $\eta>0$ sufficiently small with $|\lambda|\mass>2\kappa(1+\eta)$ and take $f\in C_c^\infty(1,1+\eta)$ with $\|f\|_2^2=\mass$. Then
    \[
    \int_0^\infty\frac{|f(r)|^2}{r}dr\leq \mass, \qquad \mathfrak D[f]=\int_0^\infty\int_0^\infty \frac{|f(r)|^2|f(s)|^2}{\max\{r,s\}}drds\geq\frac{\mass^2}{1+\eta}.
    \]
    This way, the sum of the two potential energy terms in \eqref{eq:energy_functional} is negative. The mass-preserving dilation $f_\nu(r):=\nu^{1/2}f(\nu r)$ satisfies  
    \begin{equation}\label{eq:energy_f_nu}
        \mathbf E[f_\nu]=\frac{\nu^2}{2} \|f'\|_2^2 + \nu\left(\frac{\kappa}{2} \int_0^\infty\frac{|f(r)|^2}{r}dr +\frac{\lambda}{4} \mathfrak D[f]\right).
    \end{equation}
    Hence, choosing $\nu>0$ sufficiently small makes $\mathbf{E}[f_\nu]$ negative. On the other hand, Hardy's inequality and \eqref{eq:half_line_coulomb_est} give, 
    \begin{equation}\label{eq:energy_lower_bound}
        \mathbf{E}[f]\geq \frac{1}{2}\|f'\|_2^2-\frac{|\lambda|}{2}\mass^\frac{3}{2}\|f'\|_2\geq \frac{1}{2}\left(\|f'\|_2-\frac{|\lambda|\mass^\frac{3}{2}}{2}\right)^2-\frac{|\lambda|^2\mass^3}{8}\geq -\frac{|\lambda|^2\mass^3}{8}.
    \end{equation}
    Thus $-\infty<\mathfrak{s}(\mass)<0$. Let $(f_n)$ be a minimizing sequence and \eqref{eq:energy_lower_bound} guarantees that $(f_n)$ is uniformly bounded in $H_0^1(\dom)$. By Banach–Alaoglu theorem, Rellich's theorem and uniqueness of limit, up to passing to subsequences, there exists $q\in H_0^1(\dom)$ such that
    \begin{equation}\label{eq:minimizer_limit}
        f_n \rightharpoonup q \quad \text{in }H_0^1(\dom), \qquad f_n\longrightarrow q \quad \text{in } L^2(0,R) \text{ for every }R<\infty,
    \end{equation}
    with $\|q\|_2^2\leq \mass$. 

    \medskip

    We next show that $q$ attains $\mathfrak{s}(\mass)$. Write $L:=\sup_n \|f_n'\|_2$, then weak convergence gives $\|q'\|_2\leq L$. For $r>0$, H\"older inequality gives
    \begin{equation*}
        |f_n(r)|^2=\left|\int_0^r f_n'(s)ds\right|^2\leq r\int _0^r |f_n'(s)|^2 ds\leq L^2r,\qquad |q(r)|^2\leq L^2r.
    \end{equation*}
    Consequently, by splitting the integral into three parts at $\delta$ and $\delta^{-1}$, we have
    \begin{equation*}
        \begin{split}
            \int_0^\infty \frac{\left| |f_n(r)|^2-|q(r)|^2\right|}{r}dr&\leq \int_0^\delta 2L^2 dr+\frac{2\sqrt{\mass}}{\delta}\left(\int_\delta^{\delta^{-1}}|f_n(r)-q(r)|^2dr\right)^\frac{1}{2}+ 2\mass \delta\\
            &\leq 2L^2\delta+\frac{2\sqrt{\mass}}{\delta}\|f_n-q\|_{L^2(0,\delta^{-1})}+2\mass \delta. 
        \end{split}
    \end{equation*}
    By \eqref{eq:minimizer_limit}, taking $n\to\infty$ and then $\delta\downarrow0$, we obtain 
    \begin{equation}\label{eq:minimizer_coulomb_energy}
        \int_0^\infty \frac{\left| |f_n(r)|^2-|q(r)|^2\right|}{r}dr\longrightarrow 0, \quad \text{as }n\to \infty. 
    \end{equation}
    By symmetry of the Newton kernel and $0\leq\varphi_{f_n}(r),\varphi_q(r)\leq\mass/r$,
    \begin{align*}
        \big|\mathfrak D[f_n]-\mathfrak D[q]\big|&= \left|\int_0^\infty(\varphi_{f_n}(r)+\varphi_q(r))(|f_n(r)|^2-|q(r)|^2) dr\right|\\
        &\quad \leq 2\mass \int_0^\infty\frac{\big||f_n(r)|^2-|q(r)|^2\big|}{r}dr
    \end{align*}
    which also converges to 0 by \eqref{eq:minimizer_coulomb_energy}. Together with $\|q'\|_2\leq\liminf_{n\to\infty} \|f_n'\|_2$, we have
    \begin{equation}
        \begin{split}
            \mathbf{E}[q]&\leq \frac{1}{2}\liminf_{n\to\infty} \|f_n'\|_2^2+\lim_{n\to\infty}\left(\frac{\kappa}{2}\int_0^\infty \frac{|f_n(r)|^2}{r}dr+\frac{\lambda}{4}\mathfrak{D}[f_n]\right)\\
            &\leq \lim_{n\to\infty}\mathbf{E}[f_n]=\mathfrak{s}(\mass).
        \end{split}
    \end{equation}
    Hence $q\neq 0$. If $\|q\|_2^2<\mass$, set $\alpha:=\sqrt\mass/\|q\|_2>1$. Then $\|\alpha q\|_2^2=\mass$, but
    \begin{equation}\label{eq:rescale_q}
        \mathbf E[\alpha q]=\alpha^2\mathbf{E}[q]-\frac{|\lambda|}{4}(\alpha^4-\alpha^2)\mathfrak{D}[q]<\mathbf{E}[q]\leq \mathfrak{s}(\mass),
    \end{equation}
    since $\mathbf{E}(q)<0$. Hence \eqref{eq:rescale_q} contradicts the definition of $\mathfrak{s}(\mass)$. Thus $\|q\|_2^2=\mass$ and hence $\mathbf{E}[q]=\mathfrak{s}(\mass)$.

    \medskip

    Replacing $q$ by $|q|$ preserves the mass and potential energies and does not increase the kinetic energy, so we take $q\geq 0$. A standard variational argument under the mass constraint gives \eqref{eq:stationary_equation} for some $\omega\in\R$. More specifically, symmetry of the Coulomb kernel gives
    \[
     \partial_\varepsilon\,\mathfrak{D}[q+\varepsilon \zeta]\,\bigg|_{\varepsilon=0} =4\Re \int_0^\infty \varphi_q q\,\bar \zeta dr,\qquad \zeta\in H_0^1(\dom),
    \]
    and Hardy's inequality ensures that all terms in the first variation are well-defined. 

    To prove $\omega>0$, we apply the mass-preserving dilation in \eqref{eq:energy_f_nu} with $f=q$. Since $q$ minimizes the energy at mass $\mass$, the energy is minimized at $\nu=1$. Differentiating \eqref{eq:energy_f_nu} at $\nu=1$ gives
    \[
    \int_0^\infty|q'(r)|^2 dr +\frac{\kappa}{2} \int_0^\infty\frac{|q(r)|^2}{r} dr +\frac{\lambda}{4} \int_0^\infty\int_0^\infty \frac{|q(r)|^2|q(s)|^2}{\max\{r,s\}} drds=0.
    \]
    On the other hand, testing \eqref{eq:stationary_equation} against $\bar q$ gives
    \[
    \int_0^\infty|q'(r)|^2 dr +\kappa \int_0^\infty\frac{|q(r)|^2}{r} dr+ \omega\int_0^\infty|q(r)|^2 dr+ \lambda \int_0^\infty\int_0^\infty \frac{|q(r)|^2|q(s)|^2}{\max\{r,s\}} drds=0.
    \]
    Hence, 
    \begin{equation}
        \omega\mass=3\|q'\|_2^2 +\kappa\int_0^\infty\frac{|q(r)|^2}{r} dr>0.
    \end{equation}

    \medskip

    Finally, this positive frequency gives exponential decay of $q$. First, Sobolev embedding and $q\in H_0^1(\dom)$ gives a continuous representative of $q$. The coefficient in \eqref{eq:stationary_equation} is continuous on $\R_+$, so this identifies the distributional derivative $q''$ with a continuous function. Hence $q\in C^2(\dom)$. With $\kappa>0$ and \eqref{eq:half_line_coulomb_est}, for sufficiently large $R_0$, like $R_0=\max\{1,\frac{3|\lambda|\mass}{\omega}\}$, we have
    \begin{equation*}
        \omega+\frac\kappa r-|\lambda|\varphi_q(r)\geq \omega-\frac{|\lambda|\mass}{r}> \frac{\omega}{2}, \quad \text{for }r\geq R_0.
    \end{equation*}
    Stationary equation \eqref{eq:stationary_equation}  and $q\geq 0$ therefore prove $q''(r)\geq \frac{\omega}{2}q(r)$ for large $r$. Also, $q(r)\to 0$ as $r\to\infty$ from $q\in H^1(\dom)$. 

    We compare $q$ with $q(R_0)e^{-\sqrt{\omega/2}(r-R_0)}$. The difference satisfies
    \begin{equation}\label{eq:gamma_difference_est}
        \gamma''(r)-\frac{\omega}{2}\gamma (r) =q''(r)- \frac{\omega}{2}q(r)\geq 0,\quad \gamma(r):=q(r)-q(R_0)e^{-\sqrt{\omega/2}(r-R_0)}.
    \end{equation}
    Also $\gamma(R_0)=0$ and converges to $0$ as $r\to\infty$. If $\gamma$ is positive somewhere, the boundary conditions would imply that it attains a positive interior maximum at some point, which contradicts \eqref{eq:gamma_difference_est}. Hence, 
    \begin{equation*}
        0\leq q(r)\leq q(R_0)e^{-\sqrt{\omega/2}(r-R_0)},\qquad r\geq R_0,
    \end{equation*}
    and in particular, we have $rq\in L^2(\dom)$.
\end{proof}

We now use this proposition to construct standing waves which keeps its physical mass localized. In the rescaled variable $y=x/t$, this mass concentrates at the origin, preventing strong convergence of any phase-corrected profile.
\begin{proof}[Proof of Theorem \ref{thm:coulomb_focusing} for $|\lambda|\mass>2\kappa$] 
    Assume $\mass>2\kappa/|\lambda|$, and let $q$ and $\omega$ be as in Proposition \ref{prop:stationary_solution}. Define
    \begin{equation}\label{eq:Q}
        Q(x):=\frac{q(|x|)}{\sqrt{4\pi}\,|x|}.
    \end{equation}
    Scaling \eqref{eq:radial_profile_scale} give
    \begin{equation*}
        \|Q\|_{L^2(\R^3)}^2=\mass, \qquad \|\nabla Q\|_{L^2(\R^3)}=\|q'\|_2, \qquad \|xQ\|_{L^2(\R^3)}=\|rq\|_2.
    \end{equation*}
    Hence $Q\in \Sigma$. By \eqref{eq:Q} and Newton's formula, \eqref{eq:stationary_equation} gives
    \begin{equation*}
        -\Delta Q+\frac{\kappa}{|x|}Q+\lambda\phi_Q Q+\omega Q=0
    \end{equation*}
    in the distributional sense on $\R^3$. Since $\phi_{e^{-i\omega t}Q}=\phi_Q$, the standing wave $u(t,x):=e^{-i\omega t}Q(x)$ solves \eqref{main} with potential $V(x)=\kappa/|x|$ and initial data $u(0,x)=Q(x)$. 

    Since $|u(t,x)|=|Q(x)|$, these standing waves retain a positive amount of mass in a fixed ball for all time. This is incompatible with the modified scattering asymptotics, which imply that the mass in every fixed ball tends to zero as $t\to\infty$. Thus, we finish showing that for every $|\lambda|\mass>2\kappa $, there exists some radial data in $\Sigma$ with mass $\mass$ whose solutions do not exhibit modified scattering. This proves that the mass threshold is sharp and finishes the proof of Theorem \ref{thm:coulomb_focusing}.
\end{proof}

%%%%%%%%%%%%%%%%%%%%%%%%%%%%%%%%%%%%%%%
%%%%%%%%%%%%%%%%%%%%%%%%%%%%%%%%%%%%%%%
%%%%%%%%%%%%%%%%%%%%%%%%%%%%%%%%%%%%%%%
%%%%%%%%%%%%%%%%%%%%%%%%%%%%%%%%%%%%%%%
%%%%%%%%%%%%%%%%%%%%%%%%%%%%%%%%%%%%%%%
%%%%%%%%%%%%%%%%%%%%%%%%%%%%%%%%%%%%%%%

\section{The focusing problem for general external potentials}\label{sec:general} 
We prove Theorem \ref{thm:square_focusing} in this section. Throughout this section, we consider $\lambda<0$, and $V$ satisfies assumptions \eqref{eq:A1}--\eqref{eq:A3} and \eqref{eq:F1} with $0\leq b<1/2$. Proposition \ref{prop:gwp} gives the global solution $u\in C(\R;\Sigma)$, and we use the rescaled profile $g$ from \eqref{eq:g}. All spaces and norms will be on the full space $\R^3$. 

Focusing interaction breaks the positivity of the interaction term in the rescaled energy \eqref{eq:K}. However, smallness of the data gives a uniform $L^3$ bound for $g$ while allowing $\ifun$ from \eqref{eq:I} to grow slowly. These bounds improve the convergence of the Hartree field in Proposition \ref{prop:limiting_field} and verify the hypotheses of the common frozen evolution estimates in Proposition \ref{prop:focusing_cauchy_frozen}. We then compare the original and the frozen evolutions through quadratic forms, without spatial or amplitude truncations, to obtain modified scattering. 

\subsection{Small-data bounds for the profile}\label{subsec:boot}
The positive form $\ifun$ controls the gradient, but its growth estimate \eqref{eq:I_bound} depends on $\|g(t)\|_3$. The density estimate in Lemma \ref{lem:density_diff} supplies this missing control. 
\begin{lemma}\label{lem:boot}
    Fix $b/2<\alpha<\frac{1}{4}$. There exists $\varepsilon_0>0$, depending only on $V$, $\alpha$, and $\lambda$ such that, if \eqref{eq:initial_cond} holds, for every $t\geq 1$, we have
    \begin{equation}\label{eq:small_profile}
        \|g(t)\|_3\lesssim_{V,\lambda,\alpha} \varepsilon_0, \qquad \ifun(t)\lesssim_{V,\lambda,\alpha} \varepsilon_0^2 t^{2\alpha}.
    \end{equation}
\end{lemma}
\begin{proof}
    We impose the following bootstrap on interval $[1,T]$
    \begin{equation}\label{boot}
        \ifun(t)\leq \varepsilon^2 t^{2\alpha}.
    \end{equation}
    Since $\|\nabla g(t)\|_2^2\leq \ifun(t)$ by \eqref{eq:I}, Lemma \ref{lem:density_diff} gives
    \begin{equation}\label{eq:increment}
        \big\| |g(t)|^2-|g(s)|^2\big\|_\frac{3}{2}\lesssim \varepsilon^2 s^{2\alpha-\frac{1}{2}}, \qquad \text{for }1\leq s\leq t\leq \min\{2s,T\}.
    \end{equation}
    At $t=1$, \eqref{eq:small_initial} bounds $\|\nabla g(1)\|_2^2\leq \ifun(1)\lesssim_{V,\lambda}\varepsilon_0^2$. Interpolation and Sobolev embedding give
    \[
    \|g(1)\|_3^2\lesssim \|g(1)\|_2\|\nabla g(1)\|_2\lesssim \varepsilon_0^2. 
    \]
    We decompose the interval $[1,t]$ dyadically and bound using \eqref{eq:increment} and $N=\lfloor\log_2 t\rfloor$ 
    \begin{equation}\label{eq:improved_L3_g}
        \begin{split}
            \|g(t)\|_3^2&\leq \|g(1)\|_3^2+\sum_{j=0}^{N-1} \big\| |g(2^{j+1})|^2-|g(2^j)|^2\big\|_\frac{3}{2}+\big\| |g(t)|^2-|g(2^N)|^2\big\|_\frac{3}{2}\\
            &\leq \|g(1)\|_3^2+\sum_{j=0}^\infty \varepsilon^2 2^{j(2\alpha-\frac{1}{2})}\lesssim_{V,\lambda,\alpha}\varepsilon_0^2+\varepsilon^2.
        \end{split}
    \end{equation}
    Here the summation converges because $\alpha<1/4$. Substituting \eqref{eq:improved_L3_g} into \eqref{eq:I_bound} and applying Gronwall's inequality give
    \begin{equation}\label{eq:improved_boot}
        \ifun(t)\leq \ifun(1) t^{b+C(V,\lambda,\alpha)|\lambda|(\varepsilon^2+\varepsilon_0^2)}\leq C(V,\lambda)\varepsilon_0^2t^{b+C(V,\lambda,\alpha)|\lambda|(\varepsilon^2+\varepsilon_0^2)}.
    \end{equation}
    Therefore, \eqref{eq:improved_boot} improves \eqref{boot} for $\varepsilon_0\ll \varepsilon\ll 1$, where the dependence relies on $V$, $\lambda$ and $\alpha$. Continuity therefore proves \eqref{eq:small_profile} for $t\geq 1$. 
\end{proof}

%%%%%%%%%%%%%%%%%%%%%%%%%%%%%%%%%%%%%%%
%%%%%%%%%%%%%%%%%%%%%%%%%%%%%%%%%%%%%%%
%%%%%%%%%%%%%%%%%%%%%%%%%%%%%%%%%%%%%%%

\subsection{Improved convergence of the field}\label{subsec:small_field_convergence}
We now use the small-data bounds in Lemma \ref{lem:boot} to strength the limiting field's estimates from Proposition \ref{prop:limiting_field}. From now on, we work with initial data satisfying \eqref{eq:initial_cond} with $\varepsilon_0$ picked in Lemma \ref{lem:boot}. We also fix $\lambda$, $V$ and $\alpha$, and allow the implicit constants to depend on $\lambda$, $V$ and $\alpha$ without writing them explicitly. 

We first verify the hypotheses of Proposition \ref{prop:limiting_field} are met. The Hardy-Littlewood-Sobolev inequality and interpolation give
\[
\dob[g(t)]\lesssim\|g(t)\|_{\frac{12}{5}}^4 \leq\|g(t)\|_2^2\|g(t)\|_3^2 \lesssim \mass\varepsilon_0^2.
\]
Together with \eqref{eq:small_profile}, this proves
\begin{equation}
    A\lesssim \varepsilon_0^2,\qquad B\lesssim \mass\varepsilon_0^2,
\end{equation}
with $A,B$ defined in Proposition \ref{prop:limiting_field}. Hence, Proposition \ref{prop:limiting_field} gives the existence of the nonnegative limiting field $\Phi_\infty \in \dot H^1$. 
\begin{lemma}\label{lem:improved_limiting_field}
    For $t\geq 1$, set $\eta:=1/4-\alpha>0$, we have
    \begin{equation}\label{eq:improved_field_convergence}
        \big\|\phi_g(t)-\Phi_\infty\big\|_6+\big\|\nabla\phi_g(t)-\nabla\Phi_\infty\big\|_2\lesssim \sqrt{\mass}\varepsilon_0 t^{-\eta}.
    \end{equation}
    Moreover, 
    \begin{equation}\label{eq:limiting_field_nabla_3}
        \|\nabla \Phi_\infty\|_3\lesssim \varepsilon_0^2. 
    \end{equation}
\end{lemma}
\begin{proof}
    For $1\leq s\leq t\leq 2s$, mass conservation gives
    \begin{equation}\label{eq:increment_L1}
    \big\||g(t)|^2-|g(s)|^2\big\|_1\leq 2\mass.
    \end{equation}
    Applying Lemma \ref{lem:density_diff} with \eqref{eq:small_profile} and then interpolating with \eqref{eq:increment_L1}, we have
    \begin{equation}
         \big\||g(t)|^2-|g(s)|^2\big\|_{\frac{6}{5}}\leq \big\||g(t)|^2-|g(s)|^2\big\|_1^{\frac{1}{2}}\big\||g(t)|^2-|g(s)|^2\big\|_{\frac{3}{2}}^{\frac{1}{2}}\lesssim \sqrt{\mass} \varepsilon_0 s^{-\eta}.
    \end{equation}
    The Hardy-Littlewood-Sobolev inequality therefore yields
    \[
    \|\phi_g(t)-\phi_g(s)\|_6+ \big\|\nabla\phi_g(t)-\nabla\phi_g(s)\big\|_2 \lesssim_\alpha\sqrt{\mass}\varepsilon_0 s^{-\eta}.
    \]
    Summing over dyadic intervals starting at $t$ as in \eqref{eq:improved_L3_g} proves \eqref{eq:improved_field_convergence}. 

    Finally, Lemma \ref{lem:coulomb_est} and \eqref{eq:small_profile} give
    \[
    \|\nabla\phi_g(t)\|_3 \lesssim\|g(t)\|_3^2 \lesssim_\alpha \varepsilon_0^2.
    \]
    By weak compactness and the $L^2$ convergence in \eqref{eq:improved_field_convergence}, $\nabla\phi_g(t_j)\rightharpoonup\nabla\Phi_\infty$ in $L^3$ along a sequence $t_j\to\infty$. Weak lower semicontinuity proves \eqref{eq:limiting_field_nabla_3}.
\end{proof}

%%%%%%%%%%%%%%%%%%%%%%%%%%%%%%%%%%%%%%%
%%%%%%%%%%%%%%%%%%%%%%%%%%%%%%%%%%%%%%%
%%%%%%%%%%%%%%%%%%%%%%%%%%%%%%%%%%%%%%%

\subsection{Frozen comparison and modified scattering}\label{subsec:small_frozen_evolution}
As in Subsection \ref{subsec:frozen_evolution}, we set $\non_\infty:=\lambda\Phi_\infty$ and define $h(t)=e^{-i\non_\infty\log(t)}g(t)$. We will use the full profiles $g$ and $h$ without the spatial cutoff $\chi_\delta$. For $S\geq 1$, we define the frozen evolution
\begin{equation}\label{eq:focusing_frozen_PDE}
    i\partial_t g^S-t^{-2}\Delta g^S=-t^{-1}\non g^S,\quad g^S(S)=g(S),\quad \non:=\non_\infty.
\end{equation}
and set $h^S(t)=e^{-i\non\log(t)}g^S(t)$. This is the frozen equation \eqref{eq:frozen_PDE} with initial datum $g(S)$ in place of $g_\delta(S)$ and with coefficient $\non=\non_\infty$. The proofs of Lemma \ref{lem:frozen_est} and Proposition \ref{prop:focusing_cauchy_frozen} use only an $H^1$ initial datum and its mass bound, so they apply to the evolution \eqref{eq:focusing_frozen_PDE}. Since $\non$ is real and belongs to $\dot H^1$, Lemma \ref{lem:frozen_est} gives a unique solution $g^S\in C([S,\infty);H^1)$ with 
\begin{equation}\label{eq:focusing_frozen_mass}
    \big\|g^S(t)\big\|_2^2=\mass,\quad \text{for }t\geq S.
\end{equation}

To apply Proposition \ref{prop:focusing_cauchy_frozen}, we use H\"older inequality, Sobolev embedding and \eqref{eq:limiting_field_nabla_3} and bound, for every $f\in H^1$, 
\begin{equation}
    \big\|(\nabla \non) f\big\|_2\leq |\lambda|\|\nabla \Phi_\infty\|_3\|f\|_6\leq C|\lambda|\|\nabla \Phi_\infty\|_3\|\nabla f\|_2.
\end{equation}
Thus, \eqref{eq:focusing_cauchy_assumption} holds with $a_*=C|\lambda|\|\nabla \Phi_\infty\|_3$ and $b_*=0$. Decreasing $\varepsilon_0$, if needed, we may take $a_*\leq \alpha<1/4$. Proposition \ref{prop:focusing_cauchy_frozen} and \eqref{eq:small_profile} then yield
\begin{equation}\label{eq:focusing_nabla_g_frozen}
    \big\|\nabla g^S(t)\big\|_2\lesssim \left(\frac{t}{S}\right)^\alpha\left(\varepsilon_0 S^{\alpha}\right)=\varepsilon_0 t^\alpha.
\end{equation}
Applying this slow growth \eqref{eq:focusing_frozen_cauchy_est} gives
\begin{equation}\label{eq:focusing_improved_frozen_est}
    \sup_{t\geq S}\big\|h^S(t)-h(S)\big\|_2^2\lesssim \varepsilon_0^2 S^{2\alpha-1}.
\end{equation}

We are now ready to prove the last theorem.
\begin{proof}[Proof of Theorem \ref{thm:square_focusing}]
    Choose $\alpha$ with $b/2<\alpha<1/4$ depending only on $b$, and consider $\varepsilon_0$ small enough such that both Lemma \ref{lem:boot} and $a_*\leq \alpha$ hold. The smallness of $\varepsilon_0$ depends only on $V$, $b$ and $\lambda$. 
    
    Take initial datum $u_0$ satisfying \eqref{eq:initial_cond}. Since $g^S(S)=g(S)$, \eqref{eq:focusing_frozen_mass} gives
    \begin{equation}\label{eq:focusing_rewrite_comparison}
        \|g(t)-g^S(t)\|_2^2=2\mass-2\Re\la g(t), g^S(t)\ra. 
    \end{equation}
    We rewrite \eqref{eq:PDE_g_general} as
    \begin{equation}
        i\partial_tg -t^{-2}\Delta g=-t^{-1}\non g-\lambda t^{-1}(\phi_g-\Phi_\infty)g -V(ty)g.
    \end{equation}
    Together with \eqref{eq:focusing_frozen_PDE}, this gives\footnote{This identity holds almost everywhere in time, with $\partial_tg,\partial_tg^S\in H^{-1}$ paired against $g^S,g\in H^1$, respectively. }
    \begin{equation}\label{eq:focusing_pair}
        \partial_t\Re \la g(t), g^S(t)\ra= -\int V(ty)\Im\big(g(t,y)\overline{g^S(t,y)}\big)dy-\lambda t^{-1}\int(\phi_g(t)-\Phi_\infty)\Im\big(g(t)\overline{g^S(t)}\big)dy,
    \end{equation}
    where the common Laplacian and frozen-field terms cancel. 

    The field error is integrable. We apply H\"older inequality, \eqref{eq:small_profile}, \eqref{eq:improved_field_convergence} and \eqref{eq:focusing_frozen_mass} and obtain
    \begin{equation}\label{eq:focusing_field_diff_tail}
        \begin{split}
            \left|\lambda t^{-1}\int(\phi_g(t)-\Phi_\infty)\Im\big(g(t)\overline{g^S(t)}\big)dy\right|&\leq |\lambda|t^{-1} \|\phi_g(t)-\Phi_\infty\|_6\|g(t)\|_3\|g^S(t)\|_2\\
            &\lesssim \mass \varepsilon_0^2 t^{-1-\eta}.
        \end{split}
    \end{equation}
    For the external term, since $V\geq 0$, \eqref{eq:I} and H\"older inequality bound
    \begin{equation}\label{eq:focusing_potential_term_bound}
        \left|\int V(ty)\Im\big(g(t,y)\overline{g^S(t,y)}\big)dy\right|\leq t^{-2}\ifun(t)^\frac{1}{2}\left(\ifun[g^S(t)](t)\right)^\frac{1}{2}.
    \end{equation}
    Fix $t\geq S$ and apply \eqref{eq:ifun_fixed} to the fixed spatial function $f=g^S(t)$. Assumption \eqref{eq:A3}, Hardy’s inequality, \eqref{eq:focusing_nabla_g_frozen} and \eqref{eq:focusing_frozen_mass} then give
    \begin{equation}
        \begin{split}
            \ifun[g^S(t)](t)&\leq t^b\ifun[g^S(t)](1) = t^b\left(\|\nabla g^S(t)\|_2^2 +\int V(y)|g^S(t,y)|^2\,dy\right)\\
            &\leq t^b\left((1+8C_V)\|\nabla g^S(t)\|_2^2 +C_V \mass\right)\lesssim_V \varepsilon_0^2 t^{b+2\alpha}.
        \end{split}
    \end{equation}
    Together with \eqref{eq:focusing_potential_term_bound} and \eqref{eq:small_profile}, we have
    \begin{equation}\label{eq:focusing_potential_tail}
        \left|\int V(ty)\Im\big(g(t,y)\overline{g^S(t,y)}\big)dy\right|\lesssim \varepsilon_0^2 t^{-2+2\alpha+\frac{b}{2}}.
    \end{equation}
    Since $1-2\alpha-b/2>0$, this is also integrable. In particular, this argument allows for inverse-square singularities and does not require $V(ty)g(t)\in L^2$.

    Since $g^S(S)=g(S)$ and \eqref{eq:focusing_rewrite_comparison}, integrating \eqref{eq:focusing_pair} using \eqref{eq:focusing_field_diff_tail}, \eqref{eq:focusing_potential_tail} and $\mass\leq \varepsilon_0^2\leq 1$, we obtain
    \begin{equation}\label{eq:focusing_comparison_frozen_est}
        \sup_{t\geq S} \big\|g(t)-g^S(t)\big\|_2^2\lesssim \varepsilon_0^2 \left(S^{-\eta}+S^{-1+2\alpha+\frac{b}{2}} \right).
    \end{equation}

    \medskip

    We now transfer \eqref{eq:focusing_comparison_frozen_est} to $h$. Using $h^S(S)=h(S)$, Cauchy estimate \eqref{eq:focusing_improved_frozen_est} and the comparison estimate \eqref{eq:focusing_comparison_frozen_est} give
    \begin{equation}
        \sup_{t\geq S}\|h(t)-h(S)\|_2^2\leq 2\sup_{t\geq S} \big\|g(t)-g^S(t)\big\|_2^2+2\sup_{t\geq S}\big\|h^S(t)-h(S)\big\|_2^2\lesssim \varepsilon_0^2\left(S^{-\eta}+S^{-1+2\alpha+\frac{b}{2}} \right).
    \end{equation}
    Here $S^{2\alpha-1}\leq S^{-1+2\alpha+b/2}$ because $b\geq0$ and $S\geq1$. Now, we fix $\alpha$ at the midpoint of its admissible interval:
    \[
    \alpha=\frac{1}{2}\left(\frac {b}{2}+\frac{1}{4}\right)=\frac{1+2b}{8},
    \]
    so that 
    \[
    -1+2\alpha+\frac{b}{2}= b-\frac{3}{4}<-\eta=-\frac{1-2b}{8}<0
    \]
    since $b<1/2$. Consequently, $h(t)$ is Cauchy in $L^2$ and converges strongly to some $U_\infty$. Sending $t\to\infty$, we have
    \begin{equation} \label{eq:focusing_h_convergence}
        \big\|h(t)-U_\infty\big\|_{2}\lesssim_{V,\lambda} \varepsilon_0 t^{-\frac{1-2b}{16}},\qquad \text{for }t\geq 1.
    \end{equation}
    Mass conservation and $|h(t)|=|g(t)|$ then gives 
    \[
    \|U_\infty\|_2^2=\mass,\qquad  \big\||g(t)|^2-|U_\infty|^2\big\|_1 \leq 2\sqrt{\mass} \|h(t)-U_\infty\|_2 \longrightarrow 0.
    \]
    For every $\zeta\in C_c^\infty(\R^3)$, the function $|\cdot|^{-1}*|\zeta|$ is bounded. The density convergence therefore gives $\phi_g(t)\to\phi_{U_\infty}$ in distributions. Comparing with the limiting-field convergence, we conclude that
    \[
    \Phi_\infty=\phi_{U_\infty}\in\dot H^1.
    \]
    
    Finally, undoing the definitions of $g$ and $h$ yields
    \begin{equation}
        \left\|u(t,x)-t^{-\frac{3}{2}}e^{-i\frac{|x|^2}{4t}}
        e^{i\lambda\Phi_\infty(x/t)\log t}
        U_\infty\left(\frac{x}{t}\right)\right\|_{L^2_x}=\|h(t)-U_\infty\|_2\lesssim \varepsilon_0t^{-\frac{1-2b}{16}}.
    \end{equation}
    This completes the proof.
\end{proof}

%%%%%%%%%%%%%%%%%%%%%%%%%%%%%%%%%%%%%%%
%%%%%%%%%%%%%%%%%%%%%%%%%%%%%%%%%%%%%%%
%%%%%%%%%%%%%%%%%%%%%%%%%%%%%%%%%%%%%%%
%%%%%%%%%%%%%%%%%%%%%%%%%%%%%%%%%%%%%%%
%%%%%%%%%%%%%%%%%%%%%%%%%%%%%%%%%%%%%%%
%%%%%%%%%%%%%%%%%%%%%%%%%%%%%%%%%%%%%%%

\end{document}